\documentclass[10pt]{article}
\usepackage[T1]{fontenc}
\usepackage[utf8]{inputenc}
\usepackage[english]{babel}
\usepackage[autostyle, english=american]{csquotes}
\MakeOuterQuote{"}

\usepackage{calligra}
 
\usepackage{xcolor}
\definecolor{reference}{rgb}{0.6,0.2,0.8}
\definecolor{citation}{rgb}{0.6,0.2,0.8}
\usepackage{url}
\usepackage{breakurl}
\usepackage{epigraph}
\usepackage[breaklinks,colorlinks,linktocpage,urlcolor=reference,linkcolor=reference,citecolor=citation]{hyperref}

\usepackage{tikz-cd}
\usetikzlibrary{babel}
\usepackage{verbatim}
\usepackage{mathrsfs}
\usepackage{mathtools}
\usepackage{dsfont}
\usepackage{bbold}
\usepackage{latexsym}
\usepackage{graphicx}
\usepackage{amscd,amssymb,amsmath,amsbsy,amsfonts,amsthm}
\usepackage[mathscr]{eucal}
\usepackage[all, 2cell]{xy}
\UseTwocells
\xyoption{2cell}{\UseTwocells}

\usepackage{enumerate}
\usepackage{tikz}
\usepackage[left=2.7cm, top=2.7cm, right=2.7cm, bottom=2.7cm]{geometry}
\usepackage{enumitem} 

\usepackage[backend=bibtex,style=alphabetic,sorting=nty,firstinits=true,maxbibnames=99,maxalphanames=99]{biblatex}

\usepackage[capitalise]{cleveref}
\crefformat{equation}{(#2#1#3)}

\DeclareFontFamily{OT1}{pzc}{}
\DeclareFontShape{OT1}{pzc}{m}{it}{<-> s * [1.200] pzcmi7t}{}
\DeclareMathAlphabet{\mathpzc}{OT1}{pzc}{m}{it}

\usepackage{relsize}
\usepackage[bbgreekl]{mathbbol}
\DeclareSymbolFontAlphabet{\mathbb}{AMSb} 
\DeclareSymbolFontAlphabet{\mathbbl}{bbold}

\DeclareTextFontCommand{\emdef}{\it}

\newtheorem{thm}{Theorem}
\newtheorem*{thm*}{Theorem}
\newtheorem{lem}[thm]{Lemma}
\newtheorem{cor}[thm]{Corollary}
\newtheorem*{cor*}{Corollary}
\newtheorem{prop}[thm]{Proposition}
\newtheorem*{prop*}{Proposition}

\theoremstyle{definition}
\newtheorem{defn}[thm]{Definition}
\newtheorem{construction}[thm]{Construction}

\newtheorem{notation}[thm]{Notation}

\newtheorem{ex}[thm]{Example}

\newtheorem{rem}[thm]{Remark}

\newtheorem*{conj*}{Conjecture}

\newtheorem{warn}[thm]{Warning}
\newtheorem{assumption}[thm]{Assumption}

\newcommand{\todo}[1]{\textcolor{red}{TODO: #1}}

\definecolor{note_color}{rgb}{0.0,0.7,0.0}

\newcommand{\inj}{\hookrightarrow}
\newcommand{\surj}{\twoheadrightarrow}

\renewcommand{\phi}{\varphi}
\renewcommand{\epsilon}{\varepsilon}

\DeclareMathOperator{\Vect}{Vect}

\newcommand{\Sym}{\mr{Sym}}
\newcommand{\Mod}{\mscr{M}\mr{od}}
\newcommand{\LMod}{\mathrm{LMod}}

\DeclareMathOperator{\gr}{gr}

\DeclareMathOperator{\tr}{tr}

\renewcommand{\Re}{\mathrm{Re}}
\renewcommand{\Im}{\mathrm{Im}}

\DeclareMathOperator{\bBar}{Bar}

\newcommand{\Spc}{ {\mathscr S\mathrm{pc}} }
\newcommand{\Lat}{\mr{Lat}}

\DeclareMathOperator{\Spec}{Spec}
\DeclareMathOperator{\Spf}{Spf}

\DeclareMathOperator{\Perf}{Perf}

\newcommand{\mstack}{\mathpzc}

\DeclareMathOperator{\Stk}{{\mathscr{S}tk}}

\newcommand{\dR}{{\mathrm{dR}}}

\DeclareMathOperator{\Aff}{Aff}

\DeclareMathOperator{\Gal}{Gal}

\newcommand{\Set}{ {\mathrm{Set}} }

\newcommand{\Ab}{\mr{Ab}}

\newcommand{\Hom}{\mr{Hom}}
\DeclareMathOperator{\Ext}{Ext}
\DeclareMathOperator{\Tor}{Tor}

\newcommand{\End}{\mscr E\mr{nd}}

\DeclareMathOperator{\Aut}{Aut}
\DeclareMathOperator{\Map}{Map}

\newcommand{\Fun}{\mscr F\mr{un}}

\newcommand{\Sp}{\mscr{S}\mr{p}}

\DeclareMathOperator{\Ind}{Ind}
\DeclareMathOperator{\sInd}{sInd}

\DeclareMathOperator*{\colim}{colim}
\DeclareMathOperator*{\fib}{fib}

\DeclareMathOperator*{\cofib}{cofib}
\DeclareMathOperator*{\coker}{coker}

\DeclareMathOperator{\Alg}{Alg}
\DeclareMathOperator{\AAlg}{\mscr{A}\mr{lg}}

\DeclareMathOperator{\coAlg}{coAlg}

\DeclareMathOperator{\CAlg}{CAlg}
\DeclareMathOperator{\DAlg}{\mscr{DA}\mr{lg}}

\DeclareMathOperator{\ev}{ev}

\DeclareMathOperator{\Tot}{Tot}

\newcommand{\Bin}{\mr{Bin}}

\DeclareMathOperator{\BinAlg}{\mr{BinAlg}}
\DeclareMathOperator{\DBinAlg}{\mscr{D}\mscr{B}\mr{in}\mscr{A}\mr{lg}}
\DeclareMathOperator{\Id}{Id}
\newcommand{\PShv}{{\mathscr{PS}\mathrm{hv}}}
\newcommand{\hcoShv}{{\mr{h}\mscr C\mr{o}\mathscr S\mathrm{hv}}}
\newcommand{\CoPShv}{{\mscr C\mr{o}\mathscr{PS}\mathrm{hv}}}
\newcommand{\Shv}{{\mathscr S\mathrm{hv}}}
\newcommand{\hShv}{\mr{h}{\mathscr S\mathrm{hv}}}
\newcommand{\op}{{\mathrm{op}}}

\mathchardef\mdef="2D

\usepackage{enumerate}
\usepackage{tocloft}
\usepackage{textcomp}
\DeclareFontFamily{OT1}{pzc}{}
\DeclareFontShape{OT1}{pzc}{m}{it}{<-> s * [1.200] pzcmi7t}{}
\DeclareMathAlphabet{\mathpzc}{OT1}{pzc}{m}{it}

\DeclareMathOperator{\eexp}{\mathpzc{exp}}
\DeclareMathOperator*{\holim}{holim}

\newcommand{\cotimes}{\mathbin{\widehat\otimes}}

\newcommand{\ol}{\overline}

\newcommand{\mf}{\mathfrak}
\newcommand{\mc}{\mathcal}
\newcommand{\mbb}{\mathbb}
\newcommand{\mr}{\mathrm}

\newcommand{\mscr}{\mathscr}
\newcommand{\ul}{\underline}
\newcommand{\xra}{\xrightarrow}

\newcommand{\obl}{\mathrm{obl}}
\newcommand{\RG}{R\Gamma}

\newcommand{\ra}{\rightarrow}

\newcommand{\fg}{{\mathrm{fg}}}
\newcommand{\free}{{\mathrm{free}}}

\newcommand{\id}{{\mathrm{id}}}

\newcommand{\can}{{\mathrm{can}}}

\newcommand{\an}{{\mathrm{an}}}

\DeclareMathOperator{\Bbar}{Bar}
\DeclareMathOperator{\coBar}{Cobar}
\DeclareMathOperator{\Fin}{\mathscr Fin}

\usepackage{tikz}
\usetikzlibrary{matrix}

\newcommand{\Poly}{{\mathrm{Poly}}}
\newcommand{\sing}{{\mathrm{sing}}}
\newcommand{\fpqc}{{\mathrm{fpqc}}}

\newcommand{\tto}{{\xymatrix{\ar[r]&}}}

\let\emptyset\varnothing

\newcommand{\sa}[1]{\overset{#1}{\longrightarrow}}

\newcommand{\mbf}{\mathbf}
\newcommand{\hra}{\hookrightarrow}
\newcommand{\iso}{\simeq}

\newcommand{\bC}{\mathbb C}

\newcommand{\bQ}{\mathbb Q}
\newcommand{\bZ}{\mathbb Z}
\newcommand{\bR}{\mathbb R}

\newcommand{\cO}{\mathcal O}

\newcommand{\fin}{\mathrm{fin}}
\newcommand{\aug}{\mathrm{aug}}
\newcommand{\Op}{\mathrm{Op}}

\numberwithin{thm}{section}
\numberwithin{equation}{section}

\title{Derived binomial rings I: integral Betti cohomology of log schemes}
\author{Dmitry Kubrak, Georgii Shuklin and Alexander Zakharov}

\date{}

\begin{document}
\maketitle
\epigraph{O bin\'omio de Newton \'e t\~{a}o belo como a Vénus de Milo
	
	O que h\'a \'e pouca gente para dar por isso.
	
	\'o\'o\'o\'o — \'o\'o\'o\'o\'o\'o\'o\'o\'o — \'o\'o\'o\'o\'o\'o\'o\'o\'o\'o\'o\'o\'o\'o\'o (O vento l\'a fora)}{Poema de \'Alvaro de Campos}
\bigskip 

\begin{abstract}
We introduce and study a derived version $\mathbf L\mathrm{Bin}$ of the binomial monad on the unbounded derived category $\mathscr D(\mathbb Z)$ of $\mathbb Z$-modules. This monad acts naturally on singular cohomology of any topological space, and does so more efficiently than the more classical monad $\mathbf L\mathrm{Sym}_{\mathbb Z}$ or the monad given by free $\mbb E_{\infty}$-ring. We compute all free derived binomial rings on abelian groups concentrated in a single degree, in particular identifying $C^*_{\mathrm{sing}}(K(\mathbb Z,n),\mathbb Z)$ with $\mathbf L\mathrm{Bin}(\mathbb Z[-n])$ via a different argument than in \cite{Toen_geo} and \cite{Horel}. Using this, we show that the singular cohomology functor $C^*_{\mathrm{sing}}(-,\mathbb Z)$ induces a fully faithful embedding of the category of connected nilpotent spaces of finite type to the category of derived binomial rings, explicitly describing image of the subcategory of simply-connected spaces.

We then also define a version $\mathbf L \mathcal Bin_X$ of the derived binomial monad on the $\infty$-category of $\mathscr D(\mathbb Z)$-valued sheaves on a sufficiently nice topological space $X$. As an application we give a closed formula for the singular cohomology of an fs log complex analytic space $(X,\mathcal M)$:  namely we identify the pushforward $R\pi_*\underline{\mathbb Z}$ for the corresponding Kato-Nakayama space $\pi\colon X^{\log}\rightarrow X$ with the free coaugmented derived binomial ring on the 2-term exponential complex $\mathcal O_X\ra \mathcal M^{\gr}$. This gives an extension of Steenbrink's formula and its generalization in \cite{shuklin} to $\mathbb Z$-coefficients. 
\end{abstract}

\tableofcontents

\section{Introduction}
\subsection{Spoilers for the sequel}
This project grew as an attempt to explain in a reasonable way the isomorphism of graded algebras (\cite[Appendix A]{KubrakPrikhodko_HdR})
\begin{equation}\label{eq:KZ3_vs_BG_a}
H^*_{\sing}(K(\mbb Z,3),\mbb Z)\simeq \oplus_{*_1+*_2=*}H^{*_1}(B\mbb G_{a,\mbb Z},\mc O)_{*_2},
\end{equation}
somewhat strangely relating the coherent cohomology of the classifying stack $B\mbb G_a$ over $\mbb Z$ (purely algebraic data) with singular cohomology of the Eilenberg-MacLane space $K(\mbb Z,3)$ (topological data). Note that this isomorphism shifts cohomological degrees: here the extra $*_2$-grading on $H^{*_1}(B\mbb G_{a,\mbb Z},\mc O)$ is induced by the  action of $\mbb G_m$ on $\mbb G_a$ via quadratic rescaling\footnote{Meaning that the corresponding map $\mbb G_m\times \mbb G_a \ra \mbb G_a$ is given by $(t,x)\mapsto t^2x\in \mbb G_a$, see \cite[Appendix A]{KubrakPrikhodko_HdR} for more details.}.  One can also view $K(\mbb Z,3)$ as the higher classifying stack $B^3\mbb Z$ in the homotopy-theoretic setting; this way \eqref{eq:KZ3_vs_BG_a} relates cohomology of $B^3\mathbb Z$ and $B\mathbb G_a$ in an interesting way.

It turns out that the degree shifts can be eliminated by rewriting \eqref{eq:KZ3_vs_BG_a} as the identification of the cohomology of the free derived binomial ring\footnote{One of the main objectives of this paper is to define this notion.} $\mbf L\Bin(\mbb Z[-3])$ and the free derived divided power algebra $\mbf L\Gamma^*(\mbb Z[-3])$ (regraded via the d\'ecalage isomorphism), respectively.
There is a natural deformation of the functor $\mbf L\Bin$ to the functor $\mbf L\Gamma^*$, and isomorphism \eqref{eq:KZ3_vs_BG_a} can be explained by the fact that this deformation is constant (on the cohomology algebras level) when evaluated on the object $\mbb Z[-3]\in {\mscr D}(\mbb Z)$. In fact, one gets an even stronger identification $\mbf L\Bin(\mbb Z[-3])\simeq \mbf L\Gamma(\mbb Z[-3])$ as $\mbb E_3$-algebras. However, a detailed study of this deformation (and its generalizations) is exactly what we decided to defer to the sequel of this paper. So, without further ado, let us switch to describing the actual contents of the current text.

\subsection{Binomial rings and homotopy theory: a short historical overview}

Let us first recall the notion of a classical binomial ring, following the work of \cite{Hall} and a later revisit by Elliot \cite{Elliot}.

\begin{defn}\label{def:binomial ring}
	A commutative ring $B$ is called a \textit{binomial ring} if it is torsion free as a $\mbb Z$-module and for each $b\in B$ and $n\in \mbb N$ the binomial coefficient 
	$$
	{\tbinom{b}{n}}\coloneqq \tfrac{b(b-1)\cdots(b-n+1)}{n!}\in B\otimes_{\mbb Z}\mbb Q
	$$
	lies in $B\subset B\otimes_{\mbb Z}\mbb Q$.
\end{defn}	

While any $\mbb Q$-algebra obviously gives a binomial ring, it turns out that the condition in \Cref{def:binomial ring} imposes some rather strong restrictions on the arithmetic of $B$, for example the absolute Frobenius on $B/p$ is automatically trivial \cite[Theorem 4.1(4)]{Elliot}. In fact one can fully classify binomial rings as a certain subclass of torsion-free $\lambda$-rings: namely those with all Adams operations being trivial (see \cite{Wilkerson2}). The main examples are given by the ring of integers $\mbb Z$ and, more generally, the ring of $\mbb Z$-valued functions $\Map_\Set(S,\mbb Z)$ on a set $S$.

There has been some history of using binomial rings in homotopy theory; in one way or the other all records seem to be related to attempts of developing some relatively explicit algebraic models for $H\mbb Z$-localization of spaces, see e.g. \cite{Wilkerson}, \cite{Ekedahl}, as well as recent works, \cite{Toen_geo} and \cite{Horel}; for an explicit cochain level approach to homotopy 1-types, e.g. nilmanifolds, see also \cite{suciu2023cup}. The idea is usually to associate to $X\in \Spc$ "polynomial singular cochains" in a certain sense, and then endow them with a suitable homotopy algebraic structure to obtain a fully faithful embedding from some subcategory of spaces. For $H\mbb Q$-localization such desirable model was constructed by Sullivan \cite{Sullivan}, by associating to a simplicial set the corresponding DG-algebra $\Omega^*_{\mr{PL}}(X)$ of polynomial $\mbb Q$-differential forms; unfortunately, blindly replacing $\mbb Q$ with $\mbb Z$ in his construction fails immediately due to non-acyclicity of de Rham complex of $\mbb A^1$ over $\mbb Z$. 

While for general space $X\in\Spc$, say represented by a simplicial set $X_\bullet$, there does not seem to exist a nice explicit integral algebraic model generalizing $\Omega^*_{\mr{PL}}(X)$, there is a natural candidate in the case of an Eilenberg-MacLane space $K(A,n)$ where $A\in \Lat$ is a lattice\footnote{As usual by lattice we mean a finitely generated free abelian group.}.

Recall that  $K(A,n)\in\Spc$ can be computed as the (geometric realization of the) underlying simplicial set of the simplicial  abelian group $\mr{DK}(A[n])_\bullet$ that corresponds to the complex $A[n]$ via Dold-Kan correspondence. More explicitly, 
$$
\mr{DK}(A[n])_m\simeq \bigoplus_{\alpha:[m]\surj [n]} A \simeq A^{\oplus \binom{m}{n}}\in \Lat
$$
with naturally defined face and degeneration maps. Its singular cohomology $C^*_\sing(K(A,n),\mbb Z)$ then is explicitly described as the totalization of the cosimplicial algebra
$$
\Map_\Set(\mr{DK}(A[n])_\bullet, \mbb Z)=\Map_\Set(A^{\oplus \binom{\bullet}{n}},\mbb Z) 
$$
of all set-theoretic maps from $\mr{DK}(A[n])_\bullet$ to $\mbb Z$. Now, using the abelian group structure on $A$, one can consider a smaller cosimplicial algebra:
$$
\Map_{\mr{poly}}(\mr{DK}(A[n])_\bullet, \mbb Z)\subset  \Map_\Set(\mr{DK}(A[n])_\bullet, \mbb Z)
$$
given by those maps $\mr{DK}(A[n])_\bullet \ra \mbb Z$ that are \textit{polynomial} in the sense of difference operators (see \Cref{def:polynomial maps}), or, equivalently are represented by polynomials with rational coefficients. The key observation, which seems to be due to Ekedahl (\cite{Ekedahl}), is that if $n\ge 1$ one can also use $\Map_{\mr{poly}}(\mr{DK}(A[n])_\bullet, \mbb Z)$ to compute singular cohomology of $K(A,n)$, namely,
\begin{equation}\label{eq:singular cohomology vs polynomial maps}
C_\sing^*(K(A,n),\mbb Z)\simeq \Tot_{[\bullet]\in \Delta}\Map_{\mr{poly}}(\mr{DK}(A[n])_\bullet, \mbb Z).
\end{equation}
One can speculate that the cosimplicial algebra on the right is exactly the "polynomial cochains" on $K(A,n)$ in a certain sense.

Now let us indicate how binomial rings, or rather their derived version, naturally come in handy here. One first observes that for any lattice $A\in \Lat$, algebras $\Map_{\mr{poly}}(A, \mbb Z)$ and $\Map_{\Set}(A, \mbb Z)$ are binomial; moreover, $\Map_{\mr{poly}}(A, \mbb Z)$ is the free binomial ring\footnote{Meaning that for any binomial algebra $B$ the set of maps $\Map_{\BinAlg}(\Map_{\mr{poly}}(A, \mbb Z),B)$ is canonically identified with $\Map_{\Ab}(A^\vee, B)$ in the category of abelian groups.} on the $\mbb Z$-linear dual $A^\vee$. This allows to identify the right hand side in \Cref{eq:singular cohomology vs polynomial maps} with the free derived binomial ring on $\Tot(\mr{DK}(A[n])_\bullet^\vee)\simeq A^\vee[-n]$, giving canonical equivalence
$$
C_\sing^*(K(A,n),\mbb Z)\simeq \mbf L\Bin(A^\vee[-n])
$$
as (appropriately defined) binomial rings. This allows to effectively describe maps from $C_\sing^*(K(A,n),\mbb Z)$ considered as a derived binomial ring.
In particular, for any space $X$ the mapping space between derived binomial rings $C_\sing^*(K(A,n),\mbb Z)$ and $C_\sing^*(X,\mbb Z)$ is naturally identified with the mapping space between $X$ and $K(A,n)$ in the category of spaces. This ultimately allows to show (see e.g. \cite{Horel}, or \Cref{ssec:integral homotopy type theory}) that the functor $X\mapsto C_\sing^*(X,\mbb Z)$ from $\Spc^{\op}$ to $\DBinAlg$ is fully faithful when restricted to the subcategory spanned by nilpotent spaces
of finite type.

\subsection{Derived binomial rings: current work}
Hopefully having sufficiently motivated the notion of derived binomial  rings for the reader, let us now describe more concretely what we do in this paper. First of all, taking a slightly more general approach than \cite{Horel}, in \Cref{sec:derived binomial monad} we define the derived binomial monad $\mbf L\Bin$ on the full derived category ${\mscr D}(\mbb Z)$, rather than the coconnective part ${\mscr D}(\mbb Z)^{\ge 0}$. In doing so we use the approach of \cite{Arpon} via \textit{filtered monads}: namely we first define a filtered monad $\Bin^{\le *}$ on the category $\Lat$ of lattices, then extend it to ${\mscr D}(\mbb Z)^{\le 0}$ via the "animation" procedure. Using the polynomiality of individual functors $\Bin^{\le n}$, we then extend it to the full derived category ${\mscr D}(\mbb Z)$ following \cite{BM}, \cite{Arpon} (which in turn go back to \cite{Illusie_Cotangent_II}). We define filtered monad $\Bin^{\le *}$ in terms of additively polynomial maps between abelian groups, with a hope that our exposition will shed more light on the relation between our work and the work of Ekedahl \cite{Ekedahl}. This also allows to describe in more functorial way the individual functors $\Bin^{\le n}\colon \Lat \ra \Lat$. 

After the monad $\mbf L\Bin$ is constructed we define the $\infty$-category $\DBinAlg$ of derived binomial rings as the corresponding category of modules. When restricted to the subcategory $\Ab^{\mr{tf}}$ of torsion-free classical abelian groups, $\mbf L\Bin$ recovers the classical binomial monad $\Bin$ that was studied by Elliot in \cite{Elliot}. This way the category of classical binomial rings embeds fully faithfully into the category of derived ones. Let us note though that in general, differently from the classical case, being a derived binomial ring for a derived commutative ring becomes a structure, and not a property.

There are some immediate benefits of our approach in defining $\mbf L\Bin$: for example, by design, the forgetful functor $\DBinAlg\ra \mscr D(\mbb Z)$ commutes with sifted colimits\footnote{While the forgetful functor from $\DBinAlg$ to $\mbb E_\infty$-algebras in ${\mscr D}(\mbb Z)$ in fact commutes with all colimits.}. On the other hand, when restricted to a certain natural subcategory of $\Perf_{\mbb Z}^{\ge 0}$, $\mbf L\Bin$ can be identified with the right Kan extension from $\Lat$, which altogether makes it possible to efficiently compute in many cases. One also gets a derived binomial algebra structure on the singular cohomology $C^*_\sing(X,\mbb Z)$ simply by considering the homotopy limit $\holim_X \mbb Z$ in $\DBinAlg$.

After setting up the definition, in \Cref{sec:free derived binomial rings} we compute free derived binomial rings on classical abelian groups in a single cohomological degree. These guys are the building blocks of the category of derived binomial rings and cohomology of their underlying complexes can be seen as "binomial cohomological operations". Since $\mbf L\Bin\colon {\mscr D}(\mbb Z) \ra {\mscr D}(\mbb Z)$ commutes with filtered colimits, it is in fact enough to understand $\mbf L\Bin(A[n])$ for $A$ being a lattice or a finite abelian group. 

\begin{itemize} \item \underline{Classical case ($n=0$)}. The free binomial ring $\mbf L\Bin(\mbb Z)=\Bin(\mbb Z)$ is described explicitly as the subring
$$
\Bin(\mbb Z)\simeq \mbb Z[\tbinom{x}{n}]_{n\ge 1}\subset \mbb Q[x]
$$
generated by binomial coefficients
$$
\tbinom{x}{n}\coloneqq \tfrac{x(x-1)\ldots(x-n+1)}{n!}\in \mbb Q[x],
$$
and, more functorially, as the ring of polynomial maps $\Map_{\mr{poly}}(\mbb Z,\mbb Z)$. One can also identify $\Bin(\mbb Z)$ with the ring of functions on the "Hilbert additive group scheme" $\mbb H$, defined as Cartier dual to the formal multiplicative group $\widehat{\mbb G}_m$. In general, for a lattice $A\in \Lat$ one has a canonical isomorphism $\mbf L\Bin(A)\simeq \Map_{\mr{poly}}(A^\vee,\mbb Z)$, where $A^\vee$ is the dual lattice.

\item \underline{(Most of the) coconnective range}. Similarly to \Cref{eq:singular cohomology vs polynomial maps} in (the most of) coconnective range the free derived binomial rings have a nice topological interpretation, namely (\Cref{prop: description of free coconnective binomial algebras}) for $M\in \Perf^{\le -1}_{\mbb Z}$ one has a natural equivalence:
\begin{equation}\label{eq:free binomial things in the coconnective range}
\mbf L\Bin(M^\vee)\xra{\sim } C^*_\sing(K(M),\mbb Z),
\end{equation}
where $M^\vee\coloneqq R\Hom_{\mscr D(\mbb Z)}(M,\mbb Z)\in \Perf^{\ge 1}_{\mbb Z}$ is dual complex to $M$, and $K(M)\in \Spc$ is the (generalized) Eilenberg-MacLane space of $M$. In particular, for $M=\mbb Z[n]$ with $n\ge 1$ we get an equivalence 
$$
\mbf L\Bin(\mbb Z[-n])\xra{\sim } C^*_\sing(K(\mbb Z,n),\mbb Z).
$$ We give a different proof than in \cite{Toen_geo} and \cite{Horel}, formally reducing to the case $M=\mbb Z[-1]$, and then deforming $\mbf L\Bin(\mbb Z[-1])$ to the free divided power algebra $\mbf L\Gamma(\mbb Z[-1])$. 

Note that the isomorphism \Cref{eq:free binomial things in the coconnective range} computes $\mbf L\Bin(A[-n])$ where $A$ is a lattice and $n\ge 1$ or $A$ is a finite abelian group and $n\ge 2$. 

\item \underline{Strictly connective range}. It turns out that in the connective range all the interesting arithmetic of binomial rings that we saw before collapses. More precisely, one gets the following result, which might look somewhat surprising at first glance (\Cref{prop:free binomial algebra in strictly cocnnective setting}): for $M\in \Perf^{\le -1}_{\mbb Z}$ there is a natural equivalence:
$$
\mbf L\Bin(M) \xra{\sim} \mbf L\Sym_{\mbb Z}(M\otimes_{\bZ} \mbb Q).
$$
In particular, in this range the functor $M\mapsto \mbf L\Bin$ only depends on the tensor product $M\otimes_{\bZ} \mbb Q$. Consequently, for a finite abelian group $A$, $\mbf L\Bin(A[n])\simeq \mbb Z$ for any $n\ge 1$. In fact, it is also true that $\mbf L\Bin(A[n])\simeq \mbb Z$ for $n=0$ (\Cref{prop:Bin of torsion abelian group}). For a lattice $A$, via the above isomorphism, $\mbf L\Bin(A[n])$ can be non-canonically identified with the homology $C_*^\sing(K(A\otimes_{\bZ} \mbb Q,n),\mbb Z)$ of the $n$-th Eilenberg-MacLane space of the rationalization $A\otimes_{\bZ  } \mbb Q$.

\item \underline{$A[-1]$ for $A$ finite}. The above results leave out the only case $\mbf L\Bin(A[-1])$ for $A$ finite abelian. Here we show that $\mbf L\Bin(A[-1])$ is 1-cotruncated, namely ${\mbf L^{>1}\Bin(A[-1])}\simeq 0$, and that the only interesting information is given by the functor $A\mapsto \mbf L^1\Bin(A[-1])$, which we describe explicitly in \Cref{ssec:LBin(A[-1])}. As an example of our computation, we $(\mbb Z/p)^\times$-equivariantly identify 
$$\mbf L^1\Bin(\mbb Z/p[-1])\simeq K/\mc O_K$$
where $K=\mbb Q_p(\mu_p)$ is the $p$-th cyclotomic extension and $(\mbb Z/p)^\times$ acts by multiplication on the left and via the Galois group $\Gal(K/\mbb Q_p)$ on the right (see \Cref{rem:case A=Z/p}).
\end{itemize}

Using the equivalence \Cref{eq:free binomial things in the coconnective range} in \Cref{ssec:integral homotopy type theory}, similarly to Horel \cite{Horel}, we establish a variant of integral homotopy theory in our setting. Namely, we show that the functor 
$$
C^*_\sing(-,\mbb Z)\colon \Spc^\op \ra \DBinAlg
$$
is fully faithful when restricted to the full subcategory spanned by connected nilpotent spaces of finite type (see \Cref{cor:cohomology defines fully faithful embedding}). In particular, such a connected space $X$ can be reconstructed from $C^*_\sing(X,\mbb Z)$ with its derived binomial ring structure: namely, there is a natural equivalence $X\simeq \Map_{\DBinAlg} (C^*_\sing(X,\mbb Z),\mbb Z)$ (see \Cref{prop:X can be reconstructed from cochains}). Let us note that very similar results were also obtained recently by Antieau \cite{Antieau_Spherical} in a slightly different context. In \Cref{prop:image of simply-connected spaces} also describe the essential image of $C^*_\sing(-,\mbb Z)$ when restricted to the subcategory $\Spc^{1-\mr{cn},\mr{ft}}\subset \Spc$ of simply connected spaces of finite type: it is spanned by those coconnective derived binomial derived rings $B$ such that $H^i(B)$ is finitely generated for all $i$, $H^{0}(B)\simeq \mbb Z$, $H^1(B)=0$ and $H^2(B)$ is torsion free.

\begin{rem} As we discuss  in \Cref{ssec:interpretation as cohomology of a stack}, the above isomorphism 
	$\mbf L\Bin(\mbb Z[-n])\xra{\sim } C^*_\sing(K(\mbb Z,n),\mbb Z)$ can be interpreted as a comparison between  the cohomology of the structure sheaf of two higher classifying stacks: namely, $K(\mbb Z,n)$ and $K(\mbb H,n)$ (where $\mbb H\coloneqq \Spec(\Bin(\mbb Z))$). The comparison map comes naturally from the homomorphism of group schemes
	$$
	\iota\colon \ul{\mbb Z} \ra \mbb H
	$$
	that on the level of functions corresponds to the embedding $\Bin(\mbb Z)\simeq \Map_{\mr{poly}}(\mbb Z,\mbb Z)\subset \Map_{\Set}(\mbb Z,\mbb Z)$ of polynomial functions into all.
	
	This identification allows to interpret $C^*_\sing(K(\mbb Z,n),\mbb Z)$ in terms of an object of more algebro-geometric nature (namely, the higher stack $K(\mbb H,n)$).
	 One of the new structures that one gets from this, but which is not immediately seen topologically, is the deformation\footnote{Here $\mbb G_a^{\sharp}\coloneqq \Spec (\Gamma(\mbb Z))$ is the Cartier dual to formal additive group $\widehat{\mbb G}_a$. One gets a deformation $\mbb H\rightsquigarrow \mbb G_a^\sharp$ from the canonical deformation of $\widehat{\mbb G}_m$ to its formal Lie algebra, which is $\widehat{\mbb G}_a$.} from $K(\mbb H,n)$ to $K(\mbb G_a^{\sharp},n)$, which, as we will show in the sequel, turns out to be trivial on the level of cohomology of the structure sheaf with respect to some (e.g. $\mbb E_n$-algebra) structures.
\end{rem}

\subsection{Singular cohomology of log analytic spaces}\label{intro:singular cohomology of log-schemes}
The main new application of this paper is to integral singular cohomology of log analytic spaces over complex numbers.
Let us briefly recall what sort of structure this amounts to.

\begin{defn}[{\cite[Definition 1.1.1]{Kato-Nak-99}}]
	Recall  that a log (complex) analytic space $(X,\mathcal M)$ is given by:
	\begin{itemize}
		\item[-] a complex analytic space $X$;
		\item[-] a sheaf of commutative monoids $\mathcal M$;
		\item[-] a homomorphism $\alpha\colon \mathcal M\to \mathcal O_X$ such that $\alpha^{-1}(\cO^\times_X)\iso \cO^\times_X$.
	\end{itemize}
\end{defn}

We will assume that our log analytic spaces are \textit{fine and saturated (fs)}, this imposes some nice properties of the sheaf $\mc M$, as well as its group completion $\mc M^{\gr}$ (see (\ref{ass:fs log anal})). In particular, the stalks of $\ol{\mc M^{\gr}}\coloneqq \mc M^{\gr}/\mc O_X^\times$ (which control the non-trivial part of log-structure) are free finitely generated abelian groups.

The classical approach to define log Betti (=singular) cohomology of a log analytic space, due to Kato and Nakayama \cite{Kato-Nak-99}, is to associate to $(X,\mc M)$ a topological space $X^{\log}$ (\Cref{constr:Kato-Nakayama space}) and define log Betti cohomology $\RG^{\log}_{\mr{Betti}}(X,\bZ)$ as the singular cohomology $\RG_\sing(X^{\log}, \mbb Z)$ of $X^{\log}$. To give the reader a feeling of what $X^{\log}$ looks like, let us just say for now that there is a continuous proper map $\pi\colon X^{\log}\ra X$ and that the fiber $\pi^{-1}(x)$ over a point $x\in X$ is given by a compact real torus with dimension given by the rank of the fiber of $\ol{\mc M^{\gr}}$ at $x$.

However, one can also look for alternative, more "direct", formulas for the log Betti cohomology that would only directly involve the analytic space $X$ and the monoid $\mc M$, and would not use any auxiliary topological space. One variant of such construction was studied in great detail in \cite{shuklin} in the case of $\mbb Q$-coefficients. The key observation is that all necessary information in fact is captured by the 2-term exponential complex (\Cref{not:exp(a)})
$$
\eexp(\alpha)\coloneqq [\ \! \mc O_X \xra{\exp_{\mc M}} \mc M^{\gr}\ \!]
$$
of sheaves on $X$, where the map $\exp_{\mc M}$ is given by the composition of the exponential map $\mc O_X\ra \mc O_X^\times$ and the composite embedding $\mc O_X^\times \inj \mc M \inj \mc M^{\gr}$. One can show that the complex $\eexp(\alpha)$ gives an explicit model for $R^{\le 1}\pi_*\ul{\mbb Z}$ (\Cref{lem:exp is the first truncation}); moreover, even though the definition of $\eexp(\alpha)$ uses the analytic structure on $X$, in the case when $(X,\mc M)$ is the analytification of an fs log scheme, it is essentially algebraic in the sense that it comes as the Betti realization of a certain motivic sheaf (see \cite{shuklin} for more details).

To compute $\RG^{\log}_{\mr{Betti}}(X,\bZ)$ one can proceed in two steps: namely, first evaluate the pushforward $R\pi_* \ul{\mbb Z}$ and then compute its (derived) global sections after that. So the goal is to give a closed formula for the full derived pushforward $R\pi_* \ul{\mbb Z}$ in terms of its first truncation $R^{\le 1}\pi_*\ul{\mbb Z}\simeq \eexp(\alpha)$.

Let us sketch the case of $\mbb Q$-coefficients following \cite{shuklin} and then indicate what needs to be changed to get the case of $\mbb Z$-coefficients.
Denote by
$$
\eexp(\alpha)_\bQ:=\mathpzc{exp}(\alpha)\otimes \bQ
$$
the rationalization of $\eexp(\alpha)$. The identification $\ker(\exp_{\mc M})=\ul \bZ$ induces maps
$$\ul \bZ \ra \eexp(\alpha),  \qquad \ul{\mbb Q}\ra \eexp(\alpha)_{\mbb Q}$$
of complexes of sheaves on $X$.

Using the natural commutative algebra structure on $R\pi_*\ul{\mbb Q}$ the map $\eexp(\alpha)_\bQ\simeq R^{\le 1}\pi_*\ul{\mbb Q} \ra R\pi_*\ul{\mbb Q}$ extends to a commutative algebra map 
$$\Sym_{\mbb Q}^*(\eexp(\alpha)_\bQ) \dasharrow R\pi_*\ul{\mbb Q}$$ from the free commutative $\mbb Q$-algebra on $\eexp(\alpha)$. Moreover, its composition with the natural map $\Sym_{\mbb Q}^*(\ul{\bQ})\ra \Sym_{\mbb Q}^*(\eexp(\alpha)_\bQ)$ factors through $R^0\pi_*\ul{\bQ}\simeq \ul{\bQ}$ via the unique map whose restriction to $\Sym^1 (\ul{\bQ}) \simeq \ul{\bQ}$ gives the identity map  $\id_{\ul{\bQ}}\colon \ul{\bQ}\ra \ul{\bQ}$. Altogether, this produces a map 
$$
\Sym^{\mr{coaug}}_{\mbb Q}(\eexp(\alpha)_\bQ) \dasharrow R\pi_* \ul{\mbb Q}
$$
from the pushout\footnote{Or, in other words, from the tensor product $\Sym^{\mr{coaug}}_\bQ(\eexp(\alpha)_\bQ)\simeq \ul{\mbb Q} \otimes_{\Sym^*_{\bQ} (\ul{\bQ})}\Sym^*_{\bQ} (\eexp(\alpha)_\bQ)$.}
$$
\xymatrix{\ul{\mbb Q}\ar[r]& \Sym^{{\mr{coaug}}}_{\bQ}(\eexp(\alpha)_\bQ)\\
	\Sym^*_{{\bQ}} (\ul{\bQ})\ar[r]\ar[u]&  \Sym^*_{\bQ} (\eexp(\alpha)_\bQ)\ar[u],
}
$$
which induces an equvalence
\begin{equation}\label{intro_eq:Q-coefficients}
\Sym^{\mr{coaug}}_{\mbb Q}(\eexp(\alpha)_\bQ) \simeq R\pi_* \ul{\mbb Q}.
\end{equation}

 Once the map is constructed, the proof of the latter equivalence is not hard and, using proper base change, ultimately reduces to the fact that the natural map
 \begin{equation}\label{eq:map to circle over Q}
 \Sym_{\bQ}^*(\mbb Q[-1])\ra  \RG_\sing(S^1,\mbb Q)
 \end{equation}
is an equivalence. Let us also note that the formula \Cref{intro_eq:Q-coefficients} was first observed by Steenbrink (see \cite[Section 2, (2.8)]{Steenbrink} and \cite[Section 11.2.6]{PS_Hodge}) for log structures associated with normal crossing divisors.

However, if we now try to follow the same strategy to get a formula for $\mbb Z$-coefficients we immediately run into trouble. The problem is with the functor $\Sym$; even though there are different options for the commutative algebra context in which to consider $R\pi_* \ul{\mbb Z}$ as an algebra, the free algebras in all of them will be too large. For example, the free $\mbb E_\infty$-algebra functor now has a lot of non-trivial negative cohomology groups and, even if we correct that by considering the free derived commutative algebra functor $\mathbf L\Sym^*_{\mbb Z}$ defined by \cite{Arpon}, the analogously defined $\mathbf L\Sym^{\mr{coaug}}_{\mbb Z}(\eexp(\alpha))$ would still be too big. In the latter case the problem is that \Cref{eq:map to circle over Q} does not work with $\mbb Z$-coefficients: the natural map 
$$
\mathbf L\Sym_{\mbb Z}(\mbb Z[-1]) \ra \RG_\sing(S^1,\mbb Z)
$$
is far from being an equivalence\footnote{For an explicit description of $\mathbf L\Sym_{\mbb Z}(\mbb Z[-1])$ see \cite{KubrakPrikhodko_HdR}, where one first must identify $\mathbf L\Sym_{\mbb Z}(\mbb Z[-1])$ with $\RG(B\mbb G_a,\mc O)$.}, e.g. $H^2(\mathbf L\Sym_{\mbb Z}(\mbb Z[-1]))$ is not even a finitely generated abelian group.

Note now that $S^1\simeq K(\mbb Z,1)$ and by \Cref{eq:free binomial things in the coconnective range} the natural map 
$$
\mbf L\Bin(\mbb Z[-1]) \ra \RG_\sing(S^1,\mbb Z)
$$
\underline{\textit{is}} an equivalence. It thus might be natural to try to use the free binomial ring $\mathbf L\Bin$ in place of the free derived commutative algebra monad $\mathbf L\Sym_{\mbb Z}$. 

More precisely, in \Cref{ssec:sheaves of binomial algebras}, under some mild assumptions on the topological space $X$, we define a derived binomial monad $\mbf L\mc B{in}_X$ on the $\infty$-category $\Shv(X,{\mscr D}(\mbb Z))$ of ${\mscr D}(\mbb Z)$-valued sheaves on $X$. We identify sheaves of derived binomial rings with modules over $\mbf L\mc B{in}_X$, and for a (nice enough) map $\pi\colon Y\ra X$ of topological spaces we endow the pushforward $R\pi_*\ul\bZ$ with a natural module structure over $\mbf L\mc B{in}_X$. In \Cref{ssec:cohomology of KN-space} we then show that the natural map 
$$
\mbf L \mc B{in}^{\mr{coaug}}_X(\eexp(\alpha)) \dasharrow R\pi_*\ul\bZ,
$$ 
induced by $\eexp(\alpha)\ra R\pi_*\ul\bZ$, is an equivalence. Here, similar to what we had before, $\mbf L \mc B{in}^{\mr{coaug}}_X(\eexp(\alpha))$ is defined as the pushout 
$$
\xymatrix{\ul{\mbb Z}\ar[r]& \mbf L\mc B{in}_X^{\mr{coaug}}(\eexp(\alpha))\\
	\mbf L\mc B{in}_X(\ul{\bZ})\ar[r]\ar[u]^{\ev_1}&  \mbf L \mc B{in}_X(\eexp(\alpha)),\ar[u]
} 
$$
where the map $\ev_1\colon \mbf L\mc B{in}_X(\ul{\bZ}) \ra \ul{\bZ}$ is induced by the binomial algebra structure on $\mbb Z$. In fact, as we show in \Cref{ssec:LBin'} the functor $\mbf L\mc B{in}_X^{\mr{coaug}}$ can be considered as the free derived binomial algebra, but in the undercategory $\Shv(X,{\mscr D}(\mbb Z))_{\setminus \ul \bZ}$

 As the result we get a formula 
\begin{equation}\label{intro_eq:pushforward as Bin}
	R\pi_*\ul\bZ \simeq \mbf L \mc B{in}^{\mr{coaug}}_X(\mathpzc{exp}(\alpha)).
\end{equation}	
Taking global sections also leads to an equivalence
$$
\RG^{\log}_{\mr{Betti}}(X,\bZ)\iso \RG(X,\mbf L\mc B\mr{in}_X^{\mr{coaug}}(\eexp(\alpha)))
$$
expressing \textit{integral} log Betti cohomology directly in terms of $\eexp(\alpha)$, extending Steenbrink's formula, as well as its generalization in \cite{shuklin}, to the integral case.

\subsection{Cosheaf of spaces corresponding to $\pi\colon X^{\log} \ra X$ and its algebraic approximation $X^{\log}_{\mbb H}$}

In fact not only the log Betti cohomology of $X$ can be expressed via $\mathpzc{exp}(\alpha)$ using the functor $\mbf L\mc B\mr{in}_X^{\mr{coaug}}$, but also the essential homotopy-theoretic data underlying the map $\pi\colon X^{\log} \ra X$: more precisely, the cosheaf of spaces that it naturally defines. 

More precisely, let $[X^{\log},\pi]$ denote the the $(\infty,1)$-cosheaf of spaces on $X$ which sends $U$ to the underlying homotopy type of the preimage $\pi^{-1}(U)\in \Spc$. Our version of integral homotopy theory (\Cref{prop:X can be reconstructed from cochains}) implies that the cosheaf $[X^{\log},\pi]$ can be reconstructed from $\mathpzc{exp}(\alpha)$ as the cosheaf of spaces
$$
U\mapsto \Map_{\DBinAlg}(\mbf L\mc B\mr{in}_X^{\mr{coaug}}(\eexp(\alpha))(U), \mbb Z)
$$
(see \Cref{rem:reconstructing KN from exp}). 

As we show in \Cref{ssec:cosheaf shit} the cosheaf $[X^{\log},\pi]$ in fact also enjoys some better properties than the actual map of topological spaces $\pi\colon X^{\log} \ra X$: for example (\Cref{rem:gerbe structure on Y}), one can make sense of $[X^{\log},\pi]$ as a gerbe over cosheaf of abelian groups given by $(\ol{\mc M^{\gr}})^\vee$ defined by sending $U$ to $\ol{\mc M^{\gr}}(U)^\vee\coloneqq \Hom_{\Ab}(\ol{\mc M^{\gr}}(U),\mbb Z)$. Intuitively, one would like to say similarly that $X^{\log}$ itself is a gerbe over $\ol{\mc M^{\gr}}$ (or rather a torsor over $\ol{\mc M^{\gr}}\otimes_{\mbb Z} S^1$, which from homotopical point of view is the same as $\ol{\mc M^{\gr}}[1]$), but we don't know in what sense this could be made precise. The key tool which allows to interpret $[X^{\log},\pi]$  as a gerbe is \Cref{prop:local homotopy triviality of subanalytic maps} giving a certain enhancement of proper base change theorem for locally semi-analytic maps (\Cref{def:locally semi-analytic}).

The gerbe structure on $[X^{\log},\pi]$ also allows us to geometrize the equivalence \Cref{intro_eq:pushforward as Bin}, which we do in \Cref{ssec:geometrization}. Namely, one can embed the category of cosheaves of spaces on $X$ into the category of cosheaves with values in higher stacks via the natural colimit-preserving functor $\Spc\ra \Stk_{/\mbb Z}$. The group scheme homomorphism $i\colon \ul{\mbb Z} \ra \mbb H$ extends to the map of $\Stk_{/\mbb Z}$-valued cosheaves $(\ol{\mc M^{\gr}})^\vee\ra (\ol{\mc M^{\gr}})^\vee\!\otimes_{\mbb Z}\mbb H$ and we define the cosheaf of stacks $X^{\log}_{\mbb H}$ (which we propose to be the "algebraic version" of $\pi\colon X^{\log}\ra X$) as the induced $(\ol{\mc M^{\gr}})^\vee\!\otimes_{\mbb Z}\mbb H$-gerbe from $[X^{\log},\pi]$ via the above map. One has a natural map 
$$
i_X\colon [X^{\log},\pi] \ra X^{\log}_{\mbb H}
$$
and one obtains the isomorphism \Cref{intro_eq:pushforward as Bin} by applying the derived global sections $\RG(-,\mscr O)$ (\Cref{lem:main iso through cosheaves}) to $i_X$.

\begin{rem} The construction of the gerbe $X^{\log}_{\mbb H}$ is in fact more general and applies to maps $\pi\colon Y\ra X$ with the property that locally on $U\subset X$ the preimage $\pi^{-1}(U)$ is homotopy equivalent to a compact real torus. In the case when $\pi$ is given by toric fibration, the nature of the corresponding cosheaf of stacks $Y_{\mbb H}$ is somewhat analogous to the Hodge-Tate gerbe $X^{\mr{HT}}$ for a scheme over $\mbb F_p$ (see \cite{Bhatt_Lurie_prism_for_scheme}). 
	
	Namely, $p_{\mr{HT}}\colon X^{\mr{HT}}\ra X$ also represents a cosheaf of stacks that is a gerbe over the PD-envelope $T_X^\sharp$ of the tangent bundle, and the latter is Cartier dual to a a formal group given by the completion at 0 of the cotangent bundle. The trivialization of this gerbe implies the formality of the Frobenius pushforward of the de Rham complex $F_{X*}\Omega_{X,\dR}^*$. On the other hand, the first truncation $\tau^{\le 1}F_{X*}\Omega_{X,\dR}^*$ splits if and only if so does the gerbe $p_{DI}\colon X^{DI}\ra X$ of Frobenius-liftings, defined by Deligne-Illusie in \cite{DeligneIllusie}. 
	
	The cosheaf of stacks $Y_{\mbb H}\ra Y$ is a gerbe over $(R^1\pi_*\ul\bZ)^\vee\otimes_{\mbb Z} \mbb H$, which Cartier dual to a sheaf of formal groups $R^1\pi_*\ul\bZ\otimes_{\mbb Z} \widehat{\mbb G}_m$. Differently from Hodge-Tate stack, the gerbe $Y_{\mbb H}$ is trivialized once one splits $R^{\le 1}\pi_*\ul\bZ$; the formality of the whole pushforward $R\pi_*\ul\bZ$ is then controlled by the triviality of a certain deformation of cosheaves $(R^1\pi_*\ul\bZ)^\vee\otimes_{\mbb Z} \mbb H \rightsquigarrow (R^1\pi_*\ul\bZ)^\vee\otimes_{\mbb Z} \mbb G_a^\sharp$, after applying $\RG(-,\mscr O)$.
	
\end{rem}	

\begin{rem}
Unfortunately, it does not seem to be feasible that there is a lift of the binomial monad to the motivic level that would allow to associate an \textit{integral} motive to a log structure, similarly to the construction in \cite{shuklin} in the $\mbb Q$-coefficients case.  The main reason that we went through some analysis of the cosheaves $[X^{\log},\pi]$ and $X^{\log}_{\mbb H}$ is that we expect that their construction could have a motivic lift instead. We also hope that there should be a way to define the relative motivic cohomology of $X^{\log}_{\mbb H}$ over $X$ even in the non $\mbb A^1$-homotopy-invariant context (as in \cite{Iwasa-Annala}) with $\mbb H$ replaced by the Cartier dual to a certain motivic formal group law. We hope to proceed with the construction of these generalizations in the close future.
\end{rem}	

\subsection{Related and future work} 
\subsubsection*{Binomial rings, $\delta$-rings and integral models for spaces} The starting point for our project was To\"en's work \cite{Toen_geo} on Grothendieck's schematization problem; we then realized that the key computation of singular cohomology of $K(\mbb Z,n)$ could be also proved in another way by using the derived binomial ring structure. During the active phase of our project, the work \cite{Horel} by Horel came out, which used a similar idea (though with another proof) for the latter computation and which also established integral homotopy theory in the context of \textit{cosimplicial} binomial rings. The main difference with our work is that we use the whole derived category $\mscr D(\mbb Z)$ rather than the cosimplicial part and we also avoid a direct use of the language of model categories; considering applications to integral cohomology both approaches seem to work well. We should point out that the integral homotopy theory was essentially established already by Ekedahl in \cite{Ekedahl}, though the language of $\infty$-categories allows for an (arguably) cleaner theory and formulation of final results. We should also probably mention the beautiful work of Mandell \cite{Mandell} where he showed that two nilpotent spaces of finite type are weakly homotopy equivalent if and only if their singular $\mbb Z$-cochains are isomorphic as $\mbb E_\infty$-algebras; note however that fully-faithfulness of the functor $C^*_\sing(-,\mbb Z)$ breaks in this context. A related dual comonad on connective abelian groups in relation to spaces was also studied in. There is another recent approach to integral homotopy theory \cite{Yuan} by Yuan, which is much more technically involved; we don't know what is the precise relation between his work and derived binomial rings, though both approaches model in a way the "trivialization of Frobenii action".

A version of $\mbf L \Bin$ over $\mbb Z_p$ should be related to the "free $\delta$-ring monad" $\mbf L\Sym^{\delta}$ that was defined in \cite{Holeman}: more precisely, $\mbf L \Bin$ should be given by the maximal Frobenius-invariant quotient of $\mbf L\Sym^{\delta}$ in a certain sense. In particular, any derived binomial ring over $\mbb Z_p$ should have an underlying structure of derived $\delta$-ring. Given the main result \cite[Theorem 1.1.3]{Holeman} of \textit{loc.cit.} that realizes prismatic cohomology functor as a left adjoint of the reduction functor from derived $\delta$-rings to derived commutative rings, it would be interesting to also find a purely categorical interpretation of the \'etale comparison for the prismatic cohomology \cite[Theorem 1.8(4)]{BS_prisms} in terms of the category of derived binomial algebras over $\mbb Z_p$.

The relation of our work with derived $\delta$-rings was made even more apparent by a very recent work \cite{Antieau_Spherical} of Antieau, where yet another integral model for spaces was constructed, with the target category being the fixed-point category of derived $\lambda$-rings with respect to Adams operations (see \cite[Section 5]{Antieau_Spherical}). Clarifying the precise relation between this category and $\DBinAlg$ is one of the things that the first author is planning to do as a part of a joint project with Antieau, H\"ubner and Nuiten.
\subsubsection*{Betti cohomology of log analytic spaces}
Passing to the subject of log analytic spaces we should of course mention the work \cite{Steenbrink} of Steenbrink which discusses in some detail the case when the log structure is associated with a normal crossing divisor. We hope that our work can shed more light on section \cite[(2.8)]{Steenbrink}, where the integral part of the story is discussed very briefly. The complex Steenbrink writes down in \cite[(2.8)]{Steenbrink} can be naturally identified with $\mbf L\Gamma^{n}_{\mbb Z}(R^{\le 1}\pi_{*}\ul{\mbb Z})$; he shows that it compares with $R\pi_*\ul\bZ$ for $n\gg 0$ after  $-\otimes\mbb Q$. Our work provides an integral version of his result for any fs log analytic space: namely the full pushforward $R\pi_*\ul\bZ$ is given by $\mbf L\mc{B}in_X^{\mr{coaug}}(R^{\le 1}\pi_{*}\ul{\mbb Z})$. Despite his claim in the very end of  \cite[(2.8)]{Steenbrink} about the existence of a natural map $\mbf L\Gamma^{n}_{\mbb Z}(R^{\le 1}\pi_{*}\ul{\mbb Z}) \ra R\pi_*\ul\bZ$ before $-\otimes\mbb Q$ we do not immediately see it from the binomial perspective\footnote{Instead, we clearly see a map $$\cofib \left(\mbf L\mc Bin_X^{\le n-1} (R^{\le 1}\pi_{*}\ul{\mbb Z}) \xra{[1]-1} \mbf L\mc Bin_X^{\le n}(R^{\le 1}\pi_{*}\ul{\mbb Z})\right)\dashrightarrow \mbf L\mc{B}in_X^{\mr{coaug}}(R^{\le 1}\pi_{*}\ul{\mbb Z}),$$ (here $[1]$ is the class in $\Bin^{\le 1}(\mbb Z)$ corresponding to the image of 1 in $\id(\mbb Z)$ under the decomposition of functor $\Bin^{\le 1}$ as $\id\oplus \ul{\bZ}$)
	which indeed becomes an equivalence after tensoring with $\mbb Q$ when $n\gg 0$. The sheaf $\mbf L\Gamma^{n}_{\mbb Z}(R^{\le 1}\pi_{*}\ul{\mbb Z})$ can be identified with the similar cofiber, but with $\mbf L\mc{B}in_X$ replaced by $\mbf L\Gamma_{\mbb Z}$. However, since the latter is the associated graded of the former it is not clear a priori what the natural map is, unless we assume that there is some further splitting of $\mbf L\mc Bin_X$ as $\mbf L\Gamma_{\mbb Z}$.}. In any case, we propose that $\mbf L\Gamma^{n}_{\mbb Z}(R^{\le 1}\pi_{*}\ul{\mbb Z})$ for $n\gg 0$ might as well be simply replaced by the sheaf $\mbf L\mc{B}in_X^{\mr{coaug}}(R^{\le 1}\pi_{*}\ul{\mbb Z})$ since it directly computes the full pushforward $R\pi_{*}\ul{\mbb Z}$. In the case of a simple normal crossings divisor it thus would be nice to give a description of the weight filtration on $R\pi_*\ul{\mbb Z}$ in terms of a suitable explicit model for $\mbf L\mc{B}in^{\mr{coaug}}_X(R^1\pi_{*}\ul{\mbb Z})$, similarly to \cite[(2.8)]{Steenbrink}.

Besides that, in \cite{Achinger_Ogus} Achinger and Ogus in certain cases describe the differential $d_2$ on the $E_2$-page of the Leray-Serre spectral sequence for $\pi\colon X^{\log}\ra X$. It would be interesting to explain their computations through the binomial ring perspective: namely the truncations $R^{[i,i+1]}\pi_*\ul{\mbb Z}\coloneqq \tau_{\le i+1}\tau_{\ge i}R\pi_*\ul{\mbb Z}$, which control $d_2$, are described in terms of the functor $\Bin^{[i,i+1]}\coloneqq \Bin^{\le i+1}\!\!/\Bin^{\le i-1}$ which is a fairly explicit extension of $\Gamma^{i+1}_{\bZ}$ by $\Gamma^i_{\bZ}$, at least when restricted to classical abelian groups. We plan to investigate this relation in the close future. Also, in the sequel we are going to use the "binomial" interpretation of $R\pi_*\ul{\mbb Z}$ to prove some general formality results for $R\pi_*\ul{\mbb Z}$ in certain contexts: as discussed in \Cref{ex:when first truncations splits} this is directly controlled by the deformation of functors from $\Bin$ to $\Gamma_{\mbb Z}$.

\subsection{Acknowledgments} We would like to thank Piotr Achinger, Ahmed Abbes, Ben Antieau, Emma Brink, Dustin Clausen, Ofer Gabber, Sergei O. Ivanov, Ryomei Iwasa, Shizhang Li, Akhil Mathew, Chikara Nakayama, Arthur Ogus, Grisha Papayanov, Sasha Petrov, Artem Prikhodko, Arpon Raksit, Peter Scholze, Vova Sosnilo, Bernard Teissier, Bertrand To\"en, Burt Totaro, Antoine Touz\'e and Dmitry Vaintrob for useful discussions related to this project. We are grateful to Alexandra Utiralova for helpful comments on the earlier versions of the draft.  We would also like to especially thank Vadim Vologodsky, whose interest in this project served as a great motivation for us. Finally we would like to thank the anonymous referee for the valuable suggestions and comments.

The first author is grateful to IHES for providing excellent work conditions during the time that he was working on this project. The second author gratefully thanks the Faculty of Mathematics of the Higher School of Economics, during his employment at which the work on this article was carried out. The work of the second author was supported by the Russian Science Foundation, grant 21-11-00153. The study has been partially supported by the Laboratory Of Algebraic Geometry and its Applications funded within the framework of the HSE University Basic Research Program.

\section{Derived binomial monad}\label{sec:derived binomial monad}

Let $\Ab$ denote the category of abelian groups and let $\CAlg\coloneqq \CAlg(\Ab)$ be the category of classical commutative rings. The forgetful functor $ \CAlg\ra \Ab$ has a left adjoint given by the free commutative ring $\Sym_{\mbb Z}\coloneqq \oplus_{i\ge 0} \Sym_{\mbb Z}^i$. The composition of $\Sym_{\mbb Z}$ with the forgetful functor then defines a monad on $\Ab$ and the natural functor\footnote{Namely, given a commutative ring $A$ we have a natural ring homomorphism $\Sym_{\mbb Z}(A)\ra A$, which is given by the identity map $\id_A$ on $A\simeq \Sym_{\mbb Z}^1(A)$, and then extended by multiplicativity.} $\CAlg \ra \mr{Mod}_{\Sym_{\mbb Z}}(\Ab)$ to the category of modules over this monad is an equivalence: prescribing the commutative algebra structure on an abelian group $M$ is equivalent to giving a map of abelian groups $\Sym_{\mbb Z}(M) \ra M$ that defines the module structure over the above monad. 

The goal of this section is to define another monad on the category of abelian groups, namely the \textit{binomial monad} $\Bin$. Modules over $\Bin$ are given by \textit{binomial rings}: namely, torsion-free commutative rings that are closed under operations given by evaluation of binomial coefficients. We start by giving a careful account for the relation of $\Bin$ to (additively) polynomial maps between abelian groups in \Cref{ssec:polynomial maps}; this allows to define a natural filtered monad $\Bin^{\le *}$ on the category $\Lat$ of finitely generated free abelian groups. In \Cref{ssec:classical binomial algebras} we define the classical binomial monad $\Bin$ on the category $\Ab^{\mr{tf}}$ of torsion free abelian groups and compare modules over it with classical binomial rings, considered e.g. in \cite{Elliot}. In \Cref{ssec:polynomial functors and bla} we recall the necessary results of \cite{Arpon} which  in \Cref{ssec:derived binomial rings} allow us to construct a non-abelian derived extension $\mbf L\Bin$ of $\Bin$ to the whole derived category $\mscr D(\mbb Z)$; here we crucially use that functors $\Bin^{\le n}$ are polynomial. We define \textit{derived binomial rings} as modules over $\mbf L\Bin$ in $\mscr D(\mbb Z)$. Besides that in \Cref{ssec:derived binomial rings} we study some basic properties of derived binomial rings, relation to other derived algebraic contexts (e.g. derived commutative rings, or $\mbb E_\infty$-rings) and finally give some examples, including most importantly the singular cohomology $C^*_\sing(X,\mbb Z)$ of any space $X\in\Spc$ (\Cref{ex:cohomology of topological spaces as a derived binomial algebra}).

\subsection{Polynomial maps}\label{ssec:polynomial maps}

 Let $\Lat\subset \Ab$ be the subcategory of free finitely generated abelian groups (=lattices). Let $L\in \Lat$; then $\Sym_{\mbb Z}(L)$ naturally embeds into the ring $\Map(L^\vee, \mbb Z)$ of \textit{all} $\mbb Z$-valued functions on the dual lattice $L^\vee\coloneqq \Hom_{\mbb Z}(L,\mbb Z)$. However, there are more functions on $L^\vee$ that are polynomial in the following weaker sense.
 
 \begin{defn}[Polynomial maps, \cite{passi2}]\label{def:polynomial maps} Let $A,B$ be abelian groups. By induction on $n$ we can define what it means for a map $f\colon A\ra B$ (as sets) to be \textit{polynomial of degree $\le n$}. Namely
 	\begin{itemize}
 		\item $f$ is polynomial of degree $\le -1$ if $f$ is identically 0;
 		\item $f$ is polynomial of degree $\le n$ if for each $a\in A$ the map $D_af\colon A\ra B$ defined as 
 		$$
 		(D_af)(a')=f(a+a')-f(a')
 		$$
 		is polynomial of degree $\le n-1$.
 		A map $f$ is called \textit{polynomial} if it is polynomial of some degree $\le n$. 
  	\end{itemize}

  We denote $\Map_{\le n}(A,B)$ and $\Map_{\mr{poly}}(A,B)=\colim \Map_{\le n}(A,B)$ the sets of polynomial maps of degree $\le n$ and all polynomial maps from $A$ to $B$ correspondingly.

\end{defn}	

  \begin{warn} To avoid confusion we warn the reader that above notion of polynomial map is different from the notion of \textit{polynomial law} over $\mbb Z$ due to Robi \cite{roby}, which in this case will only give those maps which are represented by polynomials with \textit{integral} coefficients.
\end{warn}

\begin{rem}\label{rem:properties of polynomial functions} The following properties are straightforward to check from the definition:
	\begin{itemize}
		\item $f\colon A\ra B$ belongs to $\Map_{\le n}(A,B)$ if and only if for any $n+1$ elements $a_0,\ldots,a_n\in A$
		$$
		(D_{a_n}\ldots D_{a_0}f)\equiv 0.
	$$ 
		\item If $f\in \Map_{\le n}(A,B)$ and $g\in \Map_{\le m}(A,B)$ then $f+g\in  \Map_{\le \max(n,m)}(A,B)$. In particular each $\Map_{\le n}(A,B)$ is an abelian group.
		\item If $B$ has a ring structure then for $f\in \Map_{\le n}(A,B)$ and $g\in \Map_{\le m}(A,B)$ their product $f\cdot g$ lies in $\Map_{\le n+m}(A,B)$.
		\item If $f\in \Map_{\le n}(A,B)$ and $g\in \Map_{\le m}(B,C)$ then their composition $g\circ f$ lies in $\Map_{\le nm}(A,C)$. This gives a natural map of sets
		\begin{equation}\label{eq:composition of polynomial maps}
			\circ \colon \Map_{\le m}(B,C)\times \Map_{\le n}(A,B) \ra \Map_{\le nm}(A,C).
		\end{equation}	
	\end{itemize}
In particular, if $B=\mbb Z$ with its unique commutative ring structure, we get that $\Map_{\mr{poly}}(A,\mbb Z)$ has a natural structure of commutative ring.
\end{rem}	
  
We will be primarily interested in the case $B=\mbb Z$.
\begin{ex}\label{ex:examples of polynomial maps} Let us explicitly describe $\Map_{\le n}(A,\mbb Z)$ in low degrees:
	\begin{enumerate}
		\item $\Map_{\le 0}(A,B)$ is canonically identified with $B$, via the map $f\mapsto f(0)$. Indeed, the condition $f(a+a')-f(a')=0$ implies that $f(a)=f(0)$ for any $a\in A$. So $f\equiv {f(0)}$ is constant.
		
		\item \label{part:maps of degree less or equal than 1} $\Map_{\le 1}(A,B)$ canonically splits as $B \oplus \Hom_{\mbb Z}(A,B)$ via the map $f\mapsto (f(0), f- {f(0)})$. Indeed, to see that the map is well-defined, note that by part 1 we have that $D_af$ is a constant function for any $a\in A$. We also have $D_af(0)=f(a+0)-f(0)=f(a)-f(0)$. Thus, if we replace $f$ by $f'=f-f(0)$ we will get $D_a(f')$ is a constant function given by $f(a)$. Then we get $f'(a)=(D_af)(a')=f'(a+a')-f'(a')$ for any $a,a'\in A$ and so $f'$ is $\mbb Z$-linear.
	    It is also straightforward to check that $\mbb Z$-linear maps $\Hom_{\mbb Z}(A,B)$ are polynomial of degree less or equal to 1. Indeed, for any $a\in A$ and $\ell \in \Hom_{\mbb Z}(A,B)$ we have that $D_a\ell$ is the constant function given by $\ell (a)$.
	    
	    \item \label{part:from sym to polynomial}In the case $B=\mbb Z$ the embedding $A^\vee\coloneqq \Hom_{\mbb Z}(A,\mbb Z)\subset \Map_{\le 1}(A,\mbb Z)$ extends to an algebra homomorphism 
	    $$
	    \Sym_{\mbb Z} (A^\vee) \ra \Map_{\mr{poly}}(A,\mbb Z)
	    $$
	    that sends a polynomial in $A^\vee$ to the corresponding function on $A$ that it defines. Note that this maps restricts to $\Sym_{\mbb Z}^{\le n} (A^\vee) \ra \Map_{\le n}(A,\mbb Z)$ for every $n\ge 0$. Similarly, if $B=\mbb Q$ we get a map 
	    $$
	    \Sym_{\mbb Z}(A^\vee)\otimes \mbb Q\simeq \Sym_{\mbb Q}(\Hom_{\mbb Z}(A,\mbb Q)) \ra \Map_{\mr{poly}}(A,\mbb Q).
	    $$
	\end{enumerate}
\end{ex}

Recall that the category of affine group schemes over a given ring $R$ is antiequivalent to the category of commutative Hopf $R$-algebras (\cite[Section 2.3]{Jantzen_Representations}). Under this antiequivalence commutative group schemes correspond to those commutative Hopf algebras that are also cocommutative.

\begin{construction}[{Hopf algebra $\mbb Z[A]$}]\label{constr:Hopf algebra structure on Z[A]} Consider the functor $\mbb Z[-]\colon \Set\ra \Ab$ sending a set $I$ to the free abelian group $\mbb Z[I]$. Functor $\mbb Z[-]$ has a natural symmetric structure: one can naturally identify $\mbb Z[S_1\times S_2]$ with $\mbb Z[S_1]\otimes \mbb Z[S_2]$ via the unique $\mbb Z$-linear map that sends $[(s_1,s_2)]$ to $[s_1]\otimes [s_2]$.
	
	Any abelian group $A\in \Ab$ gives a Hopf algebra in $\Set$: the addition map $+\colon A\times A\ra A$ gives multiplication, $0\ra A$ the unit map, the diagonal map $\Delta\colon A\ra A\times A$ gives comultiplication with counit $A\ra 0$ and, finally, the antipode $A\ra A$ is given by multiplication by $-1$. This endows $\mbb Z[A]$ with the natural Hopf algebra structure. More explicitly, the induced multiplication is defined by $[a_1]\cdot [a_2]=[a_1+a_2]$, and the comultiplication is given by $\Delta([a])=[a]\otimes [a]$; we also have $1=[0]$. The augmentation ideal\footnote{By definition this is the kernel of the counit map $\mbb Z[A]\ra \mbb Z\simeq \mbb Z[0]$.} $I_A\subset \mbb Z[A]$ is generated by elements $([a]-1)_{a\in A}$. The group $I_A/I_A^2$ is naturally identified with $A$ via the map $a\mapsto [a]-1$.
	
	Note that in the case $A\in \Lat$ is free and finitely generated, the corresponding affine scheme $\Spec \mbb Z[A]$ is a torus, more precisely $\mbb G_m\otimes_{\mbb Z} A^\vee$. Since the latter is smooth over $\mbb Z$ we can identify $I_A^n/I_A^{n+1}$ with $\Sym^n_{\mbb Z}(I_A/I_A^2)\simeq \Sym_{\mbb Z}^n(A)$. 
	
	As usual, the $I_A$-adic completion $\mbb Z[A]^\wedge_{I_A}$ defines a topological Hopf algebra; the corresponding formal group scheme is identified with the formal completion $\widehat{\mbb G}_m\otimes_{\mbb Z} A^\vee$. The completed associated graded $\prod_{n=0}^\infty I_A^n/I_A^{n+1}$ with respect to $I_A$-adic filtration is identified with $\mbb Z[[A]]$; moreover, since for any Hopf algebra for $x\in I_A$ one has $\Delta(x)=x\otimes 1 + 1\otimes x \pmod{I_A\otimes I_A}$, one has  $\Delta\colon a\mapsto [a]\otimes 1 + 1\otimes [a]\in \Spf[[A]]\cotimes \Spf[[A]]$. In other words, the corresponding associated graded formal group scheme  is identified with the additive group $\widehat{\mbb G}_a \otimes_{\mbb Z} A^\vee$.
\end{construction}	 

	Following \cite{passi} one can naturally describe $\Map_{\mr{poly}}(A,\mbb Z)$ in terms of of the group algebra $\mbb Z[A]$ (or rather its topological dual):
\begin{construction}[{Description of $\Map_{\mr{poly}}(A,B)$ via the group algebra $\mbb Z[A]$}]\label{rem:description of polynomial maps}
By the universal property of $\mbb Z[A]$ we can extend any set-theoretic map $f\colon A\ra B$  to an abelian group homomorphism $\tilde f\colon \mbb Z[A] \ra B$. We claim that $f$ is polynomial of degree $\le n$ if and only if $f$ vanishes on $I_A^{n+1}$. Indeed, for any $x\in A$ one has 
$$
(D_{a_n}\ldots D_{a_0}f)(x)= \tilde f(([a_0]-1)([a_1]-1)\ldots([a_n]-1)[x])
$$ and so the condition $(D_{a_n}\ldots D_{a_0}f)=0$ is equivalent to $\tilde f(([a_0]-1)([a_1]-1)\ldots([a_n]-1)x)=0$ for all $x\in A$. This gives a description of $\Map_{\le n}(A,B)$ as $\Hom_{\mbb Z}(\mbb Z[A]/I_A^{n+1},B)$ and $\Map_{\mr{poly}}(A,B)\simeq \colim_n \Map_{\le n}(A,B)$ as 
$$
	\Map_{\mr{poly}}(A,B)\simeq \colim_n \Hom_{\mbb Z}(\mbb Z[A]/I_A^{n+1},B)
$$
In the case $B=\mbb Z$ this gives natural identifications
$\Map_{\le n}(A,\mbb Z)\simeq (\mbb Z[A]/I_A^{n+1})^\vee$ and 
\begin{equation}\label{eq:formula for polynomial maps through group algebra}
	\Map_{\mr{poly}}(A,\mbb Z)\simeq \colim_n \ (\mbb Z[A]/I_A^{n+1})^\vee.
\end{equation}	
	In other words, $\Map_{\mr{poly}}(A,\mbb Z)$ is naturally identified with distributions\footnote{See for example \cite[Section 7]{Jantzen_Representations}.} at 0 of $\Spec \mbb Z[A]$. In the case $A\in \Lat$, the algebra $\mbb Z[A]$ is smooth over $\mbb Z$ and the Hopf algebra structure on $\mbb Z[A]$ induces a Hopf algebra structure on $\colim_n (\mbb Z[A]/I_A^{n+1})^\vee$ (see \cite[Section 7.7]{Jantzen_Representations}).  The latter, by definition, is the Hopf algebra of distributions $\mr{Dist}(\mbb G_m\otimes_{\mbb Z}A^\vee)$ on the group scheme $\Spec \mbb Z[A]$ (which by \Cref{constr:Hopf algebra structure on Z[A]} is isomorphic to $\mbb G_m\otimes_{\mbb Z}A^\vee$). Equivalently, affine group scheme $\Spec (\Map_{\mr{poly}}(A,\mbb Z))$ is the Cartier dual of the formal group scheme $\widehat{\mbb G}_m\otimes_{\mbb Z} A^\vee$.
\end{construction}

\begin{rem}\label{rem:induced splitting in low degree}
	Note that unit and counit maps give a splitting $\mbb Z[A]\simeq \mbb Z \oplus I_A$, functorially in $A$. Under the isomorphism $A\ra I_A/I_A^2$ sending $a$ to $[a]-1$ this induces a splitting of $\mbb Z[A]/I_A^2$ as $\mbb Z\oplus A$. Then the resulting identification 
	$$
	\Map_{\le 1}(A,B)\simeq \Hom_{\mbb Z}(\mbb Z[A]/I_A^2, B)\simeq B\oplus \Hom_{\mbb Z}(A,B)
	$$
	is exactly the one we constructed in \Cref{ex:examples of polynomial maps}(\ref{part:maps of degree less or equal than 1}); indeed, a map $f$ is being send to a pair $(f(0),f-f(0))$, where  $f-f(0)\in \Hom_{\mbb Z}(A,B)$.
\end{rem}	

\begin{rem}[$\Map_{\mr{poly}}(A,\mbb Z)$ is a symmetric monoidal functor]\label{rem:polynomial maps are symmetric monoidal}
	Let $A\in \Lat$. Since the group scheme $\mbb G_m\otimes_{\mbb Z}A$  is smooth, and $\mbb G_m\otimes_{\mbb Z}(A_1\oplus A_2)\simeq (\mbb G_m\otimes_{\mbb Z} A_1) \times (\mbb G_m\otimes_{\mbb Z} A_2)$, one has (see e.g. \cite[Section 7.4,(2)]{Jantzen_Representations}) $$\mr{Dist}(\mbb G_m\otimes_{\mbb Z}(A_1\oplus A_2))\simeq \mr{Dist}(\mbb G_m\otimes_{\mbb Z}A_1) \otimes \mr{Dist}(\mbb G_m\otimes_{\mbb Z} A_2).$$
	Applying the above to $A_1^\vee,A_2^\vee\in \Lat$ in place of $A_1,A_2$ we get that the natural map 
	$$
	\Map_{\mr{poly}}(A_1,\mbb Z)\otimes \Map_{\mr{poly}}(A_2,\mbb Z) \ra \Map_{\mr{poly}}(A_1\oplus A_2,\mbb Z)
	$$
	induced by multiplication of pull-backs under projections $A_1\leftarrow A_1\oplus A_2 \rightarrow A_2$ is an isomorphism. This shows that the functor $\Map_{\mr{poly}}(-,\mbb Z)\colon \Lat^\op \ra \Ab$ is symmetric monoidal, and one can identify the Hopf algebra structure on $\Map_{\mr{poly}}(A,\mbb Z)$ for $A\in \Lat$ with the one induced by the canonical Hopf algebra structure on $A$ as an object of $\Lat^\op$ via the monoidal structure given by $\oplus$ (where multiplication is given by the (opposite to) the diagonal $\Delta\colon A\ra A\oplus A$, the comultiplication by the addition $A\oplus A\ra A$, et.c.).
\end{rem}	

\begin{rem}[Associated graded of $\Map_{\le *}(A,\mbb Z)$]\label{rem:associated graded for polynomial maps}
	Let $A\in \Lat$. The Hopf algebra $\Map_{\mr{poly}}(A,\mbb Z)$ has a filtration given by $\Map_{\le *}(A,\mbb Z)$. Identifying $\Map_{\le n}(A,\mbb Z)$ with $(\mbb Z[A]/I_A^{n+1})^\vee$ we see that the associated graded Hopf algebra $\gr_* \Map_{\le *}(A,\mbb Z)$ is Cartier dual to the completed associated graded $\prod_{n=0}^\infty I_A^n/I_A^{n+1}$. By \Cref{constr:Hopf algebra structure on Z[A]} the corresponding formal group scheme is identified with $\widehat{\mbb G}_a\otimes_{\mbb Z}A^\vee$, and the Cartier dual is given by $\Spec(\mr{Dist}(\mbb G_a \otimes_{\mbb Z}A^\vee))\simeq \Spec(\mr{Dist}(\mbb G_a))\otimes_{\mbb Z} A$. The group scheme $\Spec(\mr{Dist}(\mbb G_a))$ is often denoted $\mbb G_a^\sharp$. The Hopf algebra $\mc O(\mbb G_a^\sharp)$ is given by the free divided power algebra $\Gamma^*_{\mathbb Z}(\mbb Z)=\mbb Z[\frac{x^n}{n!}]_{n\ge 1}$, and more generally there is a natural isomorphism $\mc O(\mbb G_a^\sharp\otimes A)\simeq \Gamma_{\mathbb Z}^*(A^\vee)$. 
	
	Summarizing, we get an isomorphism of Hopf algebras 
	$$
	\gr_*\Map_{\le *}(A,\mbb Z)\simeq \Gamma_{\mathbb Z}^*(A^\vee).
	$$
\end{rem}

\begin{rem}[Map $\Sym_{\mbb Z}(A^\vee)\ra \Map_{\mr{poly}}(A,\mbb Z)$]\label{rem:map from Sym to polynomial maps} By \Cref{ex:examples of polynomial maps}(\ref{part:from sym to polynomial}) we have a natural filtered map $\Sym^{\le *}_{\mbb Z}(A^\vee) \ra \Map_{\le *}(A,\mbb Z)$. Since it is functorial in $A$, it is a Hopf algebra homomorphism (indeed, both Hopf algebra structures is induced by the natural Hopf algebra structure on $A^\vee$ as an object of $\Ab^\op$ with the monoidal structure given by $\oplus$). By passing to the associated graded we get a Hopf algebra map 
	$$
	\Sym^*_{\mbb Z}(A^\vee) \ra \Gamma^*_{\mbb Z}(A^\vee).
	$$
	By the universal property of $\Sym^*_{\mbb Z}$ the above map is uniquely extended from its restriction to $A^\vee\simeq \Sym^1_{\mbb Z} (A^\vee)$, which maps isomorphically to $A^\vee\simeq \Gamma^1_{\mbb Z}(A^\vee)$. In degree $n$, the resulting map
	$
	\Sym^n_{\mbb Z}(A^\vee) \ra \Gamma^n_{\mbb Z}(A^\vee)
	$ is identified with the \textit{norm map} $\Sym^n_{\mbb Z}(A^\vee)\simeq (A^\vee \otimes \ldots \otimes A^\vee)_{\Sigma_n}\ra (A^\vee \otimes \ldots \otimes A^\vee)^{\Sigma_n}\simeq \Gamma^n_{\mbb Z}(A^\vee)$ from $\Sigma_n$-coinvariants to $\Sigma_n$-invariants: namely one lifts an element $x\in ((A^\vee)^{\otimes n})_{\Sigma_n}$ to an element $\widetilde x\in (A^\vee)^{\otimes n}$ and maps it to the average $\sum_{\sigma\in \Sigma_n} \sigma(\widetilde x)\in ((A^\vee)^{\otimes n})^{\Sigma_n}$. In the case $A^\vee\simeq \mbb Z$ the above map is given by the natural embedding $\mbb Z[x]\ra \mbb Z[\frac{x^n}{n!}]_{n\ge 1}$.
	
\end{rem}	

\begin{lem}\label{lem:over Q polynomial functions are actually polynomials}
	Let $A\in \Lat$. The natural map 
	$$
	 \Sym_{\mbb Z}(A^\vee)\otimes \mbb Q\simeq \Sym_{\mbb Q}(\Hom_{\mbb Z}(A,\mbb Q)) \ra  \Map_{\mr{poly}}(A,\mbb Q)
	$$
	from \Cref{ex:examples of polynomial maps}(\ref{part:from sym to polynomial}) is an isomorphism. In other words, any $\mbb Q$-valued polynomial function in the sense of \Cref{def:polynomial maps} is a polynomial in linear functions.
\end{lem}	

\begin{proof}
	By construction, the above map is filtered: $\Sym_{\mbb Q}^{\le n}(\Hom_{\mbb Z}(A,\mbb Q))$ lands in $\Map_{\le n}(A,\mbb Q)$. Both filtrations are exhaustive, so it would be enough to show that the associated graded map is an isomorphism. Note that $$\Map_{\le n}(A,\mbb Q)\simeq \Hom_{\mbb Z}(\mbb Z[A]/I_A^{n+1},\mbb Q)\simeq (\mbb Z[A]/I_A^{n+1})^\vee\otimes \mbb Q \simeq \Map_{\le n}(A,\mbb Z)\otimes \mbb Q,$$
	since $\mbb Z[A]/I_A^{n+1}$ is a finitely generated abelian group. By taking colimit we also get an identification $\Map_{\mr{poly}}(A,\mbb Q)\simeq \Map_{\mr{poly}}(A,\mbb Z)\otimes_{\mbb Z} \mbb Q$. By similar reasoning $\Hom_{\mbb Z}(A,\mbb Q)\simeq A^\vee\otimes \mbb Q$, and so the map in question is identified with the map $\Sym_{\mbb Z}(A^\vee)\ra \Map_{\mr{poly}}(A,\mbb Z)$ tensor $\mbb Q$. The map on the associated graded is then identified with 
	$$
	\Sym^*_{\mbb Z}(A^\vee)\otimes\mbb Q \ra \Gamma^*_{\mbb Z}(A^\vee)\otimes {\mbb Q},
	$$
	which is an isomorphism (in each degree $n$ the inverse to the norm map $((A^\vee)^{\otimes n})_{\Sigma_n}\ra ((A^\vee)^{\otimes n})^{\Sigma_n}$ is given by the natural projection $((A^\vee)^{\otimes n})^{\Sigma_n}\ra ((A^\vee)^{\otimes n})_{\Sigma_n}$ divided by $n!=|\Sigma_n|$).
\end{proof}	

\begin{rem}[$\Map_{\mr{poly}}(A,\mbb Z)$ as integer-valued polynomials]\label{rem:polynomial maps as integer-valued polynomials}
	By \Cref{lem:over Q polynomial functions are actually polynomials} we get that any polynomial $\mbb Q$-valued function on $A$ is a polynomial in $\Hom_{\mbb Z}(A,\mbb Q)$ (=linear functions) with rational coefficients. This way, $\Map_{\mr{poly}}(A,\mbb Z)\subset \Map_{\mr{poly}}(A,\mbb Q)$ is identified with the subring of those $\mbb Q$-polynomials in $\Hom_{\mbb Z}(A,\mbb Q)$ that take integer values on $A$. 
	
	Let us specialize to the case $A=\mbb Z$. The explicit description of integer-valued rational polynomials in one variable has been established more than a century ago by Polya in \cite{Polya}. Namely, let $x\in \Hom_{\mbb Z}(\mbb Z,\mbb Z)\simeq \mbb Z^\vee$ be the identity map. Then $\Map_{\mr{poly}}(\mbb Z,\mbb Z)$ is identified with the subring $\Sym_{\mbb Z}(\mbb Z^\vee)\otimes \mbb Q\simeq \mbb Q[x]$ generated over $\mbb Z$ by "binomial coefficients" $${x\choose n}\coloneqq \frac{x(x-1)\ldots (x-n+1)}{n!}\in \mbb Q[x].$$
	The corresponding function $\mbb Z\ra \mbb Z$ that it defines is given by the evaluation of the above binomial coefficient.
	
	More generally, $\Map_{\mr{poly}}(A,\mbb Z)\subset \Map_{\mr{poly}}(A,\mbb Q)\simeq \Sym_{\mbb Z}(A^\vee)\otimes \mbb Q$ is identified with the $\mbb Z$-subalgebra generated by the binomial coefficients ${[\ell]\choose n}\in \Sym_{\mbb Z}(A^\vee)\otimes \mbb Q$ where $[\ell]\in \Sym^1_{\mbb Z}(A^\vee)$ is the element which corresponds to a linear function $\ell\in A^\vee$ via the isomorphism $A^\vee\simeq \Sym^1_{\mbb Z}(A^\vee)$. 
\end{rem}	

\begin{rem}[Comparison with $\mr{Dist}(\mbb G_m)$]
Let us compare the identification $\Map_{\mr{poly}}(\mbb Z,\mbb Z)\simeq \mbb Z[\binom{x}{n}]_{n\ge 1}$ with description as algebra of distributions on $\mbb G_m$. Let us identify $\mbb Z[\mbb Z]\simeq \mbb Z[t,t^{-1}]$, where $[1]\mapsto t$; this way the augmentation ideal is identified with $(t-1)\subset \mbb Z[t,t^{-1}]$. Then, following \Cref{rem:induced splitting in low degree} and splitting $\mbb Z[t,t^{-1}]/(t-1)^2$ as $\mbb Z\cdot 1\oplus \mbb Z\cdot (t-1)$, we see that the identity function $x\in \Hom_{\mbb Z}(\mbb Z,\mbb Z)$ corresponds to a functional in $(\mbb Z[t,t^{-1}]/(t-1)^2)^\vee$ that is 0 on 1 and is equal to 1 on $t-1$. Then (see e.g. \cite[Section 7.8]{Jantzen_Representations}, there $x$ is denoted by $\delta_1$) one can identify $\binom{x}{n}$ with the functional in $(\mbb Z[t,t^{-1}]/(t-1)^{n+1})^\vee$ that is zero on $1,t-1,\ldots, (t-1)^{n-1}$ and is 1 on $(t-1)^n$. So, one can think of $\{\binom{x}{n}\}_{n\ge 0}$  (here we put $\binom{x}{0}\coloneqq 1$) as a dual basis to $\{(t-1)^n\}_{n\ge 0}$ (the latter of course is not a basis of $\mbb Z[t,t^{-1}]$, but is a topological basis of the completion ${\mbb Z[t,t^{-1}]}^\wedge_{(t-1)}$).
\end{rem}

\begin{rem}[{Hopf algebra structure on $\mbb Z[\binom{x}{n}]_{n\ge 1}$}]\label{rem:explicit Hopf algebra structure} Let us describe the Hopf algebra structure on $\Map_{\mr{poly}}(\mbb Z,\mbb Z)$ in terms of identification with $\mbb Z[\binom{x}{n}]_{n\ge 1}$. The comultiplication is induced by the addition map $\mbb Z\oplus \mbb Z\ra \mbb Z$, which dually corresponds to the diagonal embedding $\mbb Z^{\vee}\ra \mbb Z^{\vee}\oplus \mbb Z^\vee.$ Thus $x=\id_{\mbb Z}\in \mbb Z^\vee$ is being sent to $(x,x)$. This means $\Delta\colon \mbb Z[\tbinom{x}{n}]_{n\ge 1} \ra \mbb Z[\tbinom{x}{n}]_{n\ge 1}\otimes \mbb Z[\tbinom{x}{n}]_{n\ge 1}$ sends $x$ to $x\otimes 1 + 1\otimes x$; it then extends uniquely to $\mbb Z[\tbinom{x}{n}]_{n\ge 1}$ as: 
	$$
	\Delta\left(\binom{x}{n}\right )= \sum_{i+j=n, i,j\ge 0} \binom{x}{i}\otimes \binom{x}{j}
	$$
	The natural map $\Sym_{\mbb Z}((\mbb Z)^\vee)\ra \Map_{\mr{poly}}(\mbb Z,\mbb Z)$ from \Cref{ex:examples of polynomial maps}(3) is a map of Hopf algebras that is given explicitly as the embedding
	$$
	\mbb Z[x] \ra \mbb Z[\tbinom{x}{n}]_{n\ge 1}.
	$$
	\end{rem}
\begin{notation}
		Let $\mbb H\coloneqq \Spec (\Map_{\mr{poly}}(\mbb Z,\mbb Z)) \simeq \mr{Dist}(\mbb G_m)$; $\mbb H$ is called the "Hilbert additive group scheme" in \cite{Toen_geo}. Note that by the discussion in \Cref{rem:description of polynomial maps} for any $A\in \Lat$ one has a natural identification $\Spec (\Map_{\mr{poly}}(A,\mbb Z)) \simeq  \mr{Dist}(\mbb G_m\otimes_{\mbb Z} A)\simeq \mbb H\otimes_{\mbb Z}A^\vee$.
	\end{notation}		
	\begin{rem}\label{rem:the map G_a --> H modulo p}
	The homomorphism $\Sym_{\mbb Z}((\mbb Z)^\vee)\ra \Map_{\mr{poly}}(\mbb Z,\mbb Z)$ induces a map of affine group schemes $\mbb H\ra \mbb G_a$ over $\Spec \mbb Z$. By (the proof of) \Cref{lem:over Q polynomial functions are actually polynomials} we get that it is an isomorphism over the generic fiber $\Spec \mbb Q\subset \Spec \mbb Z$. However, modulo each prime $p$ this map is very far from being one. Indeed, over $\mbb F_p$ one has an identification $\widehat {\mbb G}_m\simeq \colim_n \mu_{p^n}$, so $\mbb H_{\mbb F_p}\simeq \ul{\mbb Z_p}\coloneqq \lim_n \ul{\mbb Z/p^n}$, and the map $\mbb H_{\mbb F_p}\ra \mbb G_{a,\mbb F_p}$ is given by the composition 
	$$
	\ul{\mbb Z_p} \surj \ul{\mbb Z/p} \ra \mbb G_{a,\mbb F_p},
	$$
	where the first map is the natural projection, and the second is the embedding of the constant group scheme given by $\mbb F_p$-points of $\mbb G_{a,\mbb F_p}$.
\end{rem}	

Let us finally record the last piece of data that we will need to define the binomial monad:
\begin{construction}[Canonical map  $\ev_A\colon A\mapsto \Map_{\le n}(A,\mbb Z)^\vee$]
For any $a\in A$ the association $f\mapsto f(a)$ defines a linear function on $\Map_{\le n}(A,\mbb Z)$. We claim that the corresponding map of sets $A\ra \Map_{\le n}(A,\mbb Z)^\vee$ is in fact polynomial of degree $\le n$. Indeed, via the identification  $\Map_{\le n}(A,\mbb Z)\simeq (\mbb Z[A]/I_A^{n+1})^\vee$ in \Cref{rem:description of polynomial maps} the $\mbb Z$-linearization $\mbb Z[A]\ra \Map_{\le n}(A,\mbb Z)^\vee$ factors through the natural linear evaluation map $\ev\colon \mbb Z[A]/I_A^{n+1}\ra ((\mbb Z[A]/I_A^{n+1})^\vee)^\vee$.

	 This gives a canonical class $\ev_A\in \Map_{\le n}(A,\Map_{\le n}(A,\mbb Z)^\vee)$. By \Cref{rem:properties of polynomial functions}(3) there is a natural composition map 
	 $$
	 \Map_{\le m}(\Map_{\le n}(A,\mbb Z)^\vee,\mbb Z) \times  \Map_{\le n}(A, \Map_{\le n}(A,\mbb Z)^\vee)\ra \Map_{\le nm}(A,\mbb Z),
	 $$ 
	 from which, plugging $\ev_A$ in the second variable we get natural in $A$ maps 
	$$
	 \Map_{\le m}(\Map_{\le n}(A,\mbb Z)^\vee,\mbb Z) \ra \Map_{\le nm}(A,\mbb Z).
	 $$
 Note that if we plug $A^\vee$ the left and right hand sides become more symmetric:
	 \begin{equation}\label{eq:monad structure}
 \Map_{\le m}(\Map_{\le n}(A^\vee,\mbb Z)^\vee,\mbb Z) \ra \Map_{\le nm}(A^\vee,\mbb Z).
	 \end{equation}

\end{construction}	

\subsection{Classical binomial monad}\label{ssec:classical binomial algebras}

We now can define free classical binomial rings.

\begin{defn}\label{def:free binomial algebra}
	Let $A\in \Lat$. We define the \textit{free binomial ring} $\Bin(A)$ to be $\Map_{\mr{poly}}(A^\vee,\mbb Z)$. The filtration of $\Bin(A)$ by $\Bin^{\le n}(A)\coloneqq \Map_{\le n}(A^\vee,\mbb Z)$ endows it with the filtered algebra structure. 
\end{defn}

We will denote by $\Bin(A)$ both the corresponding algebra and the underlying abelian group, hopefully not leading to any confusion.
\begin{rem}\label{rem:properties of Bin} Let $A\in \Lat$. Let us resummarize some properties of polynomial maps that we have established in previous section in terms of $\Bin$:
	
	\begin{enumerate}\item The filtration on $\Bin(A)$ given by $\Bin^{\le *}(A)$ has associated graded $\Gamma^*_{\mbb Z}(A)$, functorially in $A$ (\Cref{rem:associated graded for polynomial maps}). 
		
		\item \label{part:splitting off Bin0} One has $\Bin(0)\simeq \mbb Z$. Maps $0\ra A\ra 0$ split $\Bin(A)$ as $\mbb Z\oplus \ol{\Bin}(A)$ (the latter is the definition of $\ol{\Bin}(A)$). In particular, $\Bin^{\le 1}(A)$ splits as $\mbb Z\oplus A$ (see \Cref{rem:induced splitting in low degree}).
		
		\item There is a natural in $A$ filtered ring homomorphism $\Sym_{\mbb Z}^{\le *}(A) \ra \Bin^{\le *}(A)$ (see \Cref{ex:examples of polynomial maps}(3)). The induced map 
		$
		\Sym_{\mbb Z}^{i}(A) \ra \Gamma_{\mbb Z}^{i}(A) 
		$
		on the $i$-th graded piece is the norm map $(A^{\otimes i})_{\Sigma_i}\ra (A^{\otimes i})^{\Sigma_i}$ (see \Cref{rem:map from Sym to polynomial maps}). 
		
		\item \label{part:explicit descr of Bin A} $\Bin(A)$ naturally embeds into $\Sym_{\mbb Z}(A)\otimes \mbb Q$, and the image can be identified with the $\mbb Z$-subalgebra of $\Sym_{\mbb Z}(A)\otimes \mbb Q$ generated by binomial coefficients 
		$$
		\binom{[a]}{n}=\frac{[a]([a]-1)\ldots([a]-n+1)}{n!}\in \Sym_{\mbb Z}(A)\otimes \mbb Q,
		$$
		where $[a]$ is the image of $a\in A$ under the identification $A\simeq \Sym^1_{\mbb Z}(A)$ (see \Cref{rem:polynomial maps as integer-valued polynomials}).
		
		\item \label{part:monodicity}Functor $\Bin$ is symmetric monoidal: the natural filtered ring homomorphism $$\Bin(A_1)\otimes\Bin(A_2)\ra \Bin(A_1\oplus A_2)$$ is an isomorphism (see \Cref{rem:polynomial maps are symmetric monoidal}).
	\end{enumerate}	

\end{rem}

Unfortunately, functor $A\mapsto \Bin(A)$ does not preserve $\Lat$, so there is no way it would define a monad on it. Nevertheless, if $A\in \Lat$ then $\Bin^{\le n}(A)$ is still a free finitely generated abelian group, and \Cref{eq:monad structure} provides us with the natural in $A$ composition maps
$$
\Bin^{\le m}(\Bin^{\le n}(A)) \ra \Bin^{\le nm}(A).
$$
Such data gives rise to a \textit{filtered monad} in the sense of \cite{Arpon}.

Let $\mathbf{\mbb Z}^\times_{\ge 0}$ denote the symmetric monoidal category corresponding to the poset of non-negative integers endowed with the monoidal structure given by multiplication. Given an $\infty$-category $\mathscr C$ we endow the category of endofunctors $\End(\mathscr C)$ with the monoidal structure given by composition (see \cite[Section 4.7]{Lur_HA}).
\begin{defn}[{\cite[Definition 4.1.2]{Arpon}}]
	A \textit{filtered monad} $\mr T^{\le *}$ on an $\infty$-category $\mathscr C$ is the data of a lax-monoidal functor 
	$$
	\mr T^{\le *}\colon \mbb Z^\times_{\ge 0} \ra \End(\mathscr C).
	$$
	We denote by $\mr T^{\le n}\colon \mathscr C\ra \mathscr C$ the endofunctor corresponding to $n\in \mbb Z^\times_{\ge 0}$.
\end{defn}

Note that a lax-monoidal functor between monoidal categories in the classical sense gives rise to a lax-monoidal functor in the $\infty$-categorical setting between the corresponding monoidal $\infty$-categories.

\begin{ex}\label{ex:filtered monads} The following give natural examples of filtered monads on $\Lat$:
	\begin{enumerate}
		\item The natural transformations 
		$$
		\Bin^{\le n}\circ \Bin^{\le m} \ra \Bin^{\le nm}
		$$
	given by \Cref{eq:monad structure} for any $n,m\in \mbb Z^\times_{\ge 0}$ and $A\in \Lat$ endow the association $\Bin^{\le *}\colon n\ra \Bin^{\le n}\in \End(\Lat)$ with the structure of a filtered monad.
	\item \label{part:filtered monad Sym} The natural transformation $\Sym_{\mbb Z}\circ \Sym_{\mbb Z}\ra \Sym_{\mbb Z}$ given by monad structure on $\Sym_{\mbb Z}$ restricts to transformations
	$$
	\Sym^{\le n }\circ \Sym^{\le m}\ra  \Sym^{\le nm}
	$$
	that endow the association $\Sym^{\le *}\colon n\ra \Sym^{\le n}\in \End(\Lat)$ with the structure of a filtered monad.
	\item \label{part:morphism of filtered monads} The natural transformation $\Sym^{\le n}(A)\ra \Bin^{\le n}(A)$ (given by \Cref{ex:examples of polynomial maps}(\ref{part:from sym to polynomial}) for $A^\vee$) agrees with the composition and defines a map of filtered monads
	$$
	\Sym^{\le *} \ra \Bin^{\le *}.
	$$
	\end{enumerate}
\end{ex}

We will now show how filtered monad $\Bin^{\le *}$ defines a classical binomial monad $\Bin$ (considered, for example, in \cite{Elliot}) on the category $\Ab^{\mr{tf}}$ of torsion-free abelian groups. Note that one can identify the category $\Ab^{\mr{tf}}$ of torsion-free abelian groups with the ind-completion $\Ind(\Lat)$ of $\Lat$.

\begin{construction}[Binomial monad on $\Ab^{\mr{tf}}$]\label{constr:classical binomial monad} By the universal property of $\Ind(\Lat)$, for any $\infty$-category $\mathscr D$ that has filtered colimits, the restriction to $\Lat\subset \Ind(\Lat)\simeq \Ab^{\mr{tf}}$ induces an equivalence of $\infty$-categories
\begin{equation}\label{eq:universal property of Ind-completion}
\Fun_{\omega}(\Ab^{\mr{tf}}, \mathscr D) \simeq \Fun\! \left(\Lat, \mathscr D\right),
\end{equation}
where $\Fun_{\omega}$ denotes the category of filtered colimit preserving functors. The inverse functor of this equivalence is given by left Kan extension along $\Lat\subset \Ab^{\mr{tf}}$. In particular, given $F\in \Fun\! \left(\Lat, \mathscr D\right)$, the value of the corresponding functor $F\in \Fun_{\omega}(\Ab^{\mr{tf}}, \mathscr D)$ on a torsion-free abelian group $M$ can be computed as $\colim_i F(M_i)$ where $M=\colim M_i$ is any presentation of $M$ as a filtered colimit of $M_i\in \Lat$. By plugging $\mathscr D\coloneqq \Ind(\Lat)$, the equivalence (\ref{eq:universal property of Ind-completion}) restricts to a monoidal equivalence
 $$
\End_{\omega}^{\Lat}(\Ab^{\mr{tf}}) \simeq \End(\Lat),
 $$ 
 where $\End_{\omega}^{\Lat}$ denotes the category of filtered colimit preserving endofunctors that also preserve the subcategory $\Lat\subset \Ab^{\mr{tf}}$.
 In particular, a filtered monad $\Bin^{\le *}\colon \mbb Z^\times_{\ge 0}\ra \End(\Lat)$ gives rise to a filtered monad on $\Ab^{\mr{tf}}$, which we will continue to denote the same way. The corresponding endofunctors $\Bin^{\le n}\colon \Ab^{\mr{tf}}\ra \Ab^{\mr{tf}}$ commute with filtered colimits and this way $$\Bin\coloneqq \colim_{n\in \mbb Z^\times_{\ge 0}} \Bin^{\le n}\colon \Ab^{\mr{tf}}\ra\Ab^{\mr{tf}}$$ has a natural monad structure. Namely, the multiplication $\Bin\circ \Bin\ra \Bin$ is induced by the transformation
 $$
\colim_{n\in \mbb Z^\times_{\ge 0}} \Bin^{\le n}\circ \colim_{m\in \mbb Z^\times_{\ge 0}} \Bin^{\le m} \ra \colim_{m,n\in \mbb Z^\times_{\ge 0}} \Bin^{\le nm} \simeq \colim_{s\in  \mbb Z^\times_{\ge 0}} \Bin^{\le s}.
 $$
\end{construction}
\begin{rem}\label{rem:}
	Since for any $A\in \Lat$ the value of $\Bin(A)$ (defined as in \Cref{constr:classical binomial monad}) is obtained as a colimit of $\Bin^{\le n}(A)$ over $n$, one sees that it coincides with (the underlying abelian group) of  $\Bin(A)$ defined as in \Cref{def:free binomial algebra}. For general $A\in \Ab^{\mr{tf}}$ one can compute $\Bin(A)$ as $\colim_i \Bin(A_i)$ of all submodules $A_i\subset A$ such that $A_i\in \Lat$.
\end{rem}

\begin{rem}\label{rem:map from Sym to Bin}
	Note that $\Sym_{\mbb Z}\colon \Ab\ra \Ab$ restricts to a monad on $\Ab^{\mr{tf}}$ and that it is filtered colimit commuting: indeed, $\Sym_{\mbb Z}\simeq \oplus_{i=0}^\infty \Sym^i_{\mbb Z}$, where $\Sym^i_{\mbb Z}(A)=(A^{\otimes i})_{\Sigma_i}$ and both tensor power and coinvariants commute with filtered colimits. Thus $\Sym_{\mbb Z}=\colim_i \Sym_{\mbb Z}^{\le i}\colon \Ab^{\mr{tf}}\ra \Ab^{\mr{tf}}$ is obtained as left Kan extension from $\Lat$ and is realized as colimit of the filtered monad $\Sym^{\le *}_{\mbb Z}$ (\Cref{ex:filtered monads}(\ref{part:filtered monad Sym})) on $\Lat$. In particular the transformation $\Sym_{\mbb Z}^{\le *}\ra \Bin^{\le *}$ from \Cref{ex:filtered monads}(\ref{part:morphism of filtered monads}) induces a map
	$
	\Sym_{\mbb Z}\ra \Bin_{\mbb Z}
	$
	 of monads on $\Ab^{\mr{tf}}$.
\end{rem}	

\begin{rem}\label{rem:associated graded for Bin}
	Functor $\Gamma_{\mbb Z}\coloneqq \oplus_i \Gamma_{\mbb Z}^i \colon \Ab \ra \Ab$ with $\Gamma_{\mbb Z}^i(A)=(A^{\otimes i})^{\Sigma_i}$ also preserves the subcategory $\Ab^{\mr{tf}}$ and commutes with filtered colimits: indeed, $\Sigma_i$-invariants are computed by finite limit, which commutes with filtered colimits. Thus the restriction of $\Gamma_{\mbb Z}$ to $\Ab^{\mr{tf}}$ is also left Kan extended from $\Lat$. 
	
	Functor $\Bin$ comes with a natural exhaustive filtration $\Bin^{\le *}$. When restricted to $\Lat$, the associated graded $\oplus_i \gr_i\Bin_{|\Lat}$ is naturally identified with $\oplus_{i} (\Gamma_{\mbb Z}^i)_{|\Lat}$ (\Cref{rem:properties of Bin}), thus left Kan extending with get a similar identification $\oplus_i \gr_i\Bin\simeq \oplus_{i} \Gamma_{\mbb Z}^i$ as endofunctors of $\Ab^{\mr{tf}}$.
\end{rem}	
\begin{defn}\label{def:binomial algebra}
    A \textit{classical binomial ring} $A$ is a module over the binomial monad $\Bin$ in $\Ab^{\mr{tf}}$. In other words, it is a torsion-free abelian group $A$ together with a map $\Bin(A)\ra A$ that is compatible with the monad structure on $\Bin$. We will denote by $\BinAlg$ the category of binomial rings.
\end{defn}

\begin{rem}\label{rem: binomial rings are torsion free}More generally, one could define a monad $\Bin_{\Ab}$ on the whole category $\Ab$ of (not necessarily torsion free) abelian groups by left Kan extension along $\Lat\subset \Ab$, as in \Cref{constr:classical binomial monad}.  However, it turns out that any $\Bin_{\Ab}$-algebra in $\Ab$ will automatically be torsion-free as an abelian group, and a structure of $\Bin_{\Ab}$-algebra in $\Ab$ is equivalent to the structure of a $\Bin$-module in $\Ab^{\mr{tf}}$ (note that the restriction of $\Bin_{\Ab}$ to $\Ab^{\mr{tf}}$ agrees with $\Bin$ because both commute with filtered colimits).

Indeed, the category $\Ab$ also satisfies a universal property relative to $\Lat$: namely $\Ab$ is the so-called 1-sifted completion $\sInd(\Lat)$ (see \cite[Section 5.1.1, Example 5.3.1(3)]{Cesn-Sch}). It then follows\footnote{See \cite[Section 5.1.4]{Cesn-Sch}; by construction, in our case $\mathbf L\Bin(A)$ is the \textit{animation} of $\Bin_{\Ab}(A)$.} that $\Bin_{\Ab}$ will agree with the 0-th cohomology group of the derived binomial monad $\mathbf L\Bin$ (to be defined later in \Cref{ssec:polynomial functors and bla}); namely, $\Bin_{\Ab}(A)\simeq H^0(\mathbf L\Bin(A))$ for any $A\in \Ab$. By \Cref{rem:bin for classical abelian group} we see that the natural surjection $A\surj A^{\mr{tf}}$ to its maximal torsion-free quotient induces an isomorphism $\Bin_{\Ab}(A)\simeq \Bin_{\Ab}(A^{\mr{tf}})$, and then, by an argument analogous to \Cref{cor:homology of binomial ring}, it follows that any $\Bin_{\Ab}$-algebra $B$ must be torsion free. 
\end{rem}

\begin{rem}
	The map of monads $
	\Sym_{\mbb Z}\ra \Bin
	$ induces a functor $\BinAlg \ra \CAlg$; in particular any classical binomial ring has an underlying classical ring structure.
\end{rem}	

\begin{lem}\label{lem:map from Sym to Bin in the case of a vector space}
	Let $V\in \Vect_{\mbb Q}\subset \Ab^{\mr{tf}}$ be a $\mbb Q$-vector space. Then the map 
		$$
		\Sym_{\mbb Z}(V)\ra \Bin(V)
		$$
		is an isomorphism.
\end{lem}	
\begin{proof}
	The above map is filtered with respect to exhaustive filtrations given by $\Sym_{\mbb Z}^{\le *}(V)$ and $\Bin^{\le *}(V)$. Thus it would be enough to show that the map on the associated graded pieces is an isomorphism. There we have the norm map $\Sym_{\mbb Z}^{i}(V) \ra \Gamma_{\mbb Z}^{i}(V)$, which is an isomorphism since $V$ is a $\mbb Q$-vector space.
\end{proof}	

\begin{rem}\label{rem:map Sym to Bin after tensoring with Q}
	More generally, it is true that for any $A\in \Ab^{\mr{tf}}$ the map 
	$$
	\Sym_{\mbb Z}(A)\otimes \mbb Q \ra \Bin(A)\otimes \mbb Q
	$$
	is an isomorphism. Indeed, again by passing to associated graded we reduce to $\Sym_{\mbb Z}^{i}(A)\otimes \mbb Q \ra \Gamma_{\mbb Z}^{i}(A)\otimes \mbb Q$ being an isomorphism for any $i$. This gives a natural embedding $\Bin(A)\ra \Sym_{\mbb Z}(A) \otimes \mbb Q$ (by identifying the target with $\Bin(A)\otimes \mbb Q$). Note that its image can be identified with the $\mbb Z$-subalgebra generated by binomial coefficients $\binom{[a]}{n}= \frac{[a]([a]-1)\ldots([a]-n+1)}{n!}\in \Sym_{\mbb Z}(A) \otimes \mbb Q$ where $[a]\in \Sym^1_{\mbb Z}(A)\simeq A$ is the element corresponding to $a\in A$. Indeed, we have a commutative diagram 
	$$
	\xymatrix{ \colim_{\substack{A_i\subset A \\ A_i\in \Lat}} \Bin(A_i)\ar[r]^(.62)\sim\ar[d]& \Bin(A)\ar[d]\\
	\colim_{\substack{A_i\subset A \\ A_i\in \Lat}} \Sym_{\mbb Z}(A_i)\ar[r]^(.62)\sim& \Sym_{\mbb Z}(A)
}
	$$
	and for $A\in \Lat$ it is true by \Cref{rem:properties of Bin}(\ref{part:explicit descr of Bin A}).
\end{rem}

Let $A\in \CAlg(\Ab^{\mr{tf}})$ be a commutative ring in $\Ab^{\mr{tf}}$. Then $A$ embeds into $A\otimes \mbb Q$ and for any $a\in A$ the binomial coefficient $\binom{a}{n}\in A\otimes \mbb Q$ is well defined. We say that $A$ is \textit{binomially closed} if $\binom{a}{n}$ in fact lies in $A$ for all $a\in A$ and $n\ge 1$.
\begin{prop}\label{prop:binomial is a property and not a structure}
	Let $A\in \CAlg(\Ab^{\mr{tf}})$. Then $A$ has a binomial ring structure if and only if it is binomially closed, in which case the binomial ring structure is unique. 
\end{prop}	

\begin{proof}
	 Let $A$ be a binomial ring. Then we have a commutative diagram 
$$
\xymatrix{
	\Sym_{\mbb Z}(A) \ar[d]\ar[r]^(.55){m}& A\ar[d] \\
	\Bin(A) \ar[d]\ar[r]^(.55){m}& A\ar[d] \\
	\Sym_{\mbb Z}(A)\otimes \mbb Q \ar[r]^(.55){m}&A\otimes \mbb Q
}	
$$
(by identifying the bottom horizontal arrow with $\Bin(A)\otimes \mbb Q \ra A\otimes \mbb Q$) with all vertical arrows being embeddings. Identifying $\Bin(A)\subset \Sym_{\mbb Z}(A)\otimes \mbb Q$ with the $\mbb Z$-subalgebra generated by $\binom{[a]}{n}$ for $[a]\in \Sym^1_{\mbb Z}(A)$ we get that $\binom{a}{n} = m(\binom{[a]}{n})$ lies in $A$ for any $a\in A$ and $n$. Thus $A$ is binomially closed. We also get that the map $\Bin(A)\ra A$ is uniquely determined by the map $\Sym_{\mbb Z}(A)\ra A$, so the binomial ring structure is unique. 

Let now $A$ be a binomially closed torsion-free ring. We can restrict the map $\Sym_{\mbb Z}(A)\otimes \mbb Q \ra A\otimes \mbb Q$ to $\Bin(A)\subset \Sym_{\mbb Z}(A)\otimes \mbb Q$. Again, identifying $\Bin(A)\subset \Sym_{\mbb Z}(A)\otimes \mbb Q$ with the $\mbb Z$-subalgebra generated by $\binom{[a]}{n}$, we get that the resulting map $\Bin(A)\ra A\otimes \mbb Q$ in fact factors through $A$. To show that the resulting map $\Bin(A)\ra A$ defines a module structure over $\Bin$
 we need to check that the two compositions $\Bin(\Bin(A))\ra A$ are the same. Since both $\mbb Z$-modules are torsion-free this enough to check after tensoring with $\mbb Q$, where via \Cref{rem:map Sym to Bin after tensoring with Q} these maps can be identified with the analogous compositions $\Sym_{\mbb Z}(\Sym_{\mbb Z}(A))\ra A$ tensor $\mbb Q$. Thus the map $\Bin(A)\ra A$ above defines a binomial ring structure on $A$.

\end{proof}	

\begin{rem}
	The proof of \Cref{prop:binomial is a property and not a structure} also shows that the forgetful functor $\BinAlg\ra \CAlg(\Ab^{\mr{tf}})$ is a fully faithful embedding. One can describe the left adjoint explicitly as \text{binomial closure} $L_{\Bin}$; for a commutative torsion-free algebra $A$, $L_{\Bin}A$ is given by the $\mbb Z$-subalgebra $ L_{\Bin}(A)\subset A\otimes \mbb Q$ generated by the binomial coefficients $\binom{a}{n}\in A\otimes \mbb Q$. 
\end{rem}	

\begin{rem}\label{rem:alternative description of binomial algebras} 
 \Cref{prop:binomial is a property and not a structure} tells us in particular that the binomial ring structure on a commutative ring $A$ is a condition and not a structure: the data of a binomial ring is a commutative ring $A$ whose underlying abelian group is torsion-free and which is binomially closed. Note that this is exactly how binomial rings were introduced in \cite{Elliot}.
\end{rem}

Let us now give some natural examples of binomial rings.
\begin{ex}\label{ex:classical binomial rings}
\begin{enumerate}
    \item \label{part:Z as binomial ring}  Ring of integers $\mbb Z$ is torsion-free and binomially closed, so by the above it has a unique binomial ring structure. The map
    $$
    \Bin(\mbb Z)\ra \mbb Z,
    $$ corresponding to this module structure, is explicitly given by the ring homomorphism $\ev_1\colon \mbb Z[\binom{x}{n}]_{n\ge 1}\ra \mbb Z$, which sends $x$ to $1\in \mbb Z$ (here we identified $x$ with the class $[1]\in \mbb Z\subset \Bin^{\le 1}(\mbb Z)$). Note that this map sends $x\choose n$ for $n\ge 2$ to 0. 

    \item More generally, let $S$ be a set; then the ring $\Map_{\Set}(S,\mbb Z)$ of $\mbb Z$-valued functions on $S$ is torsion-free and binomially closed, and so has a unique binomial algebra structure.
    \item Any $\mbb Q$-algebra $A$ is a binomial ring. 
    \item Any torsion-free $\lambda$-ring $A$ whose Adams operations are trivial is a binomial ring \cite{Elliot}, and in fact any classical binomial ring comes this way.
    \end{enumerate}

\end{ex}

\subsection{Polynomial functors and their nonabelian derived extensions}\label{ssec:polynomial functors and bla}

In this section we extend the binomial monad further to the full derived category ${\mscr D}(\mbb Z)$ following \cite[Section 4]{Arpon} and \cite[Section 3]{BM} (which in turn go back \cite{Illusie_Cotangent_II}).

For a small $\infty$-category $\mscr C^0$ we denote by $\mathscr{P}_{\Sigma}(\mscr C^0)$ the full subcategory of the category $\mscr P(\mscr C^0)\coloneqq \Fun((\mscr C^0)^\op,\Spc)$ spanned by functors that preserve finite products (\cite[Definiton 5.5.8.8]{Lur_HTT}). We will call $\mathscr{P}_{\Sigma}(\mscr C^0)$ the \textit{sifted completion of $\mscr C^0$}. In the case $\mscr C^0=\Lat$ the sifted completion $\mathscr{P}_{\Sigma}(\Lat)$ is naturally identified with the category ${\mscr D}(\bZ)^{\le 0}\subset {\mscr D}(\bZ)$ of connective objects (see e.g. \cite[Example 5.1.6(2)]{Cesn-Sch}). Category $\mscr P_{\Sigma}(\mscr C^0)$ satisfies the following universal property: 
\begin{prop}[{\cite[Proposition 5.5.8.15]{Lur_HTT}}] \label{prop:maps from sifted completion}Let $\mathscr D$ be a $\infty$-category which admits all sifted colimits and let $\mathscr C\coloneqq \mathscr P_\Sigma(\mathscr C^0)$ be the sifted completion of a small $\infty$-category $\mathscr C^0$. Let $\Fun_\Sigma$ denote the $\infty$-category of functors that commute with sifted colimits. Then the restriction induces an equivalence of $\infty$-categories
	$$\Fun_\Sigma(\mathscr C,\mathscr D)\ra \Fun(\mathscr C^0,\mathscr D)$$
	 with the inverse functor given by the left Kan extension along $\mathscr C^0 \subset \mathscr C$
\end{prop}

If $F\colon \Lat\ra \mathscr D$ is a functor, we will denote the corresponding sifted colimit preserving functor ${\mscr D}(\mbb Z)^{\le 0}\ra \mathscr D$ by $\mbf LF$ and call it the \textit{non-abelian (left) derived functor} of $F$. 

\begin{rem}\label{rem:from finitely generated modules to connective} By plugging $\mathscr C^0=\Lat$ and $\mathscr D={\mscr D}(\bZ)^{\le 0}$ from \Cref{prop:maps from sifted completion} we get a monoidal equivalence of categories
$$
\End_{\Sigma}^{\Lat}({{\mscr D}(\mbb Z)}^{\le 0})\xra{\sim} \End(\Lat),
$$
where $\End_{\Sigma}^{\Lat}({{\mscr D}(\mbb Z)}^{\le 0})\subset \End_{\Sigma}({{\mscr D}(\mbb Z)}^{\le 0})$ denotes the full subcategory of endofunctors that preserve $\Lat\subset {{\mscr D}(\mbb Z)}^{\le 0}$.
This way any filtered monad $\mathrm T^{\le *}$ on $\Lat$ gives rise to a filtered monad $\mbf L\mr T^{\le *}$ on ${\mscr D}(\mbb Z)^{\le 0}$.
\end{rem}

 As noted by Brantner, if a functor $F$ is polynomial (in the sense of \Cref{def:polynomial functors} below), it can be naturally extended further to the whole derived category ${\mscr D}(\bZ)$. Let us briefly recall the relevant setup, following \cite[Section 4.2]{Arpon}. The following definition is a categorified version of \Cref{def:polynomial maps}.
 \begin{defn}[Additively polynomial functors]\label{def:polynomial functors} Let $\mathscr A$ and $\mathscr B$ be additive $\infty$-categories and assume $\mscr B$ is idempotent-complete. For a functor $F\colon \mathscr A\ra \mathscr B$ and an object $X\in \mathscr A$ the difference functor $D_XF$ is defined as
$$D_XF(Y):=\mathrm{fib}(F(X\oplus Y)\to F(Y)).$$
Then $F$ is 
\begin{itemize}
	\item[1)] \textit{of degree 0} if $F$ is a constant functor;
	\item[2)] \textit{of degree $n$} if $D_XF$ is of degree $n-1$ for any object $X\in\mathscr A$;
	\item[3)] \textit{additively polynomial} if it is of degree $n$ for some $n\geq 0.$
\end{itemize} 
 \end{defn}
 \begin{rem}
     If functors $F,G\colon \mathscr A\ra \mathscr B$ are of degrees $n$ and $m$ correspondingly, their direct sum $F\oplus G$ is additively polynomial of degree $\max(n,m)$. If $\mscr B$ is stable or abelian this applies more generally for any extension of $G$ by $F$ in the category of functors $\Fun(\mscr A,\mscr B)$. Also, given $F\colon \mathscr A\ra \mathscr{B}$ and $G\colon \mathscr{B}\ra \mathscr{C}$ the composition $G\circ F$ is again additively polynomial of degree $mn$. 
 \end{rem}
\begin{ex}
    Functors $\otimes^n$, $\Sym_{\mbb Z}^{n}$, $\wedge_{\mbb Z}^n$, $\Gamma_{\mbb Z}^{n}$ on $\Lat$ are additively polynomial of degree $n$.
\end{ex}

 \begin{lem} \label{lem:Bin^n is polynomial}
	The functor 
	$$\Bin^{\le n}\colon \Lat \ra \Lat$$
	is additively polynomial of degree $n$.
\end{lem}
\begin{proof}
	By \Cref{rem:associated graded for Bin} the functor $\Bin^{\le n}$ has a finite filtration with the associated graded terms given by $\Gamma^i_{\mbb Z}$. Since $\Gamma_{\mbb Z}^i$ is additively polynomial of degree $i$, we get that $\Bin^{\le n}$ is polynomial of degree $n$.
\end{proof}

\begin{defn}
	For $n\ge 0$ let $\mathbf P_n$ denote the poset of subsets of the set $\{0,1,\ldots,n\}$. Let $\mathbf P_n^{\le m}$ (resp. $\mathbf P_n^{\ge m}$) denote the subcategory of $\mathbf P_n$ spanned by subsets of cardinality at most (resp. at least) $m$. An \textit{$n$-cube in an $\infty$-category $\mathscr C$} is a diagram $\chi\colon \mathbf P_n \ra \mathscr C$. An $n$-cube is called:
	\begin{itemize}
		\item \textit{coCartesian} if $\chi$ is a colimit diagram: namely, the natural map 
		$$
		\colim_{I\in \mathbf P_n^{\ge n}} \chi(I) \ra \chi(\{0,\ldots,n\})
		$$
		is an equivalence;
		\item \textit{Cartesian} if $\chi$ is a limit diagram: namely, the natural map
		$$
		\chi(\emptyset) \ra \lim_{I\in \mathbf P_n^{\ge 1}} \chi(I)
		$$
		is an equivalence;
		
		\item \textit{strongly coCartesian} if it is left Kan extended from from its restriction to $\mathbf P_n^{\ge 1}$.
	\end{itemize}
\end{defn}

\begin{rem}
	If $\mathscr C$ is stable then any $n$-cube is Cartesian if and only if it is coCartesian. 
\end{rem}

\begin{defn}[Excisively polynomial functors]
Let $\mathscr C$ be an $\infty$-category admitting finite colimits and let $\mathscr D$ be stable. Then a functor $F$ is called \textit{$n$-excisive} if it sends strongly coCartesian $n$-cubes to Cartesian $n$-cubes. A functor $F$ is called \textit{excisively polynomial} if it is $n$-excisive for some $n$.
\end{defn}

The relation between the two notions of polynomial functors is given by the following: 

\begin{prop}[{\cite[Proposition 3.34]{BM}, \cite[Proposition 5.10]{Johnson_McCarthy}}]\label{prop:polynomial vs excisive} Let $\mathscr C\coloneqq \mathscr P_{\Sigma}(\mathscr C^0)$ be the sifted completion of the subcategory $\mathscr C^0\subset \mathscr C$ of its compact projective objects and let $\mathscr D$ be an $\infty$-category admitting small colimits. Let $\mr F\colon \mathscr C^0\ra \mathscr D$ be an additively polynomial functor of degree $n$. Then the left derived functor $\mbf L\mr F\colon \mathscr C\ra \mathscr D$ (see \Cref{prop:maps from sifted completion}) is $n$-excisive.	
\end{prop}

\noindent This applies in particular to $\mathscr C\coloneqq {\mscr D}(\bZ)^{\le 0}\simeq \mathscr P_\Sigma(\Lat)$.

Assume now that $\mathscr C$ is stable, and that it has a $t$-structure given by $\mathscr C^{\le 0}\subset \mathscr C$. Excisively polynomial functors are often fully determined by their restriction to $\mathscr C^{\le 0}$. For brevity we will formulate the precise statement only in the case $\mathscr C\coloneqq {\mscr D}(\bZ)$, and only for sifted colimit commuting functors.

\begin{prop}\label{prop:from connective to all}Let $\mathscr D$ be an $\infty$-category that has small colimits and let $$\Fun_{\mr{e-poly}}^\Sigma({\mscr D}(\mbb Z),\mathscr D)\subset \Fun({\mscr D}(\mbb Z),\mathscr D)$$ denote the full subcategory spanned by excisively polynomial functors that commute with sifted colimits. Then the restriction to ${\mscr D}(\bZ)^{\le 0}\subset {\mscr D}(\bZ)$ induces an equivalence 
	$$
\Fun_{\mr{e-poly}}^\Sigma({\mscr D}(\bZ),\mathscr D) \xra{\sim} \Fun_{\mr{e-poly}}^\Sigma({\mscr D}(\bZ)^{\le 0},\mathscr D).
$$	
\end{prop}
\begin{proof}
	This is \cite[Proposition 4.2.15]{Arpon} for $\mathscr C\coloneqq {\mscr D}(\bZ)$ and the $t$-structure defined by ${\mscr D}(\bZ)^{\le 0}\subset {\mscr D}(\bZ)$.
\end{proof}

\begin{rem}
	The composition of an $n$-excisive and an $m$-excisive functor is $nm$-excisive, so excisively polynomial functors are closed under composition. In particular, $\End_{\mr{e-poly}}(\mathscr D)\subset \End(\mathscr D)$ and, more generally, $\End_{\mr{e-poly}}^\Sigma(\mathscr D)\subset \End_{\mr{e-poly}}(\mathscr D)$ have natural monoidal structures induced by composition.
\end{rem}

\begin{rem}
	The constant functor $\ul X\colon {\mscr D}(\bZ)^{\le 0}\ra \mscr D$ for $X\in \mscr D$ is $0$-excisive, and is the restriction of the constant functor $\ul X\colon {\mscr D}(\mbb Z) \ra \mscr D$. Thus, constant functors go to constant functors under the above equivalence. Also, in the case $\mscr D={\mscr D}(\mbb Z)$, the identity functor $\id_{{\mscr D}(\mbb Z)}$ is 1-excisive and restricts to the identity functor $\id_{{\mscr D}(\mbb Z)^{\le 0}}$ (composed with the embedding ${\mscr D}(\mbb Z)^{\le 0}\ra {\mscr D}(\mbb Z)$).
\end{rem}	

\begin{construction}[Derived monad associated to a polynomial monad]\label{constr:derived monads associated to polynomial monads}
	Let $\Poly({\mscr D}(\mbb Z))\subset \End({\mscr D}(\mbb Z))$ denote the full subcategory spanned by additively polynomial functors. By \Cref{rem:from finitely generated modules to connective} and \Cref{prop:polynomial vs excisive} the operation of taking the "left derived functor" induces a monoidal functor
	$$
\Poly(\Lat) \xra{\mbf L(-)} \End_{\mr{e-poly}}^\Sigma({\mscr D}(\bZ)^{\le 0}).
	$$ 
	Using the equivalence in \Cref{prop:from connective to all} we then further get a monoidal functor
	$$
	\Poly(\Lat) \xra{\mbf L(-)} \End_{\mr{e-poly}}^\Sigma({\mscr D}(\bZ)).
	$$
	In particular, having a filtered monad $\mr{T}^{\le *}\colon \mbb Z_{\ge 0}^\times \ra \End (\Lat)$ with the property that each endofunctor $\mr T^{\le n}$ is in fact additively polynomial, we obtain a filtered monad $$\mbf L\mr{T}^{\le *}\colon \mbb Z_{\ge 0}^\times \ra \End ({\mscr D}(\bZ)),$$ which we will call the \textit{filtered (full) derived monad} associated to $\mr{T}^{\le *}$.
\end{construction}

\begin{construction}\label{rem:colimit of filtered monad}
	By \cite[Proposition 4.1.4]{Arpon}, if we have a filtered monad $\mr{T}^{\le *}$ on an $\infty$-category $\mathscr C$ that admits small colimits, the corresponding colimit endofunctor $\mr{T}\coloneqq \colim_{n\in \mbb Z_{\ge 0}} \mr{T}^{\le n}\in \End(\mathscr C)$ has a natural structure of a monad on $\mathscr C$.
\end{construction}

Note that by \Cref{lem:Bin^n is polynomial} the filtered monad $\Bin^{\le *}\colon \mbb Z^\times_{\ge 0}\ra \End(\Lat)$ takes values in polynomial functors and thus fits into the framework of \Cref{constr:derived monads associated to polynomial monads}.
\begin{defn}[Derived binomial monad]
	The \textit{filtered derived binomial monad} $\mbf L\Bin^{\le *}$ on ${\mscr D}(\bZ)$ is defined as the full filtered derived monad associated to the filtered monad $\Bin^{\le *}$ on $\Lat$. The \textit{derived binomial monad} $\mbf L\Bin$ on ${\mscr D}(\bZ)$ is defined as the colimit of $\mbf L\Bin^{\le *}$ (see \Cref{rem:colimit of filtered monad}).
\end{defn}

Let us list some first properties of $\mbf L\Bin$, as well as recipes of how to compute it.

\begin{rem}[$\mbf L\Bin$ restricted to ${\mscr D}(\bZ)^{\le 0}$]\label{rem:LBin on connective part}
	By construction, the restrictions of $\mbf L\Bin$ and $\mbf L\Bin^{\le n}$ to ${\mscr D}(\bZ)^{\le 0}$ are obtained as the left Kan extensions of $\Bin$ and $\Bin^{\le n}$ from $\Lat\subset {\mscr D}(\bZ)^{\le 0}$. In particular, their restrictions to $\Ab^{\mr{tf}}\subset {\mscr D}(\bZ)^{\le 0}$ coincide with the functors $\Bin$ and $\Bin^{\le n}$ from \Cref{ssec:classical binomial algebras}. For an arbitrary $M\in {\mscr D}(\bZ)^{\le 0}$, realizing it as the geometric realization $M\simeq |M_\bullet|$ of a simplicial object $M_\bullet\in \Fun(\Delta^\op,\Ab^{\mr{tf}})$, $\Bin(M)$ is computed as the geometric realization $|\Bin(M_\bullet)|$. 
\end{rem}	

\begin{rem}[Associated graded of $\mbf L\Bin$]\label{rem:ass graded for LBin}
Recall that by \Cref{rem:associated graded for Bin} the finite filtration on $\Bin^{\le n}\in \End(\Lat)$ given by $\Bin^{\le *}$ has the associated graded $\oplus_{k=0}^n\Gamma_{\mbb Z}^k$. This way, passing to the colimit, we get a description of the associated graded for the filtration on $\mbf L\Bin $ as $\oplus_{n=0}^\infty \mbf L\Gamma_{\mbb Z}^n$ (with $\gr_n(\mbf L\Bin)\simeq \mbf L\Gamma_{\mbb Z}^n$). Here $\mbf L\Gamma_{\mbb Z}^n$ denotes the non-abelian derived functor of $\Gamma_{\mbb Z}^n$.
\end{rem}

\begin{ex}\label{ex:Bin^n in low degrees} We explicitly describe $\mbf L\Bin^{\le n}$ in low degrees.
	
	\begin{enumerate}
		\item $\Bin^0\coloneqq \Bin^{\le 0}\colon \Lat\ra \Lat$ was given by the constant functor $\ul{\mbb Z}$, thus we get that $\mbf L\Bin^{\le 0}\simeq \ul{\mbb Z}$ is the constant functor.
		
		\item \label{exsub:Bin splits as zero and rest}By \Cref{rem:properties of Bin}(\ref{part:splitting off Bin0}) $\Bin^{\le n}$ splits as $\Bin^{\le 0}\oplus \ol{\Bin^{\le n}}$, which gives a splitting $\mbf L\Bin^{\le n}\simeq \ul{\bZ}\oplus \mbf L\ol{\Bin^{\le n}}$. The second summand $\mbf L\ol{\Bin^{\le n}}$ has a finite filtration $\mbf L\ol{\Bin}^{\le *}$ with the associated graded $\oplus_{i>0}^n \mbf L\Gamma_{\mbb Z}^i$.
		\item \label{ex:Bin^n in low degrees2} $\Bin^{\le 1}\colon \Lat\ra \Lat$ splits as $\ul{\mbb Z}\oplus \id_{\Lat}$ (\Cref{rem:properties of Bin}(\ref{part:splitting off Bin0})), where $\id_{\Lat}\in \End(\Lat)$ is the identity functor. Thus so is its non-abelian derived extension: $\mbf L\Bin^{\le 1}\simeq \ul{\mbb Z}\oplus \id_{{\mscr D}(\mbb Z)}$.
	\end{enumerate}	
\end{ex}

Below, we will mostly be interested in values of $\mbf L\Bin(M)$ for $M\in {\mscr D}(\mbb Z)^{\ge 0}$. For this it will be useful to know that $\mbf L\Bin$ commutes not only with sifted colimits but also some limits.

 Let $X^\bullet\colon \Delta \ra \mathscr C$ be a cosimplicial object in a complete $\infty$-category $\mathscr C$. Recall that $X^\bullet$ is called \textit{$m$-skeletal} if it is the right Kan extension of its restriction to $\Delta_{\le m}\subset \Delta$. A \textit{finite totalization} is a homotopy limit of a cosimplicial object that is $m$-skeletal for some $m$. A functor $F$ is said to commute with finite totalizations if for any $X^\bullet$, that is $m$-skeletal for some $m$, the natural map 
$$
F(\Tot X^\bullet)\ra \Tot(F(X^\bullet))
$$
is an equivalence.

\begin{lem}\label{lem:LBin_n commutes with finite totalizations} For any $n$ the functor $\mbf L\Bin^{\le n}\colon {\mscr D}(\bZ) \ra {\mscr D}(\bZ)$ commutes with finite totalizations.	
\end{lem}
\begin{proof}
	This follows from  \cite[Proposition 3.37]{BM}, since $\mbf L\Bin^{\le n}$ is $n$-excisive.
\end{proof}

\begin{rem}[Computing the value of $\mbf L\Bin(M)$]\label{rem:recipe of computing Bin} Let $M\in {\mscr D}(\mbb Z)$;  we have $M=\colim_{n} \tau^{\le n}M$. Since $\mbf L\Bin$ commutes with filtered colimits, we have $\mbf L\Bin(M)\simeq \colim_{n} \mbf L\Bin(\tau^{\le n}M)$, so we can assume $M\in {\mscr D}(\mbb Z)^{\le m}$ for some $m$. By definition, $\mbf L\Bin=\colim_{\le n}\mbf L\Bin^{\le n}$, and so it is enough to compute $\mbf L\Bin^{\le n}(M)$ for all $n$.  Any $M\in {\mscr D}(\mbb Z)^{\le m}$ can be realized as the totalization $\Tot(M^\bullet)$ of an $m$-skeletal cosimplicial object $M^\bullet\in \Fun(\Delta, {\mscr D}(\mbb Z)^{\le 0})$. \Cref{lem:LBin_n commutes with finite totalizations} then says that $\mbf L\Bin^{\le n}(M)\simeq \Tot(\mbf L \Bin^{\le n}(M^\bullet))$. Finally, each $\mbf L\Bin^{\le n}(M^i)$ for $M^i\in {\mscr D}(\mbb Z)^{\le 0}$ can be computed via \Cref{rem:LBin on connective part}.
\end{rem}	

\subsection{Derived binomial rings} \label{ssec:derived binomial rings}
The following will be the main object of study of this paper:
\begin{defn} A \textit{derived binomial ring} $B$ is a module over the monad $\mbf L\Bin$ on ${\mscr D}(\bZ)$. The $\infty$-category of derived binomial rings will be denoted $\DBinAlg$.
	\end{defn}

\begin{rem}
		By general properties of the category of modules over a monad, $\DBinAlg$ is presentable and, consequently, has all small limits and colimits (\cite[Proposition 4.1.10]{Arpon}). 
	One has the forgetful functor $$U\colon \DBinAlg \ra {\mscr D}(\bZ)$$ which is conservative and has a left adjoint  $$\mbf L\Bin\colon {\mscr D}(\bZ) \ra \DBinAlg$$ 
	given by the free derived binomial ring. Functor ${U}$ commutes with all limits and sifted colimits, see \cite[Corollary 4.2.3.5]{Lur_HA}.
\end{rem}
\begin{rem}[Initial derived binomial ring] Note that we have an isomorphism $\mbf L\Bin(0)\simeq \Bin(0)\simeq \mbb Z$.
Also, $\mbf L\Bin(0)\in \DBinAlg$ is the initial object: indeed, by adjunction
	$$
	\Map_{\DBinAlg}(\mbf L\Bin(0),B)\simeq \Map_{{\mscr D}(\bZ)}(0,{U}(B))\simeq \{*\}.
	$$
	The map $\mbf L\Bin^{\le 0} \ra \mbf L\Bin^{\le n}$ in this case is an equivalence for any $n\ge 0$: indeed, the cofiber of this map has a finite filtration with the associated graded pieces given by $\mbf L\Gamma_{\mbb Z}^k(0)\simeq \Gamma_{\mbb Z}^k(0)=0$. This way the canonical map 
	$$
	\mbf L\Bin(0)\ra B
	$$ 
	for any derived binomial ring $B$ on the level of underlying complexes can be identified with the composition 
	$$
	\mbb Z\simeq \mbf L\Bin^{\le 0}(B)\ra \mbf L\Bin(B) \ra B.
	$$
	\end{rem}

\begin{rem}\label{rem:classical binomial rings as derived}
	Note that the filtered monad $\mbf L\Bin^{\le *}$ (resp. monad $\mbf L\Bin$) preserves the subcategory $\Ab^{\mr{tf}}\subset {\mscr D}(\bZ)$ and its restriction to $\Ab^{\mr{tf}}$ by construction coincides with the filtered monad $\Bin^{\le *}$ (resp. monad $\Bin$) that we considered in \Cref{ssec:classical binomial algebras}. This induces a fully faithful embedding $\BinAlg \subset \DBinAlg$ with the essential image spanned by derived binomial algebras whose underlying complex lies in $\Ab^{\mr{tf}}$.
\end{rem}

The following will be the main example of a derived binomial ring considered in this paper:
\begin{ex}[Derived binomial ring structure on singular cohomology $C^*_\sing (X,\mbb Z)$]\label{ex:cohomology of topological spaces as a derived binomial algebra}
	Let $X\in \mathscr S\mr{pc}$ be a space and consider the binomial algebra
	$$
	\mbb Z^{X}\coloneqq \holim_X \mbb Z,
	$$
	where the limit is taken in $\DBinAlg$ and $\mbb Z$ is the initial derived binomial ring.
	Since the forgetful functor to ${\mscr D}(\bZ)$ commutes with limits we have that the underlying complex of $\mbb Z^{X}$ is given by $\holim_X \mbb Z\in {\mscr D}(\bZ)$, which computes the singular cohomology $C_\sing^*(X,\mbb Z)$. In other words, $\mbb Z^{X}$ endows $C_\sing^*(X,\mbb Z)$ with the natural structure of a derived binomial ring, which we will continue to denote in the same way. Note that $C_\sing^*(\mr{pt},\mbb Z)\simeq \mbb Z^{\mr{pt}}\simeq \mbb Z$ is the initial binomial ring.
\end{ex}

\begin{construction}[Underlying derived commutative ring] In \cite{Arpon} Raksit constructs a monad\footnote{Raksit's construction is much more general, and here we specialized just to the case $\mscr C={\mscr D}(\mbb Z)$.} $\mbf L\Sym_{\mbb Z}$ on ${\mscr D}(\mbb Z)$ and defines the $\infty$-category $\DAlg(\mbb Z)$ of \textit{derived commutative algebras} in ${\mscr D}(\mbb Z)$ as the category of modules over $\mbf L\Sym_{\mbb Z}$. This monad is obtained by the same procedure as in \Cref{constr:derived monads associated to polynomial monads} but applied to the filtered monad $\Sym^{\le *}_{\mbb Z}$ on $\Lat$. The map of filtered monads $\Sym^{\le *}_{\mbb Z}\ra \Bin^{\le *}$ on $\Lat$ (\Cref{ex:filtered monads}(3)) then induces a map
	$$
	\mbf L\Sym^{\le *}_{\mbb Z}\ra \mbf L\Bin^{\le *}
	$$
	 of filtered monads on ${\mscr D}(\mbb Z)$, as well as a map $\mbf L\Sym_{\mbb Z}\ra \mbf L\Bin$ between their colimits. This induces a natural functor 
	 $$
	 \DBinAlg\ra \DAlg(\mbb Z),
	 $$
	 that associates to a derived binomial ring its underlying derived commutative $\mbb Z$-algebra. This functor is conservative and commutes with sifted colimits since its composition with the forgetful functor $ \DAlg(\mbb Z)\ra {\mscr D}(\mbb Z)$ is the forgetful functor ${U}\colon \DBinAlg\ra {\mscr D}(\mbb Z)$.
\end{construction}	

	\begin{rem}[Underlying $\mbb E_\infty$-algebra] \label{R3.1} Let $\AAlg_{\mbb E_\infty}\!(\mbb Z)$ denote the category of $\mbb E_\infty$-algebras in ${\mscr D}(\mbb Z)$ and let $\mr{T}\coloneqq \Sym_{\mbb Z,\mbb E_\infty}\!\!\in \End({\mscr D}(\mbb Z))$ be the monad given by taking (the underlying complex of) the free $\mathbb E_\infty$--algebra; we have an equivalence $\AAlg_{\mbb E_\infty}\!(\mbb Z)\simeq \Mod_{\mr T}({\mscr D}(\mbb Z))$. In \cite{Arpon} Raksit shows that $ \Sym_{\mbb Z,\mbb E_\infty}$ comes from a certain sifted colimit preserving filtered monad $\Sym_{\mbb Z,\mbb E_\infty}^{\le *}$ on ${\mscr D}(\mbb Z)$ via the colimit procedure and constructs a natural map of filtered monads $\Sym_{\mbb Z,\mbb E_\infty}^{\le *}\!\!\ra \mbf L\Sym_{\mbb Z}^{\le *}$ (see \cite[Construction 4.2.19]{Arpon}). This produces a map of monads $\Sym_{\mbb Z,\mbb E_\infty}\!\!\ra \mbf L\Bin$, and, considering corresponding categories of modules, a conservative functor $\DBinAlg \ra \AAlg_{\mbb E_\infty}\!(\mbb Z)$. We will view the value of this functor as the underlying $\mbb E_\infty$-ring of a derived binomial ring. 
	\end{rem}

\begin{rem}	\label{rem:sym vs bin tensor Q}	 
Let us note that the maps of endofunctors $\Sym_{\mbb Z,\mbb E_\infty}\to \mbf L\Sym_{\mbb Z} \to \mbf L\Bin$ all become isomorphisms after applying $-\otimes \bQ$. Indeed, for $A\in \Lat$ we have 
$\Sym(A)\otimes \bQ \xra{\sim} \Bin(A)\otimes \bQ$ by \Cref{rem:map Sym to Bin after tensoring with Q}.
For general $M\in {\mscr D}(\mbb Z)$, the functor 
$$\End({\mscr D}(\mbb Z))\xra{\circ (-\otimes \bQ)} \Fun({\mscr D}(\mbb Z),{\mscr D}(\mbb Q))$$
commutes with colimits, so it suffices to prove that the maps
$$\Sym_{\mbb Z,\mbb E_\infty}^{\leq n}\!(M)\otimes \bQ\to \mbf L\Sym_\bZ^{\leq n}(M)\otimes \bQ \to \mbf L\Bin^{\leq n}(M)\otimes \bQ$$
are equivalence for each $n$. All three functors are $n$-excisive and so, using Proposition \ref{prop:from connective to all}, we may assume that $M\in {\mscr D}(\mbb Z)^{\leq 0}$. Then, via Proposition \ref{prop:maps from sifted completion}, we can reduce to the case $M\in \Lat$. In this case the second map is an isomorphism by \Cref{rem:map Sym to Bin after tensoring with Q}. The first is identified with the direct sum of natural maps $\Sym_{\mbb Z,\mbb E_\infty}^{i}\!(M)\simeq M^{\otimes i}_{h\Sigma_i} \ra H^0(M^{\otimes i}_{h\Sigma_i})\simeq\Sym_{\mbb Z}^i(M)$ which become isomorphisms after tensoring with $\mbb Q$ since $\Sigma_n$ is a finite group.	
\end{rem}
	
	\begin{notation}\label{not:forgetful functors from DAlg}
		We will denote by ${(-)}^\circ\colon \DBinAlg\ra \AAlg_{\mbb E_\infty}\!(\mbb Z)$ the forgetful functor induced by the map of monads $\Sym_{\mbb Z,\mbb E_\infty}\ra \mbf L\Bin$.
	\end{notation}

	The following lemma allows to understand colimits in $\DBinAlg$. 
\begin{lem}\label{lem:forgetful functor commutes with everything}
	The forgetful functor ${(-)}^\circ\colon \DBinAlg\ra \AAlg_{\mbb E_\infty}\!(\mbb Z)$ commutes with all small limits and colimits.
\end{lem}
\begin{proof}
	The argument follows closely the proof of \cite[Proposition 4.2.27]{Arpon}. All small limits and sifted colimits are preserved by ${(-)}^\circ$ since they can be computed on the level of underlying complexes in ${\mscr D}(\mbb Z)$. Since any small colimit can be realized as a sifted colimit of finite coproducts it remains to check the commutation with a single coproduct.
	
	For $A,B\in \DBinAlg$, let us consider the canonical map\footnote{Here $\boxtimes$ denotes the coproduct of $\DBinAlg$ and $\otimes$ denotes the coproduct in $\AAlg_{\mbb E_\infty}\!(\mbb Z)$.}
	\begin{center}
		$A^\circ\otimes B^\circ \sa{\varphi} (A\boxtimes B)^\circ$.
	\end{center} 
	To prove that $\varphi$ is an equivalence, we will reduce to the case when the binomial rings are free and discrete.	First, using that $\DBinAlg$ is by definition the category of modules over $\mbf L\Bin$, we may use Bar-resolutions $A\simeq \colim_{[n]\in \Delta^{\op}} \mbf L\Bin^{(n)} (A)$ and $B\simeq \colim_{[n]\in \Delta^{\op}} \mbf L\Bin^{(n)} (B)$ to reduce to the case of free derived binomial algebras (here $\mbf L\Bin^{(n)}$ denotes the $n$-th iterate of $\mbf L\Bin$). Second, note the $\mbf L\Bin$ is a left adjoint and so commutes with coproducts. Thus, for any $M,N\in {\mscr D}(\mbb Z)$, we can replace $\varphi$ with
	$$\mbf L\Bin(M)^\circ\otimes \mbf L\Bin(N)^\circ\to \mbf L\Bin(M\oplus N)^\circ.$$
	Using that the forgetful functor $\AAlg_{\mbb E_\infty}\!(\mbb Z)\ra {\mscr D}(\mbb Z)$ is symmetric monoidal (where the monoidal structure on the source is given by the tensor product) and tensor product in ${\mscr D}(\mbb Z)$ preserves filtered colimits we can identify the above map with the colimit of a filtered map $\mbf L\Bin^{\leq *}(M)\circledast \mbf L\Bin^{\leq *}(N) \to \mbf L\Bin^{\leq *}(M\oplus N)$
	where $\circledast$ is the Day convolution tensor product on $\Fun(\bZ_{\geq 0},{\mscr D}(\mbb Z))$. Thus it remains to check that the map
	$$
	\colim_{i+j\le n}\left(\mbf L\Bin^{\le i}(M)\otimes \mbf L \Bin^{\le j}(N)\right)\ra  \mbf L\Bin^{\le n}(M\oplus N)
	$$
	induced by $\phi$ is an equivalence for any $n$. Considering both sides as sifted colimit preserving polynomial functors ${\mscr D}(\mbb Z)\times {\mscr D}(\mbb Z)\ra {\mscr D}(\mbb Z)$, using Proposition \ref{prop:from connective to all}, we can reduce to the case when $M, N\in {\mscr D}(\mbb Z)^{\leq 0}$, where, using Proposition \ref{prop:maps from sifted completion}, we can reduce further to the case $M, N\in\Lat$. There it follows from  \Cref{rem:properties of Bin}(\ref{part:monodicity}).
\end{proof}

\begin{cor}\label{cor:forgetful functor to derived algebras commutes with all colimits}
	The forgetful functor ${(-)}^{\circ'}\colon \DBinAlg\ra \DAlg(\mbb Z)$ commutes with all small limits and colimits.
	
\end{cor}	

\begin{proof}
The forgetful functor $\DAlg(\mbb Z)\ra \AAlg_{\mbb E_\infty}\!(\mbb Z)$ is conservative and commutes with all small limits and colimits by \cite[Proposition 4.2.27]{Arpon}, and the composition $\DBinAlg\ra \AAlg_{\mbb E_\infty}\!(\mbb Z)$ commutes with all small limits and colimits by \Cref{lem:forgetful functor commutes with everything}.
\end{proof}	

\begin{notation}[Coproduct in $\DBinAlg$]
Note that the forgetful functor $\AAlg_{\mbb E_\infty}\!(\mbb Z)\ra \mscr D(\mbb Z)$ is symmetric monoidal, where monoidal structure on the left is given by coproduct. From now on we will denote the coproduct in $\DBinAlg$ by $\otimes$ for simplicity: by \Cref{lem:forgetful functor commutes with everything} and the above this is compatible with the forgetful functor $U\colon \DBinAlg\ra \mscr D(\mbb Z)$.
\end{notation}
\begin{rem}\label{rem:shift as Bar-construction}
	Let $f\colon M\ra N$ be a map in ${\mscr D}(\mbb Z)$; it induces a map of derived binomial algebras $\mbf L\Bin(M)\ra \mbf L\Bin(N)$. Since $\mbf L\Bin$ commutes with colimits one gets an equivalence
	$$
	\mbb Z\otimes_{\mbf L\Bin(M)}\!\mbf L\Bin(N)\simeq \mbf L\Bin(\cofib(f))
	$$
	where we identify $\mbb Z\simeq \mbf L\Bin(0)$ and the map $\mbf L\Bin(M)\ra \mbb Z$ is induced by the map $M\ra 0$. In particular, taking $N=0$ we get 
	$$
	\mbb Z\otimes_{\mbf L\Bin(M)}\!\mbb Z\simeq \mbf L\Bin(M[1]);
	$$
	in other words, $\mbf L\Bin(M[1])$ is computed as the Bar-construction of $\mbf L\Bin(M)$.
\end{rem}

As we saw in \Cref{rem:recipe of computing Bin} the computation of the underlying complex of $\mathbf L\Bin(M)$ for general $M\in \mscr D(\bZ)$ is quite involved: in particular, it involves a composition of limit and colimit which a priori don't commute. It would be more convenient in practice if we knew that not only $\mathbf L\Bin^{\le n}$, but the colimit $\mathbf L\Bin$ also commutes with finite totalizations. However, this is not true in general, see \Cref{rem:conditions on commutation with totalizations are sharp}. Nevertheless, $\mathbf L\Bin$ does commute with finite totalizations of cosimplicial objects if they are suitably bounded from below. 

\begin{construction}[Perverse t-structure on $\mscr D(\mbb Z)$]\label{constr:perverse t-structure on abelian groups}
	Recall that there is a \textit{perverse} $t$-structrure on $\mscr D(\mbb Z)$ which we will denote by with $\mscr D(\mbb Z)^{\le 0}_{\mr{perv}}\subset \mscr D(\mbb Z)$ spanned by $M\in $ such that $H^i(M)$ is 0 for $i>1$ and $H^1(M)$ is torsion. Note that we have $\mscr D(\mbb Z)^{\le 0}\subset \mscr D(\mbb Z)^{\le 0}_{\mr{perv}} \subset \mscr D(\mbb Z)^{\le 1}$. The right orthogonal $\mscr D(\mbb Z)^{\ge 0}_{\mr{perv}}$ is spanned by $M$ such that $H^i(M)=0$ for $i<0$ and $H^0(M)$ is torsion free. Note that the intersection $\Perf_{\mbb Z,\mr{perv}}^{\ge 0}\coloneqq \Perf_{\mbb Z}\cap \ \! \mscr D(\mbb Z)^{\ge 0}_{\mr{perv}}$ is identified with $(\Perf_{\mbb Z}^{\le 0})^\vee$ (full subcategory $\mscr D(\mbb Z)$ spanned by duals of objects in $\Perf_{\mbb Z}^{\le 0}$), and its easy to see that in fact $\mscr D(\mbb Z)^{\ge 0}_{\mr{perv}}\simeq \Ind((\Perf_{\mbb Z}^{\le 0})^\vee)$. We have $\mscr D(\mbb Z)^{\ge 1}\subset \mscr D(\mbb Z)^{\ge 0}_{\mr{perv}} \subset \mscr D(\mbb Z)^{\ge 0}$.
	
	The intersection $\Perf(\mbb Z)\cap \mscr D(\mbb Z)^{\heartsuit}_{\mr{perv}}$ is spanned by objects of the form $M^\vee$ where $M$ is a classical finitely generated abelian group. We will denote by $\tau^{\le i}_{\mr{perv}}$ (resp. $H^i_{\mr{perv}}$) the truncation (resp. cohomology) functors with respect to perverse $t$-structure. Note that the standard monoidal structure on $\mscr D(\mbb Z)$ is right t-exact for the standard t-structure it is \textit{left} $t$-exact for perverse one: $\mscr D(\mbb Z)^{\ge 0}_{\mr{perv}}\otimes^{\mbb L}_{\mbb Z} \mscr D(\mbb Z)^{\ge 0}_{\mr{perv}}\subset \mscr D(\mbb Z)^{\ge 0}_{\mr{perv}}$. 
	
\end{construction}	
\begin{rem}\label{rem:Gamma^n and Bin preserves ind-completion of dual to connective perf}
	We claim that for any $n$ one has $\mbf L\Gamma^n(\Perf_{\mbb Z,\mr{perv}}^{\ge 0})\subset \Perf_{\mbb Z,\mr{perv}}^{\ge 0}$. To see this, note that we have $$\mbf L\Gamma^n(M^\vee)\simeq \mbf L\Sym^n(M)^\vee,$$
	and $\mbf L\Sym^n$ preserves the subcategory $\Perf_{\mbb Z}^{\le 0}\subset {\mscr D}(\mbb Z)^{\le 0}$. Indeed, any object in $\Perf_{\mbb Z}^{\le 0}$ is the geometric realization $|A_\bullet|$ of an $m$-skeletal simplicial object $A_\bullet\in \Fun(\Delta^\op,\Lat)$; by \cite[Proposition 2.10.]{BGMNN} the simplicial object $\Sym^n(A_\bullet)\in \Fun(\Delta^\op,\Lat)$ is then $nm$-skeletal, and so its geometric realization is again perfect. Since $\mbf L\Gamma^n$ preserves filtered colimits this also gives $\mbf L\Gamma^n(\mscr D(\mbb Z)^{\ge 0}_{\mr{perv}})\subset \mscr D(\mbb Z)^{\ge 0}_{\mr{perv}}$. The same then holds for $\mbf L\Bin$, since it is a filtered colimit of $\mbf L\Bin^{\le n}$, that in turn have finite filtrations with the associated graded pieces given by $\oplus_{i=0}^n \mbf L\Gamma^k$ and so preserve $\mscr D(\mbb Z)^{\ge 0}_{\mr{perv}}$. 
\end{rem}	

\begin{rem}\label{rem:cohomological bounds on Bin}
	In particular, we have that for $M\in \mscr D(\mbb Z)^{\ge 0}_{\mr{perv}}$, the complexes $\mbf L\Bin^{\le n}(M)$ and $\mbf L\Bin(M)$ land in the coconnective part ${\mscr D}(\mbb Z)^{\ge 0}\subset {\mscr D}(\mbb Z)$.
\end{rem}	

\begin{prop}\label{prop:LBin commutes with suitable totalizations}
	$\mathbf L\Bin$ commutes with finite totalizations of cosimplicial diagrams with values in the subcategory $D(\mbb Z)^{\ge 0}_{\mr{perv}}$. In other words,  for an $m$-skeletal cosimplicial diagram  $M^\bullet$ in $D(\mbb Z)^{\ge 0}_{\mr{perv}}$ the natural map 
	$$
	\mbf L\Bin(\Tot(M^\bullet)) \ra \Tot(\mbf L\Bin(M^\bullet))
	$$
	in $\DBinAlg$ is an equivalence.
\end{prop}	

\begin{proof} 
	For a cosimplicial object $X^\bullet$ let us denote by $\mr{Fib}_n(X^\bullet)$ the cofiber of the natural map 
	$$\Tot(X^\bullet)\simeq \holim_{[\bullet]\in \Delta} X^\bullet \ra \holim_{[\bullet]\in \Delta_{\le n}} X^\bullet.$$ 
	
	Fix some $n\ge 0$ and let $M^\bullet$ be an $n$-skeletal cosimplicial diagram with values in $D(\mbb Z)^{\ge 0}_{\mr{perv}}$.
	We have a map of fiber sequences
	\begin{equation}\label{eq:commutative diagram of totalizations}
		\xymatrix{\colim_m \mr{Fib}_n(\mbf L\Bin^{\le m}(M^\bullet))\ar[r]\ar[d]& \colim_m \Tot(\mbf L\Bin^{\le m} (M^\bullet))\ar[r] \ar[d]& \colim_m \lim_{\Delta_{\le n}} \mbf L\Bin^{\le m}(M^\bullet)\ar[d]\\
			\mr{Fib}_n(\mbf L\Bin(M^\bullet))\ar[r]&\Tot(\mbf L\Bin(M^\bullet)) \ar[r]& \lim_{\Delta_{\le n}} \mbf L \Bin(M^\bullet).}
	\end{equation}
	Note that by \Cref{lem:LBin_n commutes with finite totalizations} the natural map
	$$
	\mbf L\Bin^{\le m}(\Tot(M^\bullet)) \ra \Tot(\mbf L\Bin^{\le m}(M^\bullet))
	$$
	is an equivalence. So the map in question can be identified with middle vertical arrow in \Cref{eq:commutative diagram of totalizations}. 
	
	Note that the right vertical arrow in \Cref{eq:commutative diagram of totalizations} is an equivalence since $\Delta_{\le n}$ is final to a finite diagram and filtered colimits commute with finite limits. Note also that all terms in $\mbf L\Bin^{\le m}(M^\bullet)$ and $\mbf L\Bin(M^\bullet)$ lie in ${\mscr D}(\mbb Z)^{\ge 0}$ and this forces $\mr{Fib}_n(\mbf L\Bin^{\le m}(M^\bullet))$ and $\mr{Fib}_n(\mbf L\Bin(M^\bullet))$ to land in ${\mscr D}(\mbb Z)^{\ge n}$ (see e.g. \cite[Corollary 3.13]{KP_p_Hodge}). It follows that the fiber of the corresponding map $\colim_m \mr{Fib}_n(\mbf L\Bin^{\le m}(M^\bullet))\ra \mr{Fib}_n(\mbf L\Bin(M^\bullet))$ also lies in ${\mscr D}(\mbb Z)^{\ge n}$ and, since the right vertical arrow is an equivalence, the same is true for the fiber of the map 
	$$
	\mbf L\Bin(\Tot(M^\bullet)) \ra \Tot(\mbf L\Bin(M^\bullet)).
	$$
	Since $n$ was arbitrary we get that the fiber of the latter is infinitely coconnective, and thus in fact is 0.
\end{proof}

\begin{cor}\label{cor:on coherent objects it is the right Kan extension}
	The restriction of $\mbf L\Bin\colon {\mscr D}(\mbb Z)\ra {\mscr D}(\mbb Z)$ to $\Perf_{\mbb Z,\mr{perv}}^{\ge 0}\subset {\mscr D}(\mbb Z)$ coincides with the right Kan extension of $\Bin\colon \Lat \to {\mscr D}(\mbb Z)$ along $\Lat\ra \Perf_{\mbb Z,\mr{perv}}^{\ge 0}$.
\end{cor}	
\begin{proof}Equivalently, we need to show that the functor $\Perf_{\mbb Z}^{\le 0} \ra \DBinAlg^\op$ sending $M$ to $\mathbf L\Bin(M^\vee)$ is a left Kan extension of $\Bin((-)^\vee)\colon \Lat \ra \DBinAlg^\op$. Embedding $\Perf_{\mbb Z}^{\le 0}$ into ${\mscr D}(\mbb Z)^{\le 0}\simeq \mscr P_\Sigma(\Lat)\subset \Fun(\Lat^\op, \Spc)$ one sees that to check that a functor $F\colon \Perf_{\mbb Z}^{\le 0}\ra \mscr D$ (where $\mscr D$ has small colimits) is a left Kan extension from $\Lat$, it is enough to check that for any $M\in \Perf_{\mbb Z}^{\le 0}$ and some presentation $M=\colim_i M_i$ as a colimit with $M_i\in \Lat$, the natural map $\colim_i F(M_i)\xra{\sim}F(\colim_i M_i)$ is an equivalence. Note that any $M\in \Perf_{\mbb Z}^{\le 0}$ can be represented as the geometric realization $|M_i|$ of an $m$-skeletal simplicial diagram $M_\bullet\in \Fun(\Delta^\op,\Lat)$. By passing to opposite categories we then reduce to checking that the natural map
	$$
	\mathbf L\Bin(\Tot(M_\bullet^\vee))\ra \Tot(\mathbf L\Bin(M_\bullet^\vee))
	$$
	for the dual cosimplicial diagram $M_\bullet^\vee\in \Fun(\Delta,\Lat)$ is an equivalence. This is a particular case of \Cref{prop:LBin commutes with suitable totalizations}.
\end{proof}

\section{Free derived binomial rings}\label{sec:free derived binomial rings}
In this section we compute free derived binomial rings on abelian groups concentrated in a single degree, which form the building blocks of the category $\DBinAlg$ of derived binomial rings. In the coconnective enough range they turn out to have a natural interpretation as cohomology of Eilenberg-MacLane spaces (see \Cref{ssec:cohomology of Eilenberg-MacLane spaces}); in the connective they rather demonstrate a somewhat singular behavior of $\mbf L\Bin$ (see \Cref{ssec:free derived binomial rings in other degrees}). In \Cref{ssec:LBin(A[-1])} we also compute explicitly what happens in the remaining cases: here $\mbf L\Bin$ no more has a clear topological interpretation, but its cohomology still defines some interesting functors which we fully describe. Besides that, in \Cref{ssec:integral homotopy type theory} we apply our computations to establish an integral version of rational homotopy theory\footnote{Let us note that this was also done earlier in a slightly different context by Horel in \cite{Horel}, and later yet in another slightly different context by Antieau \cite{Antieau_Spherical}.}; namely, we show that nilpotent spaces of finite type fully faithfully embed in $\DBinAlg$ via the singular $\mbb Z$-cohomology functor (\Cref{cor:cohomology defines fully faithful embedding}). In \Cref{ssec:interpretation as cohomology of a stack} we also give a geometric interpretation of $\mbf L\Bin(M^\vee)$ for $M\in \Perf_{\mbb Z}^{\le 0}$ as cohomology of the structure sheaf of a higher stack $K(\mbb H\otimes_{\mbb Z}M)$.
\subsection{Cohomology of Eilenberg-MacLane spaces}\label{ssec:cohomology of Eilenberg-MacLane spaces}
Most of the coconnective part of the category $\DBinAlg$ is in fact glued from the derived binomial rings of the form $C_\sing^*(X,\mbb Z)$, as we will see below. Before formulating the precise result, let us first recall some necessary notions.

\begin{construction}[Generalized Eilenberg-MacLane space]\label{constr:generalized EM-space}
	Recall that the category $\mathscr S\mr{pc}$ of spaces is the sifted completion (or \textit{animation}) of the category $\Set^{\mr{fin}}$ of finite sets. The functor
	$$
	\mbb Z[-]\colon \Set^{\mr{fin}} \ra {\mscr D}(\bZ)^{\le 0}
	$$
	sending a finite set $S$ to the free abelian group $\mbb Z[S]$ naturally extends to a small colimit preserving functor $\mathscr S\mr{pc} \ra {\mscr D}(\bZ)^{\le 0}$, which is known as the singular chains functor $X\mapsto C_*^\sing(X,\mbb Z)$ (and which one could also denote as $X\mapsto \mbb Z[X]$). It has a right adjoint 
	$$
	K\colon {\mscr D}(\bZ)^{\le 0} \ra \mathscr S\mr{pc} 
	$$
known as the \textit{generalized Eilenberg-MacLane space} functor.
	
	Recall that ${\mscr D}(\bZ)$ can be identified with the category $\Mod_{H\mbb Z}(\Sp)$ of $H\mbb Z$-modules in spectra. The forgetful functor ${\mscr D}(\bZ)\ra \Sp$ restricts to a functor ${\mscr D}(\bZ)^{\le 0} \ra \Sp_{\ge 0}$ to connective spectra whose left adjoint is given by the smash product $H\mbb Z\wedge -$. Since $H\mbb Z\wedge \Sigma^\infty_+ X\simeq C_*^\sing(X,\mbb Z)$
we get that the composition 
$${\mscr D}(\bZ)^{\le 0} \ra \Sp_{\ge 0} \xra{\Omega^{\infty}} \mathscr S\mr{pc}
	$$
	with the infinite loop functor gives another formula for $K$. Since $K$ is the right adjoint it commutes will all limits; from the latter formula it also follows that it commutes with sifted colimits by \cite[Corollary 4.2.3.5 and Proposition 1.4.3.9]{Lur_HA}. 
	
	Since both $C_*^\sing(-,\mbb Z)$ and $K$ commute with sifted colimits, given a space $X$ which is the geometric realization of a simplicial set $X_\bullet\in \Fun(\Delta^\op,\Set)$, its singular homology $C_*^\sing(X,\mbb Z)$ can be computed as the geometric realization of the simplicial abelian group $(\mbb Z[X])_\bullet\in \Fun(\Delta^\op,\Ab)$. For $K(-)$, if an object $M$ of ${\mscr D}(\bZ)^{\le 0}$ is the geometric realization $M\simeq |A_\bullet|$ of $A_\bullet\in\Fun(\Delta^\op,\Ab^{\mr{tf}})$, then $K(M)$ is computed by the geometric realization $|A_\bullet|\in \Spc$ where each $A_i$ considered just as a set.
	
From the latter description and classical Dold-Kan correspondence (see e.g. \cite[Section 1.2.3]{Lur_HA}) we see that $\pi_*(K(M))\simeq H_*(M)$. In particular, for $M\coloneqq A[n]$ with $A$ a discrete abelian group we get that $\pi_i(K(A[n]))=0$ if $i\neq n$ and $\pi_n(K(A[n])=A$, in other words $K(A[n])$ is equivalent to the classical $n$-th Eilenberg-MacLane space $K(A,n)$.
	
\end{construction}

\begin{rem}\label{rem: maps to EM-space}
	By adjunction, we get an equivalence
	$$
	\Map_{\mathscr S\mr{pc}}(X,K(M))\simeq \Map_{{\mscr D}(\bZ)^{\le 0}}(C_*^\sing(X,\mbb Z),M)\simeq K(\tau^{\le 0}\Hom_{{\mscr D}(\bZ)}(C_*^\sing(X,\mbb Z),M)).
	$$
	 For $M=\mbb Z[n]$ this gives $\Map_{\mathscr S\mr{pc}}(X,K(\mbb Z,n))\simeq K(\tau^{\le 0}(C^*_\sing(X,\mbb Z)[n]))$, since for any $X$ one has a natural equivalence $(C_*^\sing(X,\mbb Z))^\vee\simeq C^*_\sing(X,\mbb Z)$, where $(-)^\vee\coloneqq R\Hom_{{\mscr D}(\bZ)}(-,\mbb Z)$. 
	
	In particular, looking at $\pi_0$ we get the classical isomorphism $[X,K(\mbb Z,n)]\simeq H^n_\sing(X,\mbb Z)$. This gives a canonical class $[c_n]\in H^n_\sing(K(\mbb Z,n),\mbb Z)$ as an image of $\id_{K(\mbb Z,n)}$, and any class $x\in H^n_\sing(X,\mbb Z)$ is the pull-back of $[c_n]$ under the corresponding map $[x]\colon X\ra K(\mbb Z,n)$.
\end{rem}

\begin{construction}\label{constr:functor K and adjunctions} By plugging $X=K(M)$ in the equivalence 	$$\Map_{\mathscr S\mr{pc}}(X,K(M))\simeq \Map_{{\mscr D}(\bZ)^{\le 0}}(C_*^\sing(X,\mbb Z),M)$$ and looking at the image of $\id_{K(M)}$ we get a natural transformation 
	\begin{equation}\label{eq:transformation from homology of EM-space}
	C_*^\sing(K(-),\mbb Z) \ra \id_{{\mscr D}(\bZ)^{\le 0}} 
	\end{equation}
	of endofunctors of ${\mscr D}(\bZ)^{\le 0}$. Both functors commute with sifted colimits and so this transformation is determined by its restriction to $\Lat$ (by \Cref{prop:maps from sifted completion}). Note that the restrictions of corresponding functors are given by $M\mapsto M$ and $M\mapsto \mbb Z[M]$, while the transformation between them is given by the natural homomorphism $\mbb Z[M]\ra M$ that sends a generator $[m]$ to $m\in M$.

	 Post-composing with $(-)^\vee$ we get a transformation 
	\begin{equation*}
	(-)^\vee \ra C_\sing^*(K(-),\mathbb Z)
	\end{equation*}
	of functors from ${\mscr D}(\mbb Z)^{\le 0}$ to $({\mscr D}(\mbb Z)^{\ge 0})^\op$. Since $(-)^\vee$ sends colimits to limits, the above transformation is also determined by its restriction to $\Lat$, where it is given by the natural embedding 
	$$
	M^\vee \ra \Map_\Set(M,\mbb Z)
	$$
	of $\mbb Z$-linear functions $M\ra \mbb Z$ into all $\mbb Z$-valued functions on $M$ as a set.
	
	Finally, using the derived binomial ring structure on $C_\sing^*(K(-),\mathbb Z)$ (\Cref{ex:cohomology of topological spaces as a derived binomial algebra}) we can extend this to a transformation 
	\begin{equation}\label{eq:transformation from Bin}
	\mbf L\Bin((-)^\vee) \ra C_\sing^*(K(-),\mathbb Z)
	\end{equation}
	of functors from ${\mscr D}(\bZ)^{\le 0}$ to $\DBinAlg$. 
\end{construction}

\begin{ex}\label{ex:map from Bin to all functions}
	Let's evaluate the map (\ref{eq:transformation from Bin}) on $\mbb Z\in {\mscr D}(\bZ)^{\le 0}$. In this case the map is induced by $\Hom_{\mbb Z}(\mbb Z,\mbb Z) \ra \Map_\Set(\mbb Z,\mbb Z)$, where the source is identified with the constant $\mbb Z$-valued functions. The corresponding map $\Bin(\mbb Z) \ra \Map_\Set(\mbb Z,\mbb Z)$ is then identified with the embedding $\Map_{\mr{poly}}(\mbb Z,\mbb Z)\ra \Map_\Set(\mbb Z,\mbb Z)$. Note that this map is definitely \textit{not} an isomorphism: for example the source is countable, while the target has cardinality of continuum. The transformation (\ref{eq:transformation from Bin}) can be thought of as (the derived version of) the embedding of polynomial $\mbb Z$-valued functions into all. 
\end{ex}

Even though (\ref{eq:transformation from Bin}) is quite far from being an equivalence in the simplest case $M\coloneqq \mbb Z$, it is in fact one for any perfect $M$ that is 1-connective. But, before proving this, let us note the following
\begin{lem}
	The functor $C_\sing^*(K(-),\mathbb Z)\colon  \Perf_{\mbb Z}^{\le -1,\op}\ra \DBinAlg$ preserves finite coproducts: namely, for any $M_1,M_2\in \Perf_{\mbb Z}^{\le -1}$ the natural map
	$$
	C_\sing^*(K(M_1),\mathbb Z)\otimes C_\sing^*(K(M_2),\mathbb Z) \ra C_\sing^*(K(M_1\oplus M_2),\mathbb Z)
	$$
	is an equivalence.
\end{lem}
\begin{proof}
	Let $\Lat[1]$ be the image of $\Lat$ under the shift functor $[1]\colon {\mscr D}(\mbb Z)^{\le 0}\ra {\mscr D}(\mbb Z)^{\le -1}$. Note first that the statement is true for $M_1,M_2\in \Lat[1]$. Indeed, if $M\in \Lat[1]$ we have $K(M)\simeq S^1\otimes_{\mbb Z}M$, and so it follows from the K\"unneth formula for singular cohomology. 
	
	Note also that the coproduct in $\DBinAlg$ commutes with finite totalizations in each variable: indeed, it is enough to check this on the level of underlying objects in ${\mscr D}(\mbb Z)$, where it follows from the fact that finite totalizations are final to a finite limit. Any $M\in \Perf_{\mbb Z}^{\le -1}$ can be realized as the finite geometric realization of a simplicial diagram in $\Lat[1]$;  since $C_\sing^*(K(-),\mathbb Z)$ sends geometric realizations to totalizations this way one can reduce from the case of any $M_1,M_2\in \Perf_{\mbb Z}^{\le -1}$ to $M_1,M_2\in \Lat[1]$.
\end{proof}		
\begin{prop}\label{prop: description of free coconnective binomial algebras} For any $M\in \Perf_{\mbb Z}^{\le -1}$ the map 	(\ref{eq:transformation from Bin}) 
	$$
	\mbf L\Bin(M^\vee)\ra C_\sing^*(K(M),\mathbb Z)
	$$
 is an equivalence of derived binomial rings.
\end{prop}

\begin{proof}Note that $\Perf_{\mbb Z}^{\le -1}$ is generated by $\mbb Z[1]$ under finite direct sums and finite geometric realizations. The functor $C_\sing^*(K(-),\mathbb Z)$ sends finite geometric realizations to finite totalizations and finite coproducts to coproducts. By \Cref{prop:LBin commutes with suitable totalizations} this is also true for $\mbf L\Bin((-)^\vee)$ restricted to $\Perf_{\mbb Z}^{\le 0}$. Thus it is enough to show that (\ref{eq:transformation from Bin}) is an equivalence in the case $M=\mbb Z[1]$.
	 
	Note that in this case we have $K(\mbb Z[1])\simeq K(\mbb Z,1)\simeq S^1$, while $K(0)\simeq \mr{pt}$. Evaluating the adjunction $C_*^\sing(K(-),\mbb Z)\ra \id$ on $0\ra \mbb Z[1]\ra 0$ leads to a commutative diagram
	$$
	\xymatrix{\mbb Z\ar[r]\ar[d]& C_*^\sing(S^1,\mbb Z)\ar[d]\ar[r]& \mbb Z \ar[d]\\
	0\ar[r]& \mbb Z[1] \ar[r] & 0.}
	$$
	Note that the top horizontal maps split $C_*^\sing(S^1,\mbb Z)$ as $\mbb Z\oplus \mbb Z[1]$ while the vertical middle map is an isomorphism on $H_1$ and so is isomorphic to the projection on $\mbb Z[1]$.
		
	Dually, evaluating $(-)^\vee\ra C_\sing^*(K(-),\mbb Z)$ on $0\ra \mbb Z[1]\ra 0$ we get a commutative diagram
	$$
	\xymatrix{0 \ar[r]\ar[d]& \mbb Z[-1] \ar[d]\ar[r]& 0 \ar[d]\\
	\mbb Z \ar[r]& C^*_\sing(S^1,\mbb Z) \ar[r] & \mbb Z}
	$$
	in ${\mscr D}(\bZ)$ with $C^*_\sing(S^1,\mbb Z)$ being split as $\mbb Z\oplus \mbb Z[-1]$, and the middle vertical map being the embedding of the second summand. By adjunction we get a commutative diagram
	$$
	\xymatrix{\mbf L\Bin(0) \ar[r]\ar[d]_{\wr}& \mbf L\Bin(\mbb Z[-1]) \ar[d]\ar[r]& \mbf L\Bin(0) \ar[d]^{\wr}\\
		\mbb Z \ar[r]& C^*_\sing(S^1,\mbb Z) \ar[r] & \mbb Z}.
	$$
	with left and right vertical arrows being isomorphisms. Recall that $\mbf L\Bin(\mbb Z[-1])$ has an exhaustive filtration by $\mbf L\Bin^{\le n}(\mbb Z[-1])$ with the associated graded $\mbf L\Gamma^*(\mbb Z[-1])$. By decalage isomorphism \cite[Proposition I.4.3.2.1]{Illusie_Cotangent} (or \cite[Proposition A.2.49]{KP_p_Hodge}) we have $\mbf  L\Gamma^k(\mbb Z[-1])\simeq (\Lambda^k\mbb Z)[-k]$, which is isomorphic to 0 unless $k=0$ or $k=1$. This way we get that the underlying complex $\mbf L\Bin(\mbb Z[-1])$ is quasi-isomorphic to $\mbf L\Bin^{\le 1}(\mbb Z[-1])$, which canonically splits as $\mbb Z\oplus \mbb Z[-1]$ by \Cref{ex:Bin^n in low degrees}(\ref{ex:Bin^n in low degrees2}). Moreover, the composite map $\mbb Z[-1] \ra \mbf L\Bin^{\le 1}(\mbb Z[-1]) \ra C^*(S^1,\mbb Z)$ is identified with the initial map $\mbb Z[-1]\ra C^*(S^1,\mbb Z)$. Thus the middle vertical arrow induces the identity map $\mbb Z\oplus \mbb Z[-1]\ra \mbb Z\oplus\mbb Z[-1]$ and this ends the proof.

\end{proof}

\begin{ex} In particular, applying \Cref{prop: description of free coconnective binomial algebras} to $M$ given by a cyclic abelian group in a single cohomological degree, we get the following:
	\begin{itemize}
		\item For $M=\mbb Z[n]$ we get 
		$$\mathbf L\Bin(\mbb Z[-n])\simeq C^*_\sing(K(\mbb Z,n),\mbb Z).$$
		\item For $M=(\mbb Z/p^k)[n]$ we have $M^\vee\simeq \mbb Z/p^k[-n-1]$, and so
		$$\mathbf L\Bin((\mbb Z/p^k)[-n-1])\simeq C^*_\sing(K(\mbb Z/p^k,n),\mbb Z).$$
	\end{itemize}
Both equivalences only hold for $n\ge 1$.
\end{ex}


\begin{rem}
	Let $X\in \mathscr{S}\mr{pc}$ be a space. 	 Let $x\in H^n_\sing(X,\mbb Z)$ be a cohomology class; it defines a map $[x]\colon X\ra K(\mbb Z,n)$ and the corresponding map $\mbb Z[-n]\ra C^*_\sing(X,\mbb Z)$ in ${\mscr D}(\bZ)$ factors through the universal map $\mbb Z[-n]\ra C^*_\sing(K(\mbb Z,n),\mbb Z)$ (represented by the canonical class $c_n\in H^{n}_\sing(K(\mbb Z,n),\mbb Z)$, see \Cref{rem: maps to EM-space}): we thus have a commutative triangle
	$$
	\xymatrix{\mbb Z[-n] \ar[rr]^{1\mapsto c_n}\ar[dr]_{1\mapsto x} && C^*_\sing(K(\mbb Z,n),\mbb Z)\ar[dl]^{[x]^*}\\
	& C^*_\sing(X,\mbb Z),}
	$$
where the horizontal map is the one from \Cref{eq:transformation from homology of EM-space} in the case $M=\mbb Z[n]$. By using the derived binomial ring structure on $C^*_\sing(X,\mbb Z)$ we get a commutative diagram
	$$
\xymatrix{\mbf L\Bin(\mbb Z[-n]) \ar[rr]^{1\mapsto c_n}_\sim\ar[dr]_{1\mapsto x} && C^*_\sing(K(\mbb Z,n),\mbb Z)\ar[dl]^{[x]^*}\\
	& C^*_\sing(X,\mbb Z)}
$$
of derived binomial rings. In other words, the map 
$$
\mbf L\Bin(\mbb Z[-n])\ra C^*_\sing(X,\mbb Z)
$$
from the free derived binomial ring induced by $x$ is described topologically as the pull-back $$[x]^*\colon C^*_\sing(K(\mbb Z,n),\mbb Z) \ra C^*_\sing(X,\mbb Z)$$ with respect to the map $[x]\colon X\ra K(\mbb Z,n)$ classifying $x\in H_n(X,\mbb Z)$.
\end{rem}

\subsection{Integral homotopy theory}\label{ssec:integral homotopy type theory}

In this section we establish a variant of integral homotopy theory using derived binomial algebras. Our proof is essentially equivalent to the one in \cite{Horel}.

The key step is the following lemma

\begin{lem}\label{lem:singular cohomology is fully faithful for Eilenberg-MacLane spaces}
	Let $N, M\in \Perf_{\mbb Z}^{\le - 1}$. Then the natural map of mapping spaces
	$$
	\Map_{\Spc}(K(N), K(M)) \ra \Map_{\DBinAlg}(C^*_\sing(K(M),\mbb Z),C^*_\sing(K(N),\mbb Z))
	$$
	induced by the functor $C^*_\sing(-,\mbb Z)\colon \Spc^\op \ra \DBinAlg$ is an equivalence.
\end{lem}	

\begin{proof}
By \Cref{prop: description of free coconnective binomial algebras} we can identify 
\begin{equation}\label{eq:maps between Eilenberg-MacLane spaces}
\Map_{\DBinAlg}(C^*_\sing(K(M),\mbb Z),C^*_\sing(K(N),\mbb Z))\xra{\sim} \Map_{{\mscr D}(\mbb Z)}(M^\vee, C^*_\sing(K(N),\mbb Z))
\end{equation}
with the map given by composing the underlying map of complexes with the map $M^\vee \ra C^*_\sing(K(M),\mbb Z)$ from \Cref{constr:functor K and adjunctions}. On the other hand, by $C_*(-,\mbb Z) \dashv K$ adjunction we get an equivalence
$$
\Map_{\Spc}(K(N), K(M))\simeq \Map_{{\mscr D}(\mbb Z)}(C_*^\sing(K(N),\bZ),M)\simeq \Map_{{\mscr D}(\mbb Z)}(M^\vee, C^*_\sing(K(N),\bZ))
$$ 
with the map from the left to the right again given by composing the induced map $C^*_\sing(K(M),\mbb Z)\ra C^*_\sing(K(N),\mbb Z)$ with $M^\vee \ra C^*_\sing(K(M),\mbb Z)$. This way we get that the map in question composed with equivalence \Cref{eq:maps between Eilenberg-MacLane spaces} is an equivalence. Thus so is the original map.
\end{proof}	
\begin{notation}
For a derived binomial ring $B\in \DBinAlg$ let us denote $$[B]\coloneqq \Map_{\DBinAlg}(B,\mbb Z)\in \Spc.$$
We will call $[B]$ the \textit{homotopy type} defined by $B$. Functor $[-]$ sends colimits of derived binomial algebras to limits of spaces.
\end{notation}
\begin{rem}
Note that for any space $X\in \Spc$, the natural identification  $X\simeq \Map_{\Spc}(\{*\}, X)$ induces a comparison map 
$$c_X\colon X\ra \Map_{\DBinAlg}(C^*_\sing(X,\mbb Z),\bZ) =:[C^*_\sing(X,\mbb Z)].$$
\end{rem}
\begin{prop}[{\cite[Theorem 5.4]{Horel}}]\label{prop:X can be reconstructed from cochains} Let $X\in \Spc$ be a connected nilpotent space of finite type. Then the natural map 
	$$
	c_X\colon X\ra [C^*_\sing(X,\mbb Z)]
	$$
	is an equivalence.
\end{prop}	
\begin{proof}
	By \cite[Proposition V.6.1]{GoerssJardinee} any $X$ connected nilpotent of finite type can be expressed as the homotopy limit of a tower 
	$$
	\ldots \ra X_i \ra X_{i-1} \ra \ldots \ra X_1 \ra X_0=*
	$$
	with the property that each map $X_i \ra X_{i-1}$ is a fiber of a map $X_{i-1} \ra K(A_i,k_i)$ with $k_i\ge 2$ and $k_i\ra \infty$ as $i\ra \infty$, and $A_i$ being a finitely generated abelian group for any $i$. By \Cref{lem:singular cohomology is fully faithful for Eilenberg-MacLane spaces} applied to $N=0$ one sees that $c_X$ is an equivalence for $X\simeq K(M)$ and $M\in \Perf_{\mbb Z}^{\le -1}$, in particular it is true for $X_0\simeq K(0)$ and $K(A_i,k_i)\simeq K(A_i[k_i])$ above. Let us show by induction on $i$ that it is also true for $X_i$. Indeed, by Eilenberg-Moore theorem (see e.g. \cite[Theorem 5.4]{Toen_geo}) and \Cref{lem:forgetful functor commutes with everything} the natural map 
	$$
	C^*_\sing(X_{i-1},\mbb Z)\otimes_{C^*_\sing(K(A_i,k_i), \mbb Z)} \mbb Z \ra C^*_\sing(X_{i},\mbb Z)
	$$
	induced by the fiber square 
	$$
	\xymatrix{X_i \ar[r]\ar[d]& X_{i-1}\ar[d]\\
	\{*\} \ar[r]& K(A_i,k_i)}
	$$
	is an equivalence, which allows to identify
	$$
	[C^*_\sing(X_i,\bZ)] \simeq [C^*_\sing(X_{i-1},\bZ)]\times_{[C^*_\sing(K(A_i,k_i),\bZ)]} [\mbb Z],
	$$
	and then $X_i$ with $[C^*_\sing(X_i,\bZ)]$.
	
	Finally, since $k_i\ra \infty$ as $i$ grows we know that 
	$$
	C^*_\sing(X,\mbb Z)\simeq \colim_i C^*_\sing(X_{i},\mbb Z) \in \DBinAlg,
	$$
	and so 
	$$[C^*_\sing(X,\mbb Z)]\simeq \lim_i [C^*_\sing(X_{i},\mbb Z)].$$
	By functoriality we get that $c_X=\lim_i c_{X_i}$ is an equivalence.
\end{proof}	
\begin{rem} Following the same line of argumentation as in \cite[Section 6]{Horel}, more generally for $X$ connected of finite type one can identify $c_X\colon X\ra [C^*_\sing(X,\bZ)]$ with the $H\mbb Z$-localization map  $X \ra L_{\bZ}X$.
\end{rem}	
\begin{cor}\label{cor:cohomology defines fully faithful embedding} The restriction of the functor $C^*_\sing(-,\mbb Z)\colon \Spc^\op \ra \DBinAlg$ to the full subcategory spanned by connected nilpotent spaces of finite type is fully faithful. 
	
\end{cor}	
\begin{proof}
	It is enough to show that the natural map 
	$$
	\Map_{\Spc}(X,Y)\ra \Map_{\DBinAlg}(C^*_{\sing}(Y,\mbb Z), C^*_{\sing}(X,\mbb Z))
	$$
	is an equivalence if $Y$ is connected nilpotent of finite type. Note that since $X=\colim_X \{*\} \in \Spc$ we can rewrite the above map as 
	$$
	\lim_X \left(\Map_{\Spc}(\{*\},Y) \ra \Map_{\DBinAlg}(C^*_{\sing}(Y,\mbb Z), \mbb Z) \right) \simeq \lim_X \left(Y\ra [C^*_{\sing}(Y,\mbb Z)] \right). 
	$$
	It remains to refer to \Cref{prop:X can be reconstructed from cochains} to see that the latter map is an equivalence.
\end{proof}	

\begin{rem}
	In \cite[Theorem 2.9]{Ekedahl} Ekedahl shows that any nilpotent space $X$ is equivalent to a geometric realization $|X_\bullet|$ of a simplicial set $X_\bullet\in \Fun(\Delta^\op,\Set)$ where $X_0\coloneqq *$ is a point, and all terms $X_i$ can be endowed with the structures of finite free abelian groups so that all the maps $X_i\ra X_j$ induced by the simplicial structure on $X_\bullet$ are polynomial in the sense of \Cref{def:polynomial maps} (Ekedahl uses the term "numerical" instead of polynomial). Note that a polynomial map $X_i\ra X_j$ induces a $\mbb Z$-linear map $\mbb Z[X_i] \ra X_j$ via \Cref{rem:description of polynomial maps}, and then, dually, a linear map $X_j^\vee \ra \Bin(X_i^\vee)$, which in turn defines a map $\Bin(X_j^\vee)\ra \Bin(X_i^\vee)$ of classical binomial algebras. As a part of \cite[Theorem 2.9]{Ekedahl} he also shows that $X_\bullet$ can be chosen so that $C^*_\sing(X,\mbb Z)$ is computed by the corresponding totalization $\Tot_{[\bullet]\in \Delta} \Bin(X_\bullet^\vee)$. In our terminology, one gets an equivalence of derived binomial rings
	$$
	C^*_\sing(X,\mbb Z) \xra{\sim} \Tot_{[\bullet]\in \Delta} \Bin(X_\bullet^\vee).
	$$
	In particular, one can choose $X_\bullet$ to be "special and minimal" in the terminology of \textit{loc.cit.}, which, as Ekedahl shows, enjoys some uniqueness and functoriality properties on the (homotopy-)simplicial level.
\end{rem}	

The essential image of the category $(\Spc^{\mr{cn},\mr{nilp},\mr{ft}})^{\mr{op}}$ in $\DBinAlg$ seems to be quite difficult to describe explicitly. It is possible however to describe the essential image of the subcategory $(\Spc^{\mr{1-cn},\mr{ft}})^{\op}\subset (\Spc^{\mr{cn},\mr{nilp},\mr{ft}})^{\mr{op}}$ of \textit{simply-connected} spaces of finite type, as we briefly discuss below. The key to this description is the construction of \textit{Postnikov towers} for the following type of derived binomial rings:
\begin{defn} 
	Let us call a derived binomial ring $A\in \DBinAlg$ 
		\begin{itemize}\item \textit{finite type} if $H^i(A)$ is a finitely generated abelian group for all $i\in \mbb Z$;
			\item 
	\textit{simply connected} if $H^{i}(A)=0$ for $i<0$, $H^0(A)\simeq \mbb Z$, $H^1(A)=0$ and $H^2(A)$ is torsion free.
	\end{itemize}	

\end{defn}	
Denote by $\DBinAlg^{\mr{1-cn, ft}}\subset \DBinAlg$ the subcategory of simply-connected binomial algebras of finite type. Recall the perverse t-structure $(\mscr D(\mbb Z)^{\le 0}_{\mr{perv}}, \mscr D(\mbb Z)^{\ge 0}_{\mr{perv}})$ on $\mscr D(\mbb Z)$ (see \Cref{constr:perverse t-structure on abelian groups}). The simply-connectedness condition in the definition above is equivalent to $H^0_{\mr{perv}}(A)\simeq \mbb Z$ and $H^{i}_{\mr{perv}}(A)=0$ for $i<0$ and $i=1$.

\begin{prop}\label{prop:image of simply-connected spaces}
	Functor $C^*_\sing(-,\mbb Z)\colon (\Spc^{\mr{cn},\mr{nilp},\mr{ft}})^{\mr{op}} \ra \DBinAlg$ restricts to an equivalence 
	$$
	(\Spc^{\mr{1-cn},\mr{ft}})^{\op} \xra{\sim } \DBinAlg^{\mr{1-cn, ft}}.
	$$
	The inverse functor is given by $B \mapsto [B]:=\Map_{\BinAlg}(B,\mbb Z)$.
\end{prop}	
\begin{rem}
	In the cosimplicial setup a similar result is essentially contained in  \cite[Theorem 4.5]{Ekedahl}.
\end{rem}	
\begin{proof}   The functor $C^*_\sing(-,\mbb Z)$ restricted to $(\Spc^{\mr{1-cn},\mr{ft}})^{\op}\subset \Spc^\op$ naturally lands in $\DBinAlg^{\mr{1-cn, ft}}$; indeed $H_1(X,\mbb Z)=0$ for any simply-connected space $X$, which forces $H^1(X,\mbb Z)$ to be 0 and $H^2(X,\mbb Z)$ to be torsion free by the universal coefficients formula. By \Cref{cor:cohomology defines fully faithful embedding} this restriction is fully faithful and we only need to show that $\DBinAlg^{\mr{1-cn, ft}}$ is in its essential image. For this one can imitate the construction of Postnikov tower in the setting of derived binomial algebras.
	
	Given $B\in \DBinAlg^{\mr{1-cn, ft}}$ we will construct an $\mbb N$-indexed system 
	\begin{equation}\label{eq:inverse system}
	B_0\ra B_1 \ra B_2 \ra \ldots \ra B 
	\end{equation}
	 of derived binomial rings such that  for any $i$ we have
	 \begin{itemize}
	 	\item $B_i\simeq C^*_\sing(X_i,\mbb Z)$ for some $X_i\in \Spc^{\mr{1-cn},\mr{ft}}$;
	 	\item the induced map $\tau^{\le i}_{\mr{perv}}B_i\ra \tau^{\le i}_{\mr{perv}}B$ in $D(\mbb Z)$ is an equivalence and $H^{i+1}_{\mr{perv}}(B_i) \ra H^{i+1}_{\mr{perv}}(B)$ is an embedding (in other words fiber of $B_i \ra B$ is $(n+1)$-coconnected in perverse $t$-structure).
	\end{itemize}	
This is enough since by the fully faithfulness of $C^*_\sing(-,\mbb Z)$ the system (\ref{eq:inverse system}) comes from an inverse system $\ldots \ra X_i \ra \ldots \ra X_2 \ra X_1$ of simply-connected spaces and then $B\simeq C^*_\sing(X,\mbb Z)$ where $X:=\lim_i X_i$.

We put $B_0=B_1:=\mbb Z$; the required conditions are satisfied since $B$ is 1-connected. Assume that we have constructed $s_n\colon B_n\ra B$ for some $n\ge 1$. Let $F_n\in \mscr D(\mbb Z)$ be the fiber of $s_n$; we have a fiber sequence 
$$
F_n \ra B_n \ra B
$$
and by our assumptions $F_n\in  \mscr D(\mbb Z)^{\ge n+2}_{\mr{perv}}$. Taking the first non-trivial perverse truncation of $F_n$ we get a natural map $H^{n+2}_{\mr{perv}}(F_n)[-n-2] \ra B_n$ in $\mscr D(\mbb Z)$ which extends to a map $\mbf{L}\Bin(H^{n+2}_{\mr{perv}}(F_n)[-n-2]) \ra B_n$ of derived binomial rings. We let $B_{n+1}$ to be the pushout
$$
\xymatrix{\mbb Z \ar[r] &B_{n+1}\\
\mbf{L}\Bin(H^{n+2}_{\mr{perv}}(F_n)[-n-2]) \ar[u]\ar[r]& B_n\ar[u]}
$$
in $\DBinAlg$. Since the composite map $H^{n+2}_{\mr{perv}}(F_n)[-n-2] \ra B_n\ra B$ is 0, the map $s_n\colon B_n\ra B$ naturally extends to a map $s_{n+1}\colon B_{n+1}\ra B$. Moreover, note that by \Cref{prop: description of free coconnective binomial algebras} we have $\mbf{L}\Bin(H^{n+2}_{\mr{perv}}(F_n)[-n-2])\simeq C^*_\sing(K(H^{n+2}_{\mr{perv}}(F_n)^\vee,n+2),\mbb Z)$. Since $B_n \simeq C^*_\sing(X_n,\mbb Z)$ for some $X_n\in \DBinAlg^{\mr{1-cn, ft}}$ the map $H^{n+2}_{\mr{perv}}(F_n)[-n-2] \ra B_n$ comes from some map of spaces $X_n \ra K(H^{n+2}_{\mr{perv}}(F_n)^\vee,n+2)$. Then the pushout square above induces a map 
$$
B_{n+1} \ra C^*_\sing(X_{n+1},\mbb Z)
$$
	where $X_{n+1}:= \{*\} \times_{K(H^{n+2}_{\mr{perv}}(F_n)^\vee,n+2)} X_n $, which is an equivalence by Eilenberg-Moore theorem. Moreover, since $n\ge 1$ the Eilenberg-Maclane space  $K(H^{n+2}_{\mr{perv}}(F_n)^\vee,n+2)$ is 2-connected and so $X_{n+1}$ is still simply-connected; consequently, $B_{n+1}$ is simply-connected.
	
	It remains to show that the fiber of $B_{n+1}\ra B$ is $(n+2)$-coconnective in the perverse t-structure on $\mscr D(\mbb Z)$. For this we will construct an exhaustive filtration of $\mbb Z\simeq \mbf L\Bin(0)$ as an $\mbf{L}\Bin(H^{n+2}_{\mr{perv}}(F_n)[-n-2])$--module with the associated graded $\mbf{L}\Gamma(H^{n+2}_{\mr{perv}}(F_n)[-n-1])\otimes\mbf{L}\Bin(H^{n+2}_{\mr{perv}}(F_n)[-n-2])$. For simplicity put $A\coloneqq H^{n+2}_{\mr{perv}}(F_n)$. We represent the map $A[-n-2]\ra 0$ by the map of complexes
	$$
	\xymatrix{ \ldots \ar[r] \ar[d]& 0  \ar[r] \ar[d]& 0 \ar[d] \ar[r] & A \ar[r] \ar[d]& 0 \ar[r]\ar[d]& \ldots \\
		\ldots \ar[r] & 0 \ar[r]& A\ar[r]^{\id} & A \ar[r] & 0 \ar[r] & \ldots 
}
	$$ 
	where on top $A$ is put in degree $n+2$. Consider the induced map of Dold-Kan corresponding cosimplicial objects $\mr{DK}^\bullet(A[-n-2]) \ra \mr{DK}^\bullet([A\ra A])$; note that both source and target are obtained as finite totalizations of a cosimplicial object in $\Perf_{\mbb Z,\mr{perv}}^{\ge 0}$. By \Cref{cor:on coherent objects it is the right Kan extension} the induced map on totalizations
	$$
	\Tot(\mbf L\Bin(\mr{DK}^\bullet(A[-n-2]))) \ra \Tot(\mbf L\Bin(\mr{DK}^\bullet([A\ra A])))
	$$ 
	is equivalent to the map $\mbf{L}\Bin(A[-n-2])\ra \mbb Z$. Each term $\mbf L\Bin(\mr{DK}^k([A\ra A]))$ is given by $\mbf L\Bin(\mr{DK}^k(A[-n-2]))\otimes \mbf L\Bin(\mr{DK}^k(A[-n-1]))$ and the binomial filtration on the second term $\mbf L\Bin(\mr{DK}^\bullet(A[-n-2]))\otimes \mbf L\Bin^{\le *}(\mr{DK}^\bullet(A[-n-1]))$ gives a well-defined filtration on the cosimplicial object $\mbf L\Bin(\mr{DK}^\bullet([A\ra A]))$ as an $\mbf L\Bin(\mr{DK}^\bullet(A[-n-2]))$--module. One shows similarly to \Cref{prop:LBin commutes with suitable totalizations} that this gives an exhaustive filtration on $\Tot(\mbf L\Bin(\mr{DK}^\bullet([A\ra A])))\simeq \mbf L\Bin(0)\simeq \mbb Z$ as an $\mbf L\Bin(\mr{DK}^\bullet(A[-n-2]))$--module; corresponding $i$-th graded piece is the totalization of $\mbf L\Bin(\mr{DK}^\bullet(A[-n-2]))\otimes \mbf L\Gamma_{\mbb Z}^i(\mr{DK}^\bullet(A[-n-1]))$, which is identified with $\mbf{L}\Bin(A[-n-2])\otimes \mbf L\Gamma_{\mbb Z}^i(A[-n-1])$.
	
	We also obtain an exhaustive filtration $B_{n+1}^{\le *}$ on $B_{n+1}\coloneqq B_n\otimes_{\mbf{L}\Bin(H^{n+2}_{\mr{perv}}(F_n)[-n-2])}\mbb Z$ by considering the tensor product filtration (with constant filtrations on $B_n$ and $\mbf{L}\Bin(H^{n+2}_{\mr{perv}}(F_n)[-n-2])$). Its associated graded is $B_n\otimes_{\mbb Z} \mbf{L}\Gamma_{\mbb Z}^*(H^{n+2}_{\mr{perv}}(F_n)[-n-1])$ with the grading coming from $\mbf L\Gamma_{\mbb Z}^*$.
	Recall that $\mbf{L}\Gamma_{\mbb Z}^i(H^{n+2}_{\mr{perv}}(F_n)[-n-1])\simeq (\mbf{L}\Sym^i(H^{n+2}_{\mr{perv}}(F_n)^\vee[n+1]))^\vee$ and that $H^{n+2}_{\mr{perv}}(F_n)^\vee$ is a finitely generated classical abelian group. By \cite[Proposition 25.2.4.1]{Lur_SAG} $\mbf{L}\Sym^i(H^{n+2}_{\mr{perv}}(F_n)^\vee[n+1])$ is $(2i+n -2)$--connective in the classical t-structure on $\mscr D(\mbb Z)$; thus $\mbf{L}\Gamma_{\mbb Z}^i(H^{n+2}_{\mr{perv}}(F_n)[-n-1])$ is $(2i+n -2)$--coconnective in the perverse t-structure and, consequently, so is each graded piece $B_n\otimes_{\mbb Z} \mbf{L}\Gamma_{\mbb Z}^i(H^{n+2}_{\mr{perv}}(F_n)[-n-1])$.

	Consider the spectral sequence for $H^{*}_\mr{perv}(-)$ corresponding to the above filtration:
	$$
	E_2^{p,q}:= H^{p+q}_{\mr{perv}}(B_n \otimes \mbf{L}\Gamma_{\mbb Z}^q(H^{n+2}_{\mr{perv}}(F_n)[-n-1])) \Rightarrow H^{p+q}_\mr{perv}(B_{n+1}).
	$$
	Note that the first column is identically 0 since $B_n$ is simply connected. Also, $E_2^{p,0}\simeq H^p_{\mr{perv}}(B_n)$ since $H^0_{\mr{perv}}(\mbf{L}\Gamma_{\mbb Z}^*(H^{n+2}_{\mr{perv}}(F_n)[-n-1]))=\mbb Z$, $E_2^{p,q}=0$ for $0<q<n+1$ and $q=n+2$ because of the description of perverse cohomology of  $\mbf{L}\Gamma_{\mbb Z}^q(H^{n+2}_{\mr{perv}}(F_n)[-n-1])$ above.
	To summarize, second sheet looks like that:
\begin{flushleft}
\begin{tikzpicture}
	\matrix (m) [matrix of math nodes,
	nodes in empty cells,
	nodes={minimum width=12ex,
		minimum height=12ex,
		outer sep=-1pt},
	column sep=2.5ex, row sep=0.1ex]{
		n+2 \strut &  0  \strut & 0       &0      &\cdots & 0     & 0 & 0\\                   
		n+1    &   H^{n+2}_{\mr{perv}}(F_n)  &  0   & \cdots     & \cdots & \cdots & \cdots & \cdots \\
		n   &  0   & 0  & \cdots    & \cdots & 0 &0 & 0\\
		\vdots \strut   &  \vdots   &  \vdots & \vdots    & \vdots & \vdots & \vdots  &\vdots\\
		1      &  0   & 0  & 0   & \cdots & 0&0 & 0\\
		0      &  \mbb Z   &  0  &    H^{2}_{\mr{perv}}(F_n)&  \cdots   &  H^{n+1}_{\mr{perv}}(B_n) & H^{n+2}_{\mr{perv}}(B_n) & H^{n+3}_{\mr{perv}}(B_n) \\
		\quad\strut &   0 &  1  &  2  &  \cdots &  n+1  & n+2 & n+3  \strut \\};
	\draw[-stealth] (m-2-2) -- (m-6-7);
	\draw[thick] (m-1-1.north east) -- (m-7-1.east) ;
	\draw[thick] (m-7-1.north) -- (m-7-8.north east) ;
\end{tikzpicture}
\end{flushleft}
with the arrow depicting the first potentially non-zero differential. From the form of the spectral sequence we thus get that the natural map $H^i_{\mr{perv}}(B_{n})\ra H^{i}_{\mr{perv}}(B_{n+1})$ is an isomorphism for $i\le n$, that there is an exact sequence
$$
0\ra H^{n+1}_{\mr{perv}}(B_n) \ra H^{n+1}_{\mr{perv}}(B_{n+1}) \ra H^{n+2}_{\mr{perv}}(F_n) \ra H^{n+2}_{\mr{perv}}(B_n) \ra H^{n+2}_{\mr{perv}}(B_{n+1}) \ra 0.
$$
and for $i=n+3$ we have an embedding $H^{n+3}_{\mr{perv}}(B_{n})\ra H^{n+3}_{\mr{perv}}(B_{n+1})$. Let $F_{n+1}$ be the fiber of $B_{n+1}\ra B$ and $C_n$ be the cofiber of $B_n\ra B_{n+1}$. By the above we have $H^i_{\mr{perv}}(C_n)=0$ for $i\le n$,  $H^{n+1}_{\mr{perv}}(C_n)\simeq H^{n+2}_{\mr{perv}}(F_n)$ and $H^{n+2}_{\mr{perv}}(C_n)=0$. We have a commutative diagram
$$
\xymatrix{C_n\ar@{=}[r]& C_n \ar[r]& 0\\
	F_{n+1} \ar[r]\ar[u]& B_{n+1}\ar[r]\ar[u]& B\ar[u]\\
	F_n \ar[r]\ar[u]& B_n\ar[r]\ar[u]& B\ar@{=}[u]}
$$
where all rows and columns are fiber sequences. Looking at the left column we get a long exact sequence of perverse cohomology
\begin{equation}\label{eq:exact sequence of cohomology}
0=H^{n+1}_{\mr{perv}}(F_n)\ra H^{n+1}_{\mr{perv}}(F_{n+1}) \ra H^{n+1}_{\mr{perv}}(C_n) \ra H^{n+2}_{\mr{perv}}(F_n) \ra H^{n+2}_{\mr{perv}}(F_{n+1}) \ra H^{n+2}_{\mr{perv}}(C_n)=0.
\end{equation}

To analyse the maps in this sequence we now replace $B_{n+1}$ with analogous pushout, but in $\mscr D(\mbb Z)$ rather then $\DBinAlg$. Namely, let $P_{n+1}$ be the pushout
$$
\xymatrix{B_n \ar[r] & P_{n+1}\\
H^{n+2}_{\mr{perv}}(F_n)[-n-2]\ar[u] \ar[r]& 0	\ar[u] 
}
$$
 in $\mscr D(\mbb Z)$, where the map $H^{n+2}_{\mr{perv}}(F_n)[-n-2]\ra B_n$ is induced by $F_n \ra B_n$. We have a natural map $P_{n+1} \ra B_{n+1}$, and in fact one can describe $P_{n+1}$ in terms of the filtration on $B_{n+1}$ that we considered. Namely, note that there is a natural map $\Sigma (H^{n+2}_{\mr{perv}}(F_n)[-n-2]) \simeq H^{n+2}_{\mr{perv}}(F_n)[-n-1] \ra P_{n+1}$ (which is induced by $0\ra B_n$) whose fiber is $B_n$. Similarly, the filtered piece $B_{n+1}^{\le 1}$ maps to $B_n\otimes H^{n+2}_{\mr{perv}}(F_n)[-n-1]$ with fiber $B_n$ and $P_{n+1}$ is exactly the pull-back of this extension via the natural map $H^{n+2}_{\mr{perv}}(F_n)[-n-1] \ra B_n\otimes H^{n+2}_{\mr{perv}}(F_n)[-n-1]$ (induced by $\mbb Z \ra B_n$). Writing the spectral sequence for perverse cohomology of this extension we see that it agrees with the spectral sequence for $B_{n+1}^{\le *}$ in the range drawn above. In particular, the cofiber $\widetilde{C}_n\simeq H^{n+2}_{\mr{perv}}(F_n)[-n-1]$ of $B_n \ra P_{n+1}$ agrees with $C_n$ up $(n+2)$--nd perverse cohomology. Consequently, using \Cref{eq:exact sequence of cohomology}, if we put $\widetilde F_{n+1}$ to be the fiber of natural map $P_{n+1}\ra B$, we also have $H^{i}_{\mr{perv}}(\widetilde F_{n+1})\simeq H^{i}_{\mr{perv}}(F_{n+1})$ for $i\le n+2$. But now it is clear that in the map of fiber sequences
 $$
 \xymatrix{\widetilde F_{n+1}\ar[r] &P_{n+1}\\
F_n \ar[r] \ar[u]& B_n\ar[u]\\
 H^{n+2}_{\mr{perv}}(F_n)[-n-2] \ar[u]\ar@{=}[r] & H^{n+2}_{\mr{perv}}(F_n)[-n-2] \ar[u]
} 
 $$
the low left vertical arrow is identified with the map from $(n+2)$--nd truncation (because $P_{n+1}$ is the cofiber of the composite $H^{n+2}_{\mr{perv}}(F_n)[-n-2] \ra F_n \ra B_n$). Writing long exact sequence of perverse cohomology for the left fiber sequence we get that $H^{i}_{\mr{perv}}(\widetilde F_{n+1})\simeq H^{i}_{\mr{perv}}(F_{n+1})$ for all $i\le n+2$ are 0.
\end{proof}	

\begin{rem}
	As one can see from the construction of \Cref{prop:image of simply-connected spaces} the tower $\mbb Z =B_0 \ra B_1 \ra B_2 \ra \ldots \ra B$ really corresponds to the Postnikov tower $X \ra \ldots \ra X_2 \ra X_1 \ra X_0 =*$ (for the space $X:= \Map_{\DBinAlg}(B,\mbb Z)$) after applying $\Map_{\DBinAlg}(-,\mbb Z)$.  Indeed, for every $n$ the pushout
	$$
	\xymatrix{\mbb Z \ar[r] &B_{n+1}\\
		\mbf{L}\Bin(H^{n+2}_{\mr{perv}}(F_n)[-n-2]) \ar[u]\ar[r]& B_n\ar[u]}
	$$ 
	is turned into a pull-back square 
	$$
		\xymatrix{X_{n+1} \ar[r] \ar[d]& X_n \ar[d]\\
{*} \ar[r] & K(H^{n+2}_{\mr{perv}}(F_n)^\vee, n+2)}
	$$
Here note again that by the definition of perverse $t$-structure, $H^{n+2}_{\mr{perv}}(F_n)^\vee$ is a classical abelian group, and so by induction and long exact sequence of homotopy groups we get that $X_{n+1}$ is $n+1$-truncated. It also follows that $\pi_i(X_n)$ stablizes with $n\gg 0$ and $\pi_i(X)\simeq \pi_i(X_n)$ for $n\ge i$.
\end{rem}	

\subsection{Interpretation as cohomology of generalized classifying stacks}\label{ssec:interpretation as cohomology of a stack}

Let $R$ be a classical commutative ring. 
\begin{construction}[Higher stacks]\label{constr:higher stacks} Let $\Aff_{\!/\!R}$ be the category of (classical) affine schemes over $R$. We endow $\Aff_{\!/\!R}$ with the Grothendieck topology generated by faithfully flat covers. We let $\Stk_{\!/\!R}$ denote the $\infty$-category  $\hShv^{\fpqc}(\Aff_{\!/\!R},\mathscr S\mr{pc})\subset \Fun(\Aff_{\!/\!R}^{\op}, \Spc)$ of \textit{higher stacks}: namely, the category of $\Spc$-valued hypersheaves in fpqc topology on $\Aff_{\!/\!R}$.
	
	Having a higher stack $\mstack Y$ with a fixed point $y\in \mstack Y(\Spec R)$ (which by Yoneda lemma can also be viewed as the map $y\colon \{*\}\ra \mstack Y$, where $\{*\}=\Spec R$ is the final object of $\Stk_{\!/\!R}$) one can consider the corresponding \textit{homotopy group sheaves} 
	$$\{\pi_i(\mstack Y,y)\}_{i\in \mbb N}\in \Shv^{\fpqc}(\Aff_{\!/\!R},\Set)$$ defined as sheafification of the presheaves $X\mapsto \pi_i(\mstack Y(X),y)\in \Set$. As usual, for $i\ge 1$ the sheaves  $\pi_i(\mstack Y,y)$ have a natural structure of sheaves of groups, that are in fact commutative if $i\ge 2$.
\end{construction}	

One can extend a lot of cohomology theories to higher stacks; here we will mostly be interested only in the following:
\begin{construction}[$\mscr O$-cohomology]\label{constr:O-cohomology}
	
	Consider the functor $\mscr O_{\DAlg}^{\mr{cl}}\colon \Aff^{\op}_{\!/\!R}\ra \DAlg(R)$ sending an affine scheme $\Spec A$ to $A$ considered as a derived commutative $R$-algebra; note that by the classical faithfully flat descent this is an fpqc-sheaf. Taking right Kan extension along the embedding $\Aff_{\!/\!R}^\op\ra \Stk_{\!/\!R}^\op$ we obtain a functor
	$$
	\mscr O_{\DAlg}\colon \Stk_{\!/\!R}^\op \ra \DAlg(R).
	$$
	Since $\mscr O_{\DAlg}^{\mr{cl}}$ is a sheaf, one can see that $\mscr O_{\DAlg}\colon \Stk_{\!/\!R}^\op \ra {\mscr D}(R)$ sends colimits in $\Stk_{\!/\!R}$ to limits in $\DAlg(R)$.
\end{construction}	
It will also be useful to have a suitable notion of abelian group (higher) stack. The idea of the following definition goes back to Lawyere.
\begin{construction}[Abelian group objects]\label{constr:abelian group objects} Let $\mscr C$ be an $\infty$-category admitting finite products. An \textit{abelian group object} of $\mscr C$ by definition is a functor $A\colon \Lat^\op \ra \mscr C$ that commutes with finite products. We will denote by $\Ab(\mscr C)\subset \Fun(\Lat^\op,\mscr C)$ the full subcategory spanned by abelian group objects of $\mscr C$. One has the forgetful functor $K\colon \Ab(\mscr C)\ra \mscr C$ given by evaluation at the object $\mbb Z\in \Lat^\op$. Functor $K$ commutes with limits and, since sifted colimits commute with products, also with sifted colimits. 
	
	In the case $\mscr C=\Spc$, the category $\Ab(\Spc)$ is naturally identified with ${\mscr D}(\mbb Z)^{\le 0}$ (\cite[Example 1.2.9]{Lur_Ell_I}). The functor $K\colon \Ab(\Spc)\ra \Spc$ is then naturally identified with the generalized Eilenberg-MacLane space functor that we considered in \Cref{constr:generalized EM-space}; recall that it has a left adjoint $\mbb Z[-]\colon \Spc\ra \Ab(\Spc)$ given by the "free abelian group".  More generally, for $\mscr C= \Stk_{\!/\!R}$ the functor $K\colon \Ab(\Stk_{\!/\!R})\ra \Stk_{\!/\!R}$ also has a left adjoint given by applying $\mbb Z[-]$ on the level of presheaves, and then sheafification.
	
	We will consider the $\infty$-category $\Ab(\Stk_{\!/\!R})$ of \textit{abelian group $R$-stacks}. By \cite[Example 1.2.13]{Lur_Ell_I} $\Ab(\Stk_{\!/\!R})$ is naturally identified with the tensor product $\Stk_{\!/\!R}\otimes {\mscr D}(\mbb Z)^{\le 0}$, which, in turn, can be further identified with the $\infty$-category $\hShv^{\fpqc}(\Aff_{\!/\!R},  {\mscr D}(\mbb Z)^{\le 0})$ of ${\mscr D}(\mbb Z)^{\le 0}$-valued hypersheaves on $\Aff_{\!/\!R}$. One has a natural embedding $\Shv^{\fpqc}(\Aff_{\!/\!R},  \Ab)\subset \Ab(\Stk_{\!/\!R})$; in particular, any classical abelian group $R$-scheme naturally defines an object of $\Ab(\Stk_{\!/\!R})$.
	
\end{construction}	

\begin{rem}Note that given an abelian group $R$-stack $H\colon \Lat^\op \ra \Stk_{\!/\!R}$, evaluating $H$ on the unique map $0\ra \mbb Z$ in $\Lat^\op$ endows $K(H)\coloneqq H(\mbb Z)$ with canonical point $0\colon \Spec R=H(0)\ra H(\mbb Z)\simeq K(H)$. We then can consider the corresponding homotopy group sheaves $\pi_i(K(H),0)$ as in \Cref{constr:higher stacks}, which we will simply denote $\pi_i(H)$. Under the identification $\Ab(\Stk_{\!/\!R})\simeq \hShv^{\fpqc}(\Aff_{\!/\!R},  {\mscr D}(\mbb Z)^{\le 0})$, the homotopy groups $\pi_i(H)$ correspond to the classical sheaves $\mc H^{-i}(H)\in \hShv^{\fpqc}(\Aff_{\!/\!R},  \Ab)$ given by $(-i)$-th cohomology. E.g. by \cite[Proposition 3.3.3]{BS_proetale}, Postnikov towers converge in $\hShv^{\fpqc}(\Aff_{\!/\!R},  {\mscr D}(\mbb Z)^{\le 0})$; consequently we have Whitehead theorem in this setting: a map $H_1\ra H_2$ in $\Ab(\Stk_{\!/\!R})$ is an equivalence if and only if it induces an isomorphism on all $\pi_i$.
\end{rem}	

\begin{rem}\label{rem:tensoring with a group scheme}
	Given an abelian group stack $H\in \Ab(\Stk_{\!/\!R})$, one has a natural colimit preserving functor $H\otimes_{\mbb Z}-\colon {\mscr D}(\mbb Z)^{\le 0}\ra \Ab(\Stk_{\!/\!R})$ given by tensor product with $H$: $M\mapsto H\otimes_{\mbb Z} M\in \Ab(\Stk_{\!/\!R})$. One can view it via \Cref{prop:maps from sifted completion} as the unique sifted colimit preserving functor that sends a finitely generated free abelian group $A\in \Lat$ to $H\otimes_{\mbb Z}A$. Since $K\colon \Ab(\Stk_{\!/\!R}) \ra \Stk_{\!/\!R}$ commutes with sifted colimits, and $\mscr O_{\DAlg}$ sends colimits in stacks to limits in $\DAlg(R)$ the functor 
	$$
	\mscr O_{\DAlg}(K(H\otimes_{{\mbb Z}} - ))\colon {\mscr D}(\mbb Z)^{\le 0,\op} \ra  \DAlg(R)
	$$
	sends sifted colimits in ${\mscr D}(\mbb Z)^{\le 0}$ to limits in $\DAlg(R)$.
	
	 In the case $M=\mbb Z[n]$ the tensor product $H\otimes_{\mbb Z}\mbb Z[n]\simeq H[n]$ is often denoted as $B^n H$. Also, in the case $M=\mbb Z/n$ we will denote $H\otimes_{{\mbb Z}}\mbb Z/n$ by $[H/n]$ (meaning "the quotient stack via the action through multiplication by $n$"). Indeed, since $\mbb Z/n\simeq \cofib(\mbb Z\xra{\cdot n} \mbb Z)$ we have $[H/n]\simeq \cofib(H\xra{\cdot n}H)\in  \Ab(\Stk_{\!/\!R})$.
\end{rem}

\begin{construction}[The map $i\colon \ul\bZ \ra \mbb H$]\label{constr:map from Z to H} Let $R=\mbb Z$ and let $\mbb H=\Spec(\Bin(\mbb Z))$. Let $\ul{\bZ}$ be the hypersheafification of the constant presheaf given by $\mbb Z$.
	One has $\ul{\bZ}=\mbb Z[*]$, and so by adjunction, abelian group sheaves homomorphisms $\ul \bZ \ra \mbb H$ in $\Shv^{\fpqc}(\Aff_{\!/\!R},  \Ab)$ are in bijection with $\mbb H(\Spec \mbb Z)$ or, in other words, algebra homomorphisms $\Bin(\mbb Z)\ra \mbb Z$. There is a preferred such homomorphism $\ev_1\colon \Bin(\mbb Z)\ra \mbb Z$ given by the binomial ring structure on $\mbb Z$ (\Cref{ex:classical binomial rings}(\ref{part:Z as binomial ring})), and we will denote the corresponding map of group sheaves $i\colon \ul \bZ \ra \mbb H$. On the level of functions (after applying $\mscr O_{\DAlg}$) the pull-back $i^*$ is given by the natural embedding $\Map_{\mr{poly}}(\mbb Z,\mbb Z) \ra \Map_{\Set}(\mbb Z,\mbb Z)$ of polynomial functions into all.
\end{construction}	

\begin{rem}[$\mathbf L\Bin(M^\vee)\in \DAlg(\mbb Z)$ as $\mscr O$-cohomology]\label{rem:LBin as cohomology of a stack} Note that via the discussion in \Cref{rem:description of polynomial maps} and definition of $\Bin$ we have a natural in $A\in \Lat$ identification 
	$$
	\Spec(\Bin(A^\vee))\simeq \Spec(\Map_{\mr{poly}}(A,\mbb Z))\simeq \mbb H\otimes_{\mbb Z} A
	$$
	as group $\mbb Z$-schemes. By taking back the global sections of $\mscr O$, we get the identification of the restriction of the functor
	$$
	\mathbf L \Bin((-)^\vee)\colon \Perf_{\mbb Z}^{\le 0,\op}\ra \DAlg(\mbb Z)
	$$
	to $\Lat^{\op}\subset  \Perf_{\mbb Z}^{\le 0,\op}$ with $\mscr O_{\DAlg}(K(\mbb H\otimes_{\mbb Z} -))$. Since both functors send finite geometric realizations to totalizations one gets from this an identification of functors
	$$
	\mathbf L \Bin((-)^\vee)\xra{\sim} \mscr O_{\DAlg}(K(\mbb H\otimes_{\mbb Z} -));
	$$
	in particular, for any $M\in \Perf_{\mbb Z}^{\le 0}$ we get a natural equivalence
	$$
	\mathbf L \Bin(M^\vee) \xra{\sim} \mscr O_{\DAlg}(K(\mbb H\otimes_{\mbb Z} M))
	$$
	of derived commutative $\mbb Z$-algebras. 
\end{rem}	

\begin{rem}
	Similarly, having $M\in {\mscr D}(\mbb Z)^{\le 0}$ we can consider the associated constant abelian group stack $\ul M\in \Ab(\Stk_{/\mbb Z})$ and its underlying stack $K(\ul M)$. Then we claim that there is a natural equivalence 
	$$
	C^*_{\sing}(K(-),\mbb Z)\simeq \mscr O_{\DAlg}(K(-))
	$$
	of functors ${\mscr D}(\mbb Z)^{\le 0,\op}\ra \DAlg(\mbb Z)$. There are various ways to see this; for example, both functors send sifted colimits in ${\mscr D}(\mbb Z)$ to limits in $\DAlg(\mbb Z)$, and their restrictions to $\Lat^\op \subset {\mscr D}(\mbb Z)^{\le 0,\op}$ can both be identified with the functor
	$$
	\Map_{\Set}(-,\mbb Z)\colon \Lat^\op\ra \DAlg(\mbb Z)
	$$
	sending a lattice $A\in \Lat$ to the algebra of all $\mbb Z$-valued functions on $A$.
\end{rem}	

\begin{rem}The map $i\colon \ul \bZ\ra \mbb H$ that we constructed in \Cref{constr:map from Z to H} induces a natural in $M$ map $M\ra \mbb H\otimes_{\mbb Z} M$ of abelian group stacks. By applying $\mscr O_{\DAlg}$ we then get a transformation 
	$$
	\mbf L\Bin((-)^\vee)\simeq \mscr O_{\DAlg}(K(\mbb H\otimes_{\mbb Z} -)) \tto  \mscr O_{\DAlg}(K(-))\simeq C^*_{\sing}(K(-),\mbb Z)
	$$
	of functors $\Perf_{\mbb Z}^{\le 0,\op}\ra \DAlg(\mbb Z)$. 
	Moreover, its restriction to $\Lat^{\op}$ is identified with the transformation $\Map_{\mr{poly}}(-,\mbb Z)\ra \Map_\Set(-,\mbb Z)$, and (since both functors sends finite geometric realizations to totalizations) this way it coincides with the transformation \Cref{eq:transformation from Bin} (or rather its composition with the forgetful functor $\DBinAlg\ra \DAlg(\mbb Z)$). In particular, from \Cref{prop: description of free coconnective binomial algebras} we then get the following:
	\begin{cor}[{\cite[Corollary 3.4]{Toen_geo}}]
		Let $M\in \Perf_{\mbb Z}^{\le -1,\op}$. Then the map $\ul M\ra \mbb H\otimes_{\mbb Z} M$ provided by tensoring $i\colon \ul{\mbb Z}\ra \mbb H$ with $M$ induces an equivalence of derived commutative algebras
		$$
		 \mscr O_{\DAlg}(K(\mbb H\otimes_{\mbb Z} M)) \xra{\sim}  \mscr O_{\DAlg}(K(\ul M)).
		$$
	\end{cor}	
	
\end{rem}	

\subsection{Free derived binomial rings in other degrees}\label{ssec:free derived binomial rings in other degrees}

In this section we finish the computation of free binomial algebras in the remaining cohomological degrees. We will first deal with the connective part, where the description is surprisingly easy, and then describe what happens in the remaining cases given by torsion abelian groups put in cohomological degree 1. 

Let us start with the following lemma:

\begin{lem}\label{lem:LSym and LBin are equivalent for Q-vector space}
	For any $M\in {\mscr D}(\mbb Z)$ the map of derived commutative rings
	$$
	\mathbf L\Sym_{\mbb Z}(M\otimes \mbb Q) \ra \mathbf L\Bin(M\otimes \mbb Q)
	$$
	is an equivalence.
\end{lem}	

\begin{proof}
The above map is filtered with respect to exhaustive filtrations, so it is enough to show an analogous statement for $\mathbf L\Sym_{\mbb Z}^{\le n}(M\otimes \mbb Q) \ra \mathbf L\Bin^{\le n}(M\otimes \mbb Q)$. Both sides are $n$-excisive functors in $M$ that commute with sifted colimits, so it is enough to check that the map is an equivalence for $M\in\Lat$. This follows from \Cref{lem:map from Sym to Bin in the case of a vector space}.
\end{proof}	

Using the identification in \Cref{lem:LSym and LBin are equivalent for Q-vector space}, for any $M\in {\mscr D}(\mbb Z)$, $M\ra M\otimes \mbb Q$ induces a natural map 
$$
\alpha\colon \mathbf L\Bin(M) \ra \mathbf L\Sym_{\mbb Z}(M\otimes \mbb Q).
$$Quite surprisingly, it turns out to be an equivalence for any $M\in {\mscr D}(\mbb Z)^{\le -1}$:

\begin{prop}\label{prop:free binomial algebra in strictly cocnnective setting}
	Let $M\in {\mscr D}(\mbb Z)^{\le -1}$. Then the map $\alpha$ induces an equivalence of derived commutative algebras
	$$
	\alpha\colon \mathbf L\Bin(M) \xra{\sim} \mathbf L\Sym_{\mbb Z}(M\otimes \mbb Q).
	$$
\end{prop}	

\begin{proof}
Both $\mathbf L\Bin$ (or rather $\mathbf L\Bin^{\circ'}\!\colon {\mscr D}(\mbb Z)\ra \DAlg(\mbb Z)$ in the notations of \Cref{cor:forgetful functor to derived algebras commutes with all colimits}) and $ \mathbf L\Sym_{\mbb Z}\colon {\mscr D}(\mbb Z)\ra \DAlg(\mbb Z)$ commute with all colimits. Since ${\mscr D}(\mbb Z)^{\le -1}$ is generated under colimits by $M=\mbb Z[1]$ it would be enough to check the statement in this case. 

Note that the functor ${\mscr D}(\mbb Z)\ra {\mscr D}(\mbb Q)\times \prod_{p}{\mscr D}(\mbb F_p)$ given by
$$
N\mapsto N_{\mbb Q} \times \prod_{p} (N\otimes \mbb F_p)
$$
is conservative. Thus to check that $\alpha$ induces an equivalence it is enough to check this after tensoring with $\mbb Q$ and modulo each prime $p$. For tensor with $\mbb Q$ we have $\mathbf L\Bin(M)\otimes \mbb Q\simeq \mathbf L\Sym_{\mbb Z}(M)\otimes \mbb Q$ for any $M\in {\mscr D}(\mbb Z)$ by the end of \Cref{rem:sym vs bin tensor Q}, and $\mathbf L\Sym_{\mbb Z}(M)\otimes \mbb Q\simeq \mathbf L\Sym_{\mbb Z}(M\otimes \mbb Q)\otimes \mbb Q$ (both functors are direct sums of $n$-excisive sifted colimit commuting and agree on $M\in \Lat$). Modulo $p$ we have the following: 
\begin{lem} The natural map
	$\mbb F_p\ra \mathbf L\Bin(\mbb Z[1])\otimes\mbb F_p$ (from the initial derived commutative ring over $\mbb F_p$) is an equivalence.
\end{lem}	
\begin{proof}
	Note that by \Cref{rem:shift as Bar-construction} one can compute $\mathbf L\Bin(\mbb Z[1])$ as $\mbb Z\otimes_{\Bin(\mbb Z)}^{\mbb L}\mbb Z$. For the reduction mod $p$ we have 
	$$
	\mathbf L\Bin(\mbb Z[1])\otimes\mbb F_p \simeq \mbb F_p\otimes^{\mbb L}_{\Bin(\mbb Z)/p} \mbb F_p.
	$$
Note that by \Cref{rem:the map G_a --> H modulo p} the algebra $\Bin(\mbb Z)/p$ is identified\footnote{As in \Cref{rem:the map G_a --> H modulo p} the affine scheme $\Spec(\Bin(\mbb Z)/p)$ is identified with the profinite constant scheme  $\underline{\mathbb Z}_p:= \lim_n \underline{\mbb Z/p^n}$. It is in fact an isomorphism of group schemes.} with $\colim_{n}\Map_\Set(\mbb Z/p^n,\mbb F_p)$ with maps induced by reductions $\ldots\ra\mbb Z/p^n\ra \mbb Z/p^{n-1}\ra \ldots\ra \mbb Z/p$. The restriction of homomorphism $\Bin(\mbb Z)/p\ra \mbb F_p$ to $\Map(\mbb Z/p^n,\mbb F_p)\subset \Bin(\mbb Z)/p$ is given by the pull-back under the embedding $0\hookrightarrow \mbb Z/p^n$. Note that the latter map is an open embedding (if we consider the above groups as constant group schemes over $\mbb F_p$) and thus $\mbb F_p\simeq \Map_\Set(0,\mbb F_p)$ is a flat $\Map_\Set(\mbb Z/p^n,\mbb F_p)$-module. By passing to colimit over $n$ it follows that $\mbb F_p$ is a flat $\Bin(\mbb Z)/p$-module. Thus we get
$$
\mbb F_p\otimes^{\mbb L}_{\Bin(\mbb Z)/p} \mbb F_p\simeq \mbb F_p\otimes_{\Bin(\mbb Z)/p} \mbb F_p\simeq \mbb F_p,
$$
as desired. Geometrically, this corresponds to the fiber square 
$$
	\xymatrix{\{0\}\ar[d]\ar[r]&\{0\}\ar[d]\\
	\{0\}\ar[r] &\underline{\mathbb Z}_p,
}
$$
of affine schemes. 
\end{proof}	
It is also true that the map $\mbb F_p\ra \mathbf L\Sym_{\mbb Z}(\mbb Z[1]\otimes \mbb Q)\otimes \mbb F_p$ is an equivalence: indeed, $\mbb F_p$ maps isomorphically to $\mathbf L\Sym_{\mbb Z}^0\otimes \mbb F_p$, while $\mathbf L\Sym_{\mbb Z}^i(\bZ[1]\otimes \mbb Q)$ is a $\mbb Q$-module for $i>0$, and so
$\mathbf L\Sym_{\mbb Z}^i(\bZ[1]\otimes \mbb Q)\otimes \mbb F_p\simeq 0$. Thus, from this and the lemma above we get that the map 
$$
\mathbf L\Bin(\mbb Z[1])\otimes\mbb F_p \ra \mathbf L\Sym_{\mbb Z}(\mbb Z[1])\otimes\mbb F_p
$$
is an equivalence for every $p$, and so $\alpha\colon \mathbf L\Bin(\mbb Z[1])\ra  \mathbf L\Sym_{\mbb Z}(\mbb Q[1])$ is also an equivalence.

\end{proof}	 

\begin{rem}\label{rem:free binomial algebra on torsion module in positive degrees is 0}
	In particular, we get that if $A$ is a torsion abelian group (meaning $A\otimes \mbb Q\simeq 0$) one has $\mathbf L\Bin(A[n])=\mbb Z$ for $n\ge 1$.
\end{rem}	

\begin{rem}\label{rem:connective isomorphism is not filtered}
We note that the isomorphism in \Cref{prop:free binomial algebra in strictly cocnnective setting} is \textit{not} a filtered isomorphism. Indeed, let's say $M\in \Perf_{\mbb Z}^{\le -1}$; the associated graded of $\mathbf L\Bin^{\le *}(M)$ is given by $\mathbf L\Gamma^*(M)$ and each graded component is again a perfect module. However, on the right hand side the associated graded pieces are complexes of $\mathbb Q$-vector spaces. This shows that in the case $M$ is 1-connective the spectral sequence converging from cohomology of  $\mathbf L\Gamma^*(M)$ to cohomology of $\mathbf L\Bin(M)$ is very non-generate: in fact it can not stabilize on any finite page. This is in contrast with the \textit{strictly coconnective} case where, as we will show in the sequel, the analogous spectral sequence always degenerates. 
\end{rem}	

\noindent In fact, the statement in \Cref{rem:free binomial algebra on torsion module in positive degrees is 0} can be extended slightly to the right in the cohomological direction:

\begin{prop}\label{prop:Bin of torsion abelian group}
	Let $A\in \Ab$ be a torsion abelian group (meaning $A\otimes \mbb Q\simeq 0$). Then the natural map 
	$$
	\mbb Z\ra \mathbf L\Bin(A)
	$$
	is an equivalence.
\end{prop}	

\begin{proof}

The proof is somewhat similar to \Cref{prop:free binomial algebra in strictly cocnnective setting}. It is enough to check that the map is an equivalence after tensoring with $\mbb Q$ and modulo each prime $p$. 	It is enough to consider the case $A=\mbb Z/p^k$. We have $\mbb Z/p^k=\cofib(\mbb Z\xra{\cdot p^k}\mbb Z)$ and so 
	$$
	\mathbf L\Bin(\mbb Z/p^k)\simeq \Bin(\mbb Z)\otimes_{\Bin(\mbb Z)}^{\mbb L}\mbb Z
	$$ 
	where the map $\Bin(\mbb Z)\ra \Bin(\mbb Z)$ is induced by multiplication by $p^k$ on $\mbb Z$. In terms of the group scheme $\mbb H\coloneqq \Spec(\Bin(\mbb Z))$ the latter map is the pull-back with respect to the multiplication by $p^k$ map $[p^k]\colon \mbb H\xra{} \mbb H$. Over $\mbb Q$ we have $\mbb H_{\mbb Q}\simeq \mbb G_{a,\mbb Q}$ and multiplication by $p^k$ is an isomorphism, so
	$
	\mathbf L\Bin(\mbb Z/p^k)\otimes\mbb Q\simeq \mbb Q.
	$
	Modulo $\ell\neq p$ we have $\mbb H_{\mbb F_\ell}=\Spec( \Bin(\mbb Z)\otimes \mathbb F_\ell)\simeq \lim_n\ul{\mbb Z/\ell^n}$ and multiplication by $p^k$ is again an isomorphism. So,
	$
	\mathbf L\Bin(\mbb Z/p^k)\otimes\mbb F_\ell\simeq \mbb F_\ell.
	$

	It remains to understand the case $\ell=p$. We have $\mbb H_{\mbb F_p}=\Spec (\Bin(\mbb Z)\otimes \mbb F_p)\simeq \lim_n\ul{\mbb Z/p^n}$. Multiplication by $p^k$ is no longer an isomorphism, but it is an open embedding: indeed, via the above identification we have a fiber square 
	$$
	\xymatrix{\mbb H_{\mbb F_p}\ar[r]^{\cdot p^k}\ar[d]&\mbb H_{\mbb F_p}\ar@{->>}[d]\\
		\{0\}\ar[r] &\ul{\mbb Z/p^k},
	}
	$$
	so the map $\Bin(\mbb Z)\otimes \mbb F_p\rightarrow \Bin(\mbb Z)\otimes \mbb F_p$ is the pushout of the surjection $\Map_\Set(\mbb Z/p^k,\mbb F_p)\ra \mbb F_p$ (induced by the open embedding $\{0\}\hookrightarrow \ul{\mbb Z/p^k}$ of constant $\mbb F_p$-group schemes) with respect to the map $\Map_\Set(\mbb Z/p^k,\mbb F_p)\ra \Bin(\mbb Z)\otimes \mbb F_p$. In particular, it is flat and surjective at the same moment. From this and \Cref{rem:shift as Bar-construction} we get 
	$$
	\mathbf L\Bin(\mbb Z/p^k)\otimes\mbb F_p\simeq (\Bin(\mbb Z)\otimes \mbb F_p)\otimes_{(\Bin(\mbb Z)\otimes \mbb F_p)}\mbb F_p\simeq \mbb F_p,
	$$
	and then also the equivalence $\mathbb Z\xra{\sim}\mathbf L\Bin(\mbb Z/p^k)$.
\end{proof}

\begin{rem}
	As in \Cref{rem:connective isomorphism is not filtered} we note that the isomorphism in \Cref{prop:Bin of torsion abelian group} is very far from being a filtered isomorphism (is we endow left hand side with the constant $\mbb N$-indexed filtration). Indeed $\gr^1$ on the left is $0$, and is equal to $A$ on the right. In fact we don't really know what the filtered pieces $\mathbf L\Bin^{\le n}(A)$ are explicitly.
\end{rem}	
\begin{rem}\label{rem:bin for classical abelian group}
	For any abelian group $A$ we have a short exact sequence $0\ra \mr{Tors}(A)\ra A \ra A^{\mr{tf}}\ra 0$ where $\mr{Tors}(A)$ is the maximal torsion subgroup and $A^{\mr{tf}}$ is torsion free. On the level of derived categories we have $A^{\mr{tf}}\simeq \cofib(\mr{Tors}(A)\ra A)$, so by \Cref{rem:shift as Bar-construction} we get
	$$
	\Bin(A^{\mr{tf}})\simeq \mathbf L\Bin(A^{\mr{tf}})\simeq \mathbf L\Bin(A)\otimes_{\mathbf L\Bin(\mr{Tors}(A))}\mbb Z.
	$$
	However, by \Cref{prop:Bin of torsion abelian group} $\mathbf L\Bin(\mr{Tors}(A))\simeq \mbb Z$. This shows that there is a natural equivalence
	$$
	 \mathbf L\Bin(A)\simeq \Bin(A^{\mr{tf}})
	$$
	for any classical abelian group $A$; in particular $ \mathbf L\Bin(A)$ only depends on $A^{\mr{tf}}$.
\end{rem}

\begin{cor}\label{cor:homology of binomial ring} Let $B\in \DBinAlg$ be a derived binomial ring. Then $H^i(B)$ is a $\mbb Q$-vector space for $i <0$ and $H^0(B)$ is a torsion-free abelian group.
\end{cor}	
\begin{proof}
	It will be enough to show that $H^i(B)$ is torsion-free for $i\le 0$ and divisible for $i<0$. For the first assertion let $x\in H^i(B)$ be a class with $nx=0$. Then $x$ defines a map $ \mathbf L\Bin (\mbb Z/n [-i]) \ra B$. However, by \Cref{rem:free binomial algebra on torsion module in positive degrees is 0} and \Cref{prop:Bin of torsion abelian group} we get that $ \mathbf L\Bin (\mbb Z/n [-i]) \simeq \mbb Z$, and so the natural map $\mbb Z/n [-i] \ra  \mathbf L\Bin (\mbb Z/n [-i])$ is homotopic to the zero map. We get that the map $\mbb Z/n [-i] \ra B$ in $\mscr D(\mbb Z)$ that classifies $x$ is also 0, and so $x$ itself is 0. The divisibility assertion follows similarly: indeed, by \Cref{prop:free binomial algebra in strictly cocnnective setting} for $i< 0$ we have $ \mathbf L\Bin(\mbb Z[-i])\simeq  \mathbf L\Sym(\mbb Q[-i])$, and so the natural map $\mbb Z[i] \ra B$ in $\mscr D(\mbb Z)$ that classifies an element $x\in H^i(B)$ automatically extends to a map $\mbb Q[i] \ra B$. 
\end{proof}	

\begin{rem}
	The fact that $H^i(B)$ for $i<0$ are $\mbb Q$-vector spaces for a connective binomial ring can also be deduced formally from the fact that a classical split square zero extension $\mbb Z\oplus M$ in (classical) commutative rings is a binomial ring if and only if $M$ is a $\mbb Q$-vector space\footnote{For the only if part, note that $M$ is torsion-free by \Cref{rem: binomial rings are torsion free}, and then for any $m\in M$ such that $m^2=0$ we have ${m\choose n}= \frac{1}{n}\cdot m$, so if $\mbb Z \oplus M$ is binomially closed then $\frac{1}{n}\cdot m\in M$ for any $n$.}.
	
	Indeed, expanding on  \Cref{rem: binomial rings are torsion free}, one can identify the subcategory $\DBinAlg^\mr{cn}$ of connective derived binomial rings with the \textit{animation} (see \cite[Section 5.1.4]{Cesn-Sch}) of the category $\BinAlg$ of classical binomial rings; specifically, $\DBinAlg^\mr{cn}$ is the category of product-commuting functors from the opposite category $\BinAlg^{\mr{free},\op}$ of free finitely generated classical binomial rings to $\Spc$ (\textit{anima} in terminology of \cite{Cesn-Sch}). As such, any $B\in \DBinAlg^\mr{cn}$ comes with a natural Postnikov tower $B\ra \ldots \ra \tau_{\le n}B \ra \ldots \ra \tau_{\le 0}B$ induced by the Postnikov tower construction on $\Spc$ (also see \cite[Section 5.1.4]{Cesn-Sch}); it agrees with the Postnikov tower on the level of underlying animated abelian group. For a given $n$ one can consider the fiber product $B_n:=\tau_{\le n}B \times_{\tau_{\le n-1} B} \mbb Z$, which is naturally on object in  $\DBinAlg^\mr{cn}_{\mbb Z/ \! / \mbb Z}$ via the unit map $\mbb Z \ra B_n$ and the projection $B_n:=\tau_{\le n}B \times_{\tau_{\le n-1} B} \mbb Z\ra \mbb Z$.  Note that the underlying animated abelian group of $B_n$ is given by $\mbb Z \oplus \pi_n(B)[n]$ (here $\pi_n(B)=H^{-n}(B)$). Taking the $n$-fold loop space $\Omega^nB_n \in \DBinAlg^\mr{cn}_{\mbb Z/ \! / \mbb Z}$ we get an $\mbb E_n$-monoid in classical\footnote{Note that the underlying abelian group of $\Omega^nB_n $ is $\mbb Z\oplus \pi_n(B)[n][-n]\simeq \mbb Z\oplus \pi_n(B)$.} (coaugmented) binomial rings $\BinAlg_{\mbb Z/ \! / \mbb Z}$. Forgetting the $\mbb E_n$-monoid structure on $\Omega^nB_n$ to $\mbb E_1$ we get an associative monoid object of $\BinAlg_{\mbb Z/ \! / \mbb Z}$ and using Quillen's identification \cite[Proposition 1.4]{quillen1970} this gives a split square zero extension of $\mbb Z$ of the form $\mbb Z \oplus \pi_n(B)$ which also a binomial ring. By the above this forces $\pi_n(B)$ to be a $\mbb Q$-vector space.
	
	Let us also remark that even more abstractly one can develop the theory of derivations, cotangent complex, split and non-split square extensions for animated binomial rings, as e.g. in \cite[Sections 5.1.8-9]{Cesn-Sch}. The extension ideal of an animated split square zero  in this setting will automatically be a $\mbb Q$-vector space, since it is true for the classical split square zero extensions. Then one can show (similarly to \cite[Example 5.1.10(3)]{Cesn-Sch}) that $\tau_{\le n}B$ is a (non-split) square zero extension of $\tau_{\le n-1} B$, and so the corresponding extension ideal (given by $\pi_n(B)[n]$) is automatically a $\mbb Q$-vector space. More generally, one can show that $\DBinAlg^\mr{cn}$ is identified with the  full subcategory of $\mr{Ani}(\mr{Ring})$ spanned by those animated commutative rings $A$ such that $\pi_0(A)$ is a classical binomial ring and $\pi_i(A)$ is a $\mbb Q$-vector space for $i>0$. We leave it to the interested reader to fill in the omitted details.
\end{rem}	

\subsection{$\mbf L\Bin(A[-1])$ for torsion $A$}\label{ssec:LBin(A[-1])}
It remains to understand the structure of $\mathbf L\Bin(A[-1])$ where $A$ is a torsion abelian group.  Let $\Ab^{\fin}_p\subset \Ab$ denote the subcategory of finite abelian groups of $p$-power order. Using that any torsion abelian group is a filtered colimit of finite abelian groups, and that $\mathbf L\Bin$ sends direct sums to (derived) tensor products, we can reduce to the case $A\in \Ab^{\fin}_p$ for some prime $p$.

Recall the following description of $\mbb H$ modulo $p^d$, generalizing one in \Cref{rem:the map G_a --> H modulo p}:

\begin{lem}[{\cite[Corollaire 2.3]{Toen_geo}}]\label{lem:H modulo p^n}
	For any $n\ge 1$ there is a natural isomorphism of affine group schemes over $\mbb Z/p^n$
		$$
		\mbb H_{\mbb Z/p^n}\xra{\sim} \ul{\mbb Z_p}\coloneqq \lim_{i\ge 0}\ul{\mbb Z/p^i}
		$$
		such that the composition with the map $i\colon \ul{\mbb Z} \ra \mbb H$ is the natural map $\ul{\mbb Z}\ra \ul {\mbb Z_p}$.
\end{lem}	
\begin{proof} Note that
	\begin{equation}\label{eq:congruence}
	(t-1)^{p^{n-1+s}}\equiv (t^{p^s}-1)^{p^{n-1}} \pmod {p^{n}}.
	\end{equation}
	We see that for any $s$, the map $\mu_{p^s,\mbb Z/p^n}\ \ra \mbb G_{m,\mbb Z/p^n}$ (on the level of functions given by the projection $\mbb Z/p^n[t,t^{-1}] \surj \mbb Z/p^n[t,t^{-1}]/(t^{p^s}-1)$) factors through the formal completion $\widehat{\mbb G}_{m,\mbb Z/p^n}\subset \mbb G_{m,\mbb Z/p^n}$. For any $s\ge 1$ we get maps $\mu_{p^s,\mbb Z/p^n}\ra \widehat{\mbb G}_{m,\mbb Z/p^n} \ra \mbb G_{m,\mbb Z/p^n}$ and, passing to Cartier duals, also maps $\ul{\mbb Z} \ra \mbb H_{\mbb Z/p^n} \ra \ul{\mbb Z/p^s}$ with the composition given by the natural projection.

	Moreover, we claim that the induced map $ \colim_{s\ge 1} (\mu_{p^s,\mbb Z/p^n})\ra \widehat{\mbb G}_{m,\mbb Z/p^n}$ is an isomorphism. Indeed, we need to check that the induced map
	 $$
	 \lim_s\mbb Z/p^n[t,t^{-1}]/(t-1)^{p^s} \ra \lim_s\mbb Z/p^n[t,t^{-1}]/(t^{p^s}-1)
	 $$
	 is an isomorphism.
	This follows from the \Cref{eq:congruence} and the embedding of ideals $
		(t^{p^{n+s-1}}-1 ) \subset (t-1)^{p^s} \pmod {p^n}$, which one can show by decomposing $t$ as $(t-1)+1$ and keeping track of $p$-adic valuations of the binomial coefficients appearing.
\end{proof}

 \begin{notation}\label{not:Pontryagin dual} For $A\in \Ab^{\fin}_p$, let us denote by $A^D\coloneqq \Hom_{\mbb Z}(A,\mbb Q/\mbb Z)$ the Pontryagin dual of $A$. Using the short exact sequence $ 0 \ra \mbb Z\ra \mbb Q \ra \mbb Q/\mbb Z\ra 0$ one also sees that $A^D\simeq \Ext^1_\Ab(A,\mbb Z)\simeq \Hom_{{\mscr D}(\mbb Z)}(A,\mbb Z)[1]\simeq A^\vee[1]$. 
 \end{notation}

 For $M\in {\mscr D}(\mbb Z)$ let us also denote by $\mbf L^i\Bin(M)\in \Ab$ the $i$-th cohomology of $\mbf L\Bin(M)$: $$\mbf L^i\Bin(M)\coloneqq H^i(\mbf L\Bin(M)),$$
and similarly for other non-abelian derived functors.
\begin{rem}\label{rem:LGamma of finite group is finite}
 	For any $k\in \mbb Z$ multiplication by $k$ on $A$ induces a map on $\mbf L \Gamma_{\mbb Z}^n(A[-1])$ which is multiplication by $k^n$. In particular, if $A$ is finite, multiplication by $|A|$ on $A$ induces a zero map on $\mbf L \Gamma_{\mbb Z}^n(A[-1])$; consequently, $\mbf L^i \Gamma_{\mbb Z}^n(A[-1])$ is a finite abelian group killed by $|A|^n$ for any $i$. 
 	
 	This way, if $A\in \Ab^{\fin}_p$, we get that $\mbf L^i \Gamma_{\mbb Z}^n(A[-1])\in \Ab^{\fin}_p$ for any $n>0$ and $i\in \mbb Z$. More generally, since $\mbf L^i\Bin^{\le n}(A[-1])$ has a finite filtration with associated graded pieces given by $\mbf L^i \Gamma_{\mbb Z}^n(A[-1])$ we get that $\mbf L^i\Bin^{\le n}(A[-1])$ is also in $\Ab^{\fin}_p$ if $i\neq 0$. Also, by \Cref{ex:Bin^n in low degrees}(\ref{exsub:Bin splits as zero and rest}), $\mbf L^0\Bin^{\le n}(A[-1])$ splits as 
 	$$\mbf L^0\Bin^{\le n}(A[-1])\simeq \mbb Z \oplus \mbf L^0\ol{\Bin^{\le n}}(A[-1]),$$ where the second summand is in $\Ab^{\fin}_p$.
 \end{rem}	
 
\begin{lem} \label{lem:L^1Bin for A being p-torsion} Let $A\in \Ab^{\fin}_p$ and fix $n\ge 1$. Then
\begin{enumerate}
			\item The natural map 
	$$
	i\colon \ul{A} \ra \mbb H_{\mbb Z/p^n}\otimes A
	$$
	in $\Ab(\Stk_{/\mbb Z/p^n})$ given by tensoring $i\colon \ul{\mbb Z}\ra \mbb H_{\mbb Z/p^n}$ with $A$ is an equivalence.
	
	\item There are natural equivalences
	$$
	\mbf L\Bin(A[-1])\otimes_{\mbb Z}^{\mbb L} \mbb Z/p^n \simeq \mscr O_{\DAlg}(K(\mbb H_{\mbb Z/p^n}\otimes A^D)) \simeq \Map_\Set(A^D,\mbb Z/p^n)
	$$
	of derived commutative rings.
	\item $\mbf L^i\Bin(A[-1])=0$ unless $i=0$ or 1; $\mbf L^0\Bin(A[-1])=\mbb Z$ and $\mbf L^1\Bin(A[-1])$ is torsion and $p$-divisible.
\end{enumerate}
\end{lem}	
\begin{proof}
		Part 1 follows from the repleteness of fpqc-topos \cite{Toen_geo}. Namely, we know that $\mbb H_{\mbb Z/p^n}\simeq \lim_{i}\ul{\mbb Z/p^i}$ as affine group schemes. By \cite[Proposition B.1]{Toen_geo} homotopy group sheaves (see \Cref{constr:higher stacks}) commute with sequential colimits and so it is also true that $\mbb H_{\mbb Z/p^n}\in \Ab(\Stk_{/{\mbb Z/p^n}})$ computes $\holim_i \ul{\mbb Z/p^i}\in \Ab(\Stk_{/{\mbb Z/p^n}})$. Since $A$ is finite it lies in $\Perf_{\mbb Z}$ and so we can move $-\otimes_{\mbb Z}^{\mbb L} A$ inside the limit
	$$
	\mbb H_{\mbb Z/p^n}\otimes_{\mbb Z}^{\mbb L} A \simeq \holim_i (\ul{\mbb Z/p^i}\otimes_{\mbb Z}^{\mbb L}A)\in \Ab(\Stk_{/\mbb Z/p^n}).
	$$
	In the last limit we can restrict to $i\gg 0$, where both $\pi_0(\ul{\mbb Z/p^i}\otimes_{\mbb Z}^{\mbb L}A)\simeq A/p^i$ and $\pi_1(\ul{\mbb Z/p^i}\otimes_{\mbb Z}^{\mbb L}A) =A[p^i]$\footnote{Here, abusing notation, we denote by $[p^i]$ the elementary $p^i$-torsion, and not homological shift by $p^i$.} will be given by $A$. The maps in the diagram are identity on $\pi_0$ and multiplication by $p$ on $\pi_1$, so we get 
	$$
	\pi_0(\holim_i (\ul{\mbb Z/p^i}\otimes_{\mbb Z}^{\mbb L} A))\simeq \holim(\ldots \xra{\id} \ul A \xra{\id} \ul A \xra{\id} \ul A)\simeq \ul A
	$$ 
	and 
	$$
	\pi_1(\holim_i (\ul{\mbb Z/p^i}\otimes_{\mbb Z}^{\mbb L} A))\simeq \holim(\ldots \xra{\cdot p}  \ul A \xra{\cdot p} \ul A \xra{\cdot p} \ul A)\simeq 0
	$$
	(and all $\pi_j$ for $j>1$ are equal to 0). Thus we get that the natural map 
	$$
	\mbb H_{\mbb Z/p^n}\otimes_{\mbb Z}^{\mbb L} A\ra \pi_0(\mbb H_{\mbb Z/p^n}\otimes_{\mbb Z}^{\mbb L} A)\simeq \ul A
	$$
	is an equivalence. The map 
	$
	i\colon \ul  A\ra \mbb H_{\mbb Z/p^n}\otimes_{\mbb Z}^{\mbb L} A
	$ is given by $\ul  A\simeq \ul{\mbb Z}\otimes_{\mbb Z}A \ra \holim_{i\gg 0} (\ul{\mbb Z/p^i}\otimes_{\mbb Z}^{\mbb L} A)$ which maps identically to $\pi_0$ of any term in limit, and so we get that it is also an equivalence.
	
	The first isomorphisms in Part 2 follows from \Cref{rem:LBin as cohomology of a stack} and identification $A^D\simeq A^\vee[1]$ from \Cref{not:Pontryagin dual}. The second isomorphism is obtained from Part 1
	by applying $\mscr O_{\DAlg}(K(-))$.
	
	For Part 3, note that $\mbf L^i\Bin(A[-1]))\simeq \colim_{n} \mbf L^i\Bin^{\le n}(A[-1]))$, since cohomology commutes with filtered colimits. This way, from \Cref{rem:LGamma of finite group is finite} we get that $\mbf L^i\Bin(A[-1]))$ is $p$-torsion if $i\neq 0$ and is the direct sum of $\mbb Z$ and a $p$-torsion abelian group if $i=0$. By the universal coefficients formula, for any $i\in \mbb Z$ we have 
	\begin{equation}\label{eq:ses for reduction of Bin}
	0\ra (\mbf L^i\Bin(A[-1]))/p \ra H^i(\mbf L\Bin(A[-1])\otimes \mbb F_p) \ra (\mbf L^{i+1}\Bin(A[-1])[p] \ra 0.
	\end{equation}
	On the other hand, by Part 2 we know that $H^i(\mbf L\Bin(A[-1])\otimes \mbb F_p)\simeq 0$ for $i\neq 0$. Since a $p$-torsion abelian group $M$ is 0 if and only if $M[p]$ is we get that $\mbf L^{i}\Bin(A[-1])\simeq 0$ for $i\neq 0,1$. We also get that $(\mbf L^1\Bin(A[-1]))/p\simeq 0$, which means that $\mbf L^1\Bin(A[-1])$ is $p$-divisible, and that $(\mbf L^{0}\Bin(A[-1])[p]\simeq 0$, which means that $(\mbf L^{0}\ol{\Bin}(A[-1])\simeq 0$ and $\mbf L^0\Bin(A[-1])\simeq \mbb Z$.
\end{proof}	

\begin{rem}\label{rem:L^1Bin for A p-torsion}
	For $i=0$ by Part 2 of \Cref{lem:L^1Bin for A being p-torsion} we get that short exact sequence \Cref{eq:ses for reduction of Bin} takes the form 
	$$0\ra \mbb F_p\ra \Map_{\Set}(A,\mbb F_p)\ra (\mbf L^1\Bin(A[-1]))[p]\ra 0,$$
	which gives a non-canonical identification $(\mbf L^1\Bin(A[-1]))[p]\simeq \mbb F_p^{\oplus |A|-1}$. Since $\mbf L^1\Bin(A[-1])$ is $p$-divisible, this further gives a non-canonical identification $\mbf L^1\Bin(A[-1])\simeq  (\bQ_p/\mbb Z_p)^{\oplus |A|-1}$.
\end{rem}	

\begin{cor}
	For any finite abelian group $A\in \Ab^\fin$ we have $\mbf L^i\Bin(A[-1])=0$ unless $i=0$ or 1. $\mbf L^0\Bin(A[-1])\simeq \mbb Z$ and $\mbf L^1\Bin(A[-1])$ is torsion and divisible. 
\end{cor}	
\begin{proof}
	Let $A\simeq \oplus_p A_p$ be the decomposition of $A$ as a direct sum of its $p$-primary components. We have an equivalence 
	$$
	\mbf L\Bin(A[-1])\simeq \otimes_{p}^{\mbb L} \ \! \mbf L \Bin(A_p[-1]).
	$$ 
	By the K\"unneth formula, since $\mbb Q_p/\mbb Z_p\otimes^{\mbb L} \mbb Q_\ell/\mbb Z_\ell\simeq 0$ for any primes $p\neq \ell$, and \Cref{lem:L^1Bin for A being p-torsion} we get that $\mbf L^i\Bin(A[-1])=0$ unless $i=0$ or 1 and  $\mbf L^1\Bin(A[-1])\simeq \oplus_p \ \!\mbf L^1\Bin(A_p[-1])$, while $\mbf L^0\Bin(A[-1])\simeq \mbb Z$. The statement follows.
\end{proof}	
\begin{rem}
	As in \Cref{rem:L^1Bin for A p-torsion}, in terms of the decomposition $A\simeq \oplus_p A_p$ one can non-canonically identify 
	$$
	\mbf L^1\Bin(A[-1])\simeq \oplus_p \ \! (\bQ_p/\mbb Z_p)^{\oplus |A_p|-1}. 
	$$
	In \Cref{cor:canonical formula for L^1Bin} we will give a more canonical (but less explicit) formula for $\mbf L^1\Bin(A[-1])$.
\end{rem}	

\begin{rem} Let $\Phi\coloneqq \mbf L^1\Bin(-[-1])\colon \Ab^\fin \ra \Ab^{\mr{tors}}$. Note that via the K\"unneth formula, for any $A_1,A_2\in \Ab^{\fin}$ one gets an identification
	$$
	\Phi(A_1\oplus A_2)\simeq \Phi(A_1)\oplus \Phi(A_2) \oplus \Tor^1(\Phi(A_1),\Phi(A_2)).
	$$
	We see that the functor $\Phi\coloneqq \mbf L^1\Bin(-[-1])\colon \Ab^\fin \ra \Ab^{\mr{tors}}$ 
	behaves as the "torsion multiplicative formal group law".
\end{rem}	

\paragraph{Description of $\mbf L^1\Bin^{\le t}(A[-1])$.}

\begin{construction}

Consider the functor $\Phi^t(A)\coloneqq \mbf L^1\Bin^{\le t}(A[-1])$ as a functor
$\Ab^\fin \ra \Ab$. By the \Cref{rem:LGamma of finite group is finite} it in fact takes values in $\Ab^\fin\subset \Ab$.
By universal coefficients formula we have a short exact sequence:
\begin{equation}\label{eq:ses for Phi dual}0\ra \Ext^1_{\Ab}(\Phi^t(A),\mbb{Z})
\ra \Hom_{{\mscr D}(\mbb Z)}^0(\mbf L\Bin^{\le t}(A[-1]),\mbb{Z})
\ra \mbb{Z}=\Hom_{\Ab}(\mbf L^0\Bin^{\le t}A[-1],\mbb Z)\ra 
0.
\end{equation}
We shall describe the functor
$\Psi^t(A)\coloneqq \Hom_{{\mscr D}(\mbb Z)}^0(\mbf L\Bin^{\le t}(A[-1]),\mbb{Z})=
H_0(\mbf L\Bin^{\le t}(A[-1])^\vee)$.
It is a contravariant functor from $\Ab^\fin$ to $\Ab$. By the discussion in \Cref{not:Pontryagin dual} the left term in \Cref{eq:ses for Phi dual} is naturally identified with $(\Phi^t(A))^D$, and so \Cref{eq:ses for Phi dual} gives a short exact sequence of functors
$$
0\ra (\Phi^t)^D \ra \Psi^t \ra \ul{\bZ}\ra 0
$$
from ${({\Ab^{\fin}})}^{\op}$ to $\Ab$. Thus knowing $\Psi^t(A)$ with the map $\Psi^t(A)\ra \mbb Z$ would be enough to reconstruct $\Phi^t(A)$. 
\end{construction}

\begin{thm}\label{thm:filtered pieces of L^1Bin}
	There is a functorial in $A\in \Ab^\fin$ isomorphism
	$$\Psi^t(A)\simeq \mbb{Z}[A^D]/J_A^{t+1},$$
	where $J_A\subset \mbb{Z}[A^D]$ is the augmentation ideal.
	The natural morphism $\Psi^t(A)\ra \mbb{Z}$ is identified with
	the augmentation.
	Consequently, one gets a functorial in $A$ isomorphism
	$$\Phi^t(A)\simeq (J_A/J_A^{t+1})^D.$$
\end{thm}
\begin{proof}
	Pick a free resolution  $[L_0\xra{d}L_1]\sim A[-1]$.
	Recall that by \Cref{rem:description of polynomial maps} for a free abelian group $L$ there is 
	a natural isomorphism 
	$$\Bin^{\le t}(L)\coloneqq \Map_{\le t}(L^\vee,\mbb Z)\xra{\sim} (\mbb{Z}[L^\vee]/J_L^{t+1})^\vee,$$
	where $J_L\subset \mbb{Z}[L^\vee]$ is the augmentation ideal.
	Let $\mr{DK}^\bullet([L^0\xra{d}L^1])\simeq L_0\oplus L_1^{\oplus \bullet}$ denote the cosimplicial abelian group Dold-Kan corresponding to the 2-term complex $[L^0\xra{d}L^1]$; we have $A[-1]\simeq \Tot_{[\bullet]\in \Delta}(\mr{DK}^\bullet([L^0\xra{d}L^1]))$. By \Cref{lem:LBin_n commutes with finite totalizations}, $\mbf L\Bin^{\le t}$ commutes with finite totalizations and so we get
	$$\mbf L\Bin^{\le t} A[-1]\simeq
	\Tot_{[\bullet]\in \Delta} \mbf L\Bin^{\le t}\!\left(\mr{DK}^\bullet([L^0\xra{d}L^1])\right)\simeq 	\Tot_{[\bullet]\in \Delta} \Bin(L_0\oplus L_1^{\oplus \bullet}),$$
and then 
	$$\Psi^t(A)=\Hom_{{\mscr D}(\mbb Z)}^0(\mbf L\Bin^{\le t}A[-1],\mbb{Z})\simeq 
	\coker\left(f_0-f_1\colon \mbb{Z}[L_0^\vee\oplus L_1^\vee]/J_{L_0\oplus L_1}^{t+1}\ra
	\mbb{Z}[L_0^\vee]/J_{L_0}^{t+1}\right),$$
	where the maps $f_i$ are given by $f_0=i^\vee_0$ and $f_1=i^\vee_0\oplus d^\vee$ correspondingly (here $i_0\colon L_0\hookrightarrow L_0\oplus L_1$ is the natural embedding and $d^\vee\colon L_1^\vee \ra L_0^\vee$ is dual to the differential $d\colon L_0\ra L_1$). Both $f_0$ and $f_1$
	are homomorphisms of algebras.
	Let $u\in L_0^\vee$ and $v\in L_1^\vee$, and consider 
	the basis element $[u]\cdot [v]\in\mbb{Z}[L_0^\vee \oplus L_1^\vee]$.
	By definition, $f_0([u]\cdot [v])=[u]$ and $f_1([u]\cdot [v])=[u]\cdot [d^\vee(v)]$.
	We see that $\Im (f_0-f_1)\subset \mbb{Z}[L_0^\vee]/J_{L_0}^{t+1}$ is an ideal generated by 
	$1-[d^\vee(v)]$ for all $v\in L_1^\vee\subset \mbb{Z}[L_1^\vee]$.
	This gives
	$$\Psi^t(A)\xra{\sim}\mbb{Z}[L_0^\vee]/(J_{L_0}^{t+1}, 1-[d^\vee(v)])_{v\in L_1^\vee}\simeq   \mbb{Z}[L_0^\vee/L_1^\vee]/J_{L_0}^{t+1}.$$
	Finally, noting that $L_0^\vee/L_1^\vee\simeq \Ext^1(A,\mbb Z)\simeq A^D$ we get the first statement. 
	From the proof above one also sees that the augmentation map 
	$\mbb{Z}[A^D]/J_{A}^{t+1}\ra\mbb{Z}$
	is exactly  induced by the restriction $$\Hom^0_{{\mscr D}(\mbb Z)}(\mbf L\Bin^{\le t}(A[-1]),\mbb{Z})\ra
	\Hom_\Ab(\mbf L^0\Bin^{\le t}(A[-1]),\mbb{Z})=\mbb{Z}.$$
	Thus by \Cref{eq:ses for Phi dual} we can identify 
	$$
	 \Ext^1_{\Ab}(\Phi^t(A),\mbb{Z})\simeq \ker(\mbb{Z}[A^D]/J_{A}^{t+1}\ra\mbb{Z})\simeq J_A/J_{A}^{t+1},
	$$
	and then, passing to Pontryagin duals again, also
	$$
	\Phi^t(A)\simeq \Ext^1_{\Ab}(\Phi^t(A),\mbb{Z})^D\simeq (J_A/J_{A}^{t+1})^D.
	$$

\end{proof}

\begin{cor}\label{cor:canonical formula for L^1Bin}
	Let $A\in \Ab^\fin$. Then there is a natural in $A$ isomorphism 
	$$\mbf L^1\Bin(A[-1])\simeq \colim_{t\ge 1} (J_A/J_A^{t+1})^D,$$
	where $J_A\subset \mbb Z[A^D]$ is the augmentation ideal.
\end{cor}
\begin{proof}
	Since cohomology commutes with filtered colimits, by \Cref{thm:filtered pieces of L^1Bin} we get 
	$$
	\mbf L^1\Bin(A)[-1]\simeq \colim_{t\ge 1} \mbf L^1\Bin^{\le t}(A)[-1] \simeq  \colim_{t\ge 1} (J_A/J_A^{t+1})^D.
	$$
\end{proof}	

In the case $A=\mbb{Z}/p$ we get a particularly nice description.
Let $K\coloneqq \mbb Q_p(\zeta_p)$ where $\zeta_p^p=1$ is a primitive root of unity and let $\mc O_K\subset K$ denote the ring of integers. The element $\pi\coloneqq \zeta_p-1$ is a uniformizer of $\mc O_K$. Also note that $\Aut_\Ab(\bZ/p)\simeq (\bZ/p)^\times$ and that $(\bZ/p)^\times\simeq \Gal(K/\mbb Q_p)$.
\begin{cor}\label{cor:L^1Bin for Z/p}
	There is an isomorphism
	$$\mbf L^1\Bin^{\le t}(\mbb{Z}/p[-1])\simeq 
	((\pi)/(\pi^{t+1}))^D$$
	that is compatible with the filtrations and the action
	of $(\bZ/p)^\times$.
\end{cor}
\begin{proof} Let $L\coloneqq \mbb Z[(\mbb Z/p)^D]$; we can $(\mbb Z/p)^\times$-equivariantly identify $L$ with $\mbb Z[x]/(x^p-1)$, where $i\in (\mbb Z/p)^\times$ sends $x$ to $x^i$.  By \Cref{thm:filtered pieces of L^1Bin} $\mbf L^1\Bin^{\le t}(\mbb{Z}/p[-1])$ is identified with the Pontryagin dual of the quotient $(x-1) L /(x-1)^{t+1} L$. By \Cref{rem:LGamma of finite group is finite} this is a finite $p$-torsion group, so $$(x-1) L /(x-1)^{t+1} L \simeq ((x-1) L /(x-1)^{t+1} L)\otimes_{\mbb Z} \mbb Z_p\simeq (x-1) L^\wedge_p /(x-1)^{t+1} L^\wedge_p,$$
	where $L^\wedge_p\coloneqq \mbb Z_p[(\mbb Z/p)^D]$. Let $\phi_p(x)\coloneqq x^{p-1}+x^{p-2}+\ldots +1$. We have $x^p-1=(x-1) \cdot \phi_p(x)$, and so ideal  $(x-1) L^\wedge_p$ has the $\mc O_K\simeq \mbb Z_p[x]/\phi_p(x)$-module structure; moreover the corresponding module is naturally identified with the ideal generated by $(x-1)\in\mbb Z_p[x]/\phi_p(x)$, which corresponds to $\pi\in \mc O_K$. This gives the required isomorphism.
\end{proof}	
\begin{rem}[Case $A=\mbb Z/p$]\label{rem:case A=Z/p}
	Let $K/\mbb Q_p$ be a finite extension and $\mathfrak{d}_{K/\mbb Q_p}\subset \mc O_K$ the different ideal of $K$; recall that by definition the inverse ideal $\mathfrak{d}_{K/\mbb Q_p}^{-1}\subset K$ is identified with the dual $\mbb Z_p$-lattice to $\mc O_K$ via the pairing $(\alpha,\beta)\mapsto \tr_{K/\mbb Q_p}(\alpha\beta)$. The Pontryagin dual $(\mc O_K)^D$ (defined, say, as the group of continuous homomorphisms $ \Hom_\Ab^{\mr{cts}}(\mc O_K,\mbb Q/\mbb Z)$) is $\Gal(K/\mbb Q_p)$-equivariantly identified with $K/\mathfrak{d}_{K/\mbb Q_p}^{-1}$ via sending $\alpha\in K/\mathfrak{d}_{K/\mbb Q_p}^{-1}$ to the homomorphism $\beta \in \mc O_K \mapsto \tr_{K/\mbb Q_p}(\alpha\beta)$, where the latter is a well-defined element of $\mbb Q_p/\mbb Z_p\subset \mbb Q/\mbb Z$. More generally, for any $x,y\in \mc O_K$, such that $x|y$, this restricts to a $\Gal(K/\mbb Q_p)$-equivariant isomorphism
$$
((y)/(x))^D\simeq (x\cdot\mathfrak{d}_{K/\mbb Q_p})^{-1}/ (y\cdot \mathfrak{d}_{K/\mbb Q_p})^{-1}
$$
	In the case $K=\mbb Q_p(\zeta_p)$ the different ideal is given by $(\zeta_p-1)^{p-2}=(\pi^{p-2})\subset \mc O_K$. By \Cref{cor:L^1Bin for Z/p} we get an isomorphism 
	$$
	\mbf L^1\Bin^{\le t}(\mbb{Z}/p[-1])\simeq ((\pi)/(\pi^{t+1}))^D \simeq (\pi^{-t-1-p})/(\pi^{1-p})\simeq (\pi^{-t})/\mc O_K
	$$
	 Here in the last isomorphism we used the equality of ideals $(\zeta_p-1)^{p-1}=(p)$, which follows from the fact that $\mr{val}_p(\zeta_p-1)=1/p-1$. By passing to the limit as $t\ra \infty$ we also get a $(\mbb Z/p)^{\times}$-equivariant isomorphism 
	 $$
	 \mbf L^1\Bin(\mbb{Z}/p[-1])\simeq K/\mc O_K,
	 $$
	 where $(\mbb Z/p)^{\times}$-acts on the right  through the Galois action.
\end{rem}	
\begin{rem}\label{rem:conditions on commutation with totalizations are sharp} Note that $\mbb Z/p[-1]$ is equivalent to a (finite) totalization of the Dold-Kan corresponding cosimplicial object $\mr{DK}^\bullet(\mbb Z/p[-1])\simeq \mbb Z/p^{\oplus \bullet}$:
	$$
	\mbb Z/p[-1] \xra{\sim} \Tot_{[\bullet]\in \Delta} \mbb Z/p^{\oplus \bullet}.
	$$
	By \Cref{prop:Bin of torsion abelian group}, $\mbf L\Bin(\mbb Z/p^{\oplus \bullet})\simeq \mbf L\Bin(0)\simeq \mbb Z$, and so 
	$$
	\Tot_{[\bullet]\in \Delta} \mbf L\Bin(\mbb Z/p^{\oplus \bullet})\simeq \Tot_{[\bullet]\in \Delta} \mbb Z \simeq \mbb Z,
	$$
	which does not agree with $\mbf L\Bin(\mbb Z/p[-1])$. This shows that $\mbf L\Bin$ does not commute with finite totalizations if we even slightly relax the boundedness condition in \Cref{prop:LBin commutes with suitable totalizations} (note that even though $\mbb Z/p^{\oplus \bullet}$ is coconnective in the standard t-structure it does not lie in $\mscr D(\mbb Z)^{\ge 0}_{\mr{perv}}$).
\end{rem}

\section{Derived binomial rings and cohomology of log analytic spaces}\label{sec:log stuff}
In this section we apply the derived binomial monad $\mbf L\Bin$ to give a formula for integral Betti(=singular) cohomology of fs log analytic spaces. The latter was defined in {\cite[Definition 1.1.1]{Kato-Nak-99}} in terms of an auxiliary topological space $\pi\colon X^{\log}\ra X$ which now carries the name of "Kato-Nakayama space". In \Cref{ssec:log betti} we recall its definition and also introduce the 2-term complex $\eexp(\alpha)\coloneqq \mc O_X \ra \mc M^{\gr}$, which gives an explicit model for $R^{\le 1}\pi_*\ul{\bZ}$. In \Cref{ssec:sheaves of binomial algebras}, for a sufficiently nice topological space $X$ we construct the derived binomial monad $\mbf L\mc{B}in_X$ on the category $\Shv(X,\mscr D(\mbb Z))$ of $\mscr D(\mbb Z)$-valued sheaves on $X$ and identify modules over $\mbf L\mc{B}in_X$ with the category of sheaves of derived binomial rings on $X$. In \Cref{ssec:LBin'} we define a slightly modified version $\mbf L\mc{B}in^{\mr{coaug}}_X$ of a free binomial ring in the undercategory $\Shv(X,\mscr D(\mbb Z))_{\setminus \ul \bZ}$. In \Cref{ssec:cohomology of KN-space} we then show that given a proper map $\pi\colon Y\ra X$ of sufficiently nice topological spaces such that each fiber is a compact real torus, the pushforward $R\pi_*\ul{\mbb Z}$ is the free binomial ring $\mbf L\mc{B}in^{\mr{coaug}}_X(R^{\le 1}\pi_*\ul{\bZ})$; applying this to $\pi\colon X^{\log}\ra X$ gives a formula for $\RG_\sing^{\log}(X,\mbb Z)$ as global sections of $\mbf L\mc{B}in^{\mr{coaug}}_X(\eexp(\alpha))$. In \Cref{ssec:cosheaf shit} we show that one can also reconstruct from $\eexp(\alpha)$  the cosheaf $[X^{\log},\pi]$ represented by $X^{\log}$ and also define a gerbe structure on the latter over the cosheaf of groups $\ol{\mc M^{\gr}}^\vee$. Finally, in \Cref{ssec:geometrization} we propose an algebraic version $X^{\log}_{\mbb H}$ of $[X^{\log},\pi]$ which we hope to be able to generalize to other cohomological contexts in the future.

\subsection{Log singular cohomology}\label{ssect:log_Betti_coh}\label{ssec:log betti}

\begin{defn}[{\cite[Definition 1.1.1]{Kato-Nak-99}}]
	 Recall  that a log complex analytic space $(X,\mathcal M)$ is given by:
	 \begin{itemize}
	 	\item[-] a complex analytic space $X$;
	 	\item[-] a sheaf of commutative monoids $\mathcal M$;
	 	\item[-] a homomorphism\footnote{Here we consider the sheaf  $\mathcal O_X$ of holomorphic functions as a sheaf of monoids with the operation given by multiplication.} $\alpha\colon \mathcal M\to \mathcal O_X$ such that $\alpha^{-1}(\cO^\times_X)\iso \cO^\times_X$.
	 \end{itemize}
\end{defn}
\begin{assumption}\label{ass:fs log anal}
In this paper we will only consider \textit{fine and saturated (fs)} log analytic spaces (\cite[Definition 1.1.2]{Kato-Nak-99}). The key property of fs log analytic spaces is that any fiber of the sheaf $\mathcal M$ is a fine and saturated monoid (see \cite[Section I.1.3]{Ogus_Book}). In particular, the \textit{ghost sheaf}  $\overline{\mathcal M^{\gr}}:=\coker(\cO^\times_X\to \mathcal M^{\gr})$ (where $\mathcal M^{\gr}$ denotes the group completion of a sheaf of monoids) is constructible     with stalks given by torsion-free finitely generated abelian groups (in fact this is the only property of fs log structure that we are going to use).
\end{assumption}

The classical approach to defining log Betti (=singular) cohomology of a log analytic space is via the construction of the corresponding Kato-Nakayama space $X^{\log}$ (\cite{Kato-Nak-99}), that we now recall.


\begin{construction}[Kato-Nakayama space]\label{constr:Kato-Nakayama space} To an fs log analytic space $(X,\mathcal M)$ one can associate a certain natural topological space $X^{\log}$, together with a continuous map $X^{\log} \ra X$. The space $X^{\log}$ is called the \textit{Kato-Nakayama space} of $(X,\mathcal M)$. The underlying set of $X^{\log}$ is given by the pairs $(x,\varphi)$ of a point $x\in X$, and a group homomorphism
	$\varphi\colon \mathcal{M}^{\gr}_x\ra S^{1}$
	that makes the diagram
	\[
	\xymatrix{{\mathcal{O}_{X,x}^\times}\ar[r]\ar[d]_{\mathrm{arg}_x} & \mathcal{M}^{\gr}_x\ar[ld]^{\varphi}\\
		S^{1}
	}
	\]
	commutative. Here $\arg_x$ denotes the \textit{argument} function that sends $$f\mapsto \arg_x(f)\coloneqq \frac{f(x)}{|f(x)|}\in S^1\subset \mbb C.$$
	
	The topology on $X^{\log}$ is defined as follows. By construction, we have a natural map of sets $\pi\colon X^{\log}\ra X$, sending $(x,\phi)$ to $x$. For any open $U\subset X$ and a section ${m}\in\mathcal{M}(U)$,
	we can define a function 
	\[
	\mathrm{arg}({m})\colon \pi^{-1}(U)\longrightarrow S^{1}; 
	\]
	$$(x,\varphi)\longmapsto\varphi({m}_{x}).$$
	The space $X^{log}$ is endowed with the weak topology defined
	by $\pi$ and maps $\mathrm{arg}({m})$ for all $m\in \mathcal M^{\gr}_{\an}(U)$ and open $U\subset X^{an}$. 
	
\end{construction}

\begin{rem}\label{rem:map pi}
	By \cite[Corollary V.1.2.8]{Ogus_Book} the map $\pi\colon X^{\log}\ra X$ is proper. The fiber $\pi^{-1}(x)\hookrightarrow X^{\log}$ over a point $x\in X$ is identified with the set of group homomorphisms
	 $$\pi^{-1}(x)\simeq \Hom(\ol{\mc M^{\gr}_x},S^1),$$
	 with the topology induced from $S^1$. In particular, $\pi^{-1}(x)$ is a compact real torus, whose dimension is equal to the rank of $\ol{\mc M^{\gr}_x}$.
\end{rem}	

\begin{defn}[Log singular cohomology]
	The	\textit{log Betti cohomology}\footnote{Or, in other words, singular log cohomology.} $\RG^{\log}_{\sing}(X,\bZ)$ of a log analytic space $X$ is defined as the the singular cohomology $\RG_\sing(X^{\log},\mbb Z)$ of the topological space $X^{\log}$.
\end{defn}	

\begin{rem}[Log-schemes over $\mbb C$ and analytification]
	To an  \textit{fs log scheme} $(X,\mc M)$ (\cite[Definition 1.7]{Nakayama}) over $\mbb C$ with $X$ being of finite type one can naturally associate its \textit{analytification} $(X_{\an},\mc M_{\an})$ (where $\mc M_{\an}$ is simply defined by taking pull-back of the log-structure from $X$ to $X_\an$) which is an fs log analytic space. Following \cite{Kato-Nak-99} one can define log Betti cohomology of  $(X,\mc M)$ as log Betti cohomology of $(X_{\an},\mc M_{\an})$. As shown in \cite{Kato-Nak-99}, it agrees with the appropriately defined log \'etale and log de Rham cohomology under naturally defined comparison isomorphisms.
\end{rem}

As discussed in the introduction, one can look for alternative, more direct descriptions of the log-Betti cohomology that would only directly involve the analytic space $X$ and the monoid $\mc M$. Let $(X,\mathcal M)$ be an fs log analytic space. We will denote by 
$\exp_{\mc M}\colon \mathcal O_{X} \xra{} \mathcal M^{\gr}$
the composition of the classical exponential map $\mc O_{X}\xra{\exp} \mc O^\times_{X}$ 
with the natural inclusion $\cO_{X}^\times \hookrightarrow \mathcal M^{\gr}$ given by the log structure.

\begin{notation}[2-term complex of sheaves $\eexp(\alpha)$] \label{not:exp(a)}Let  
$$\eexp(\alpha)\coloneqq \fib(\exp_{\mc M}\colon \mathcal O_{X} \ra \mathcal M^{\gr})$$
be the fiber of $\exp_{\mc M}$ in the category $\Shv(X,{\mscr D}(\mbb Z))$ of sheaves of abelian groups on $X$.  Equivalently, we consider an object of $\Shv(X,{\mscr D}(\mbb Z))$ represented by the 2-term complex $ [\mathcal O_{X} \ra \mathcal M^{\gr}]$. The cohomology sheaves of $\eexp(\alpha)$ are given by $H^0(\eexp(\alpha))\simeq \ker(\exp_{\mc M})\simeq \ker (\exp)\simeq \ul \bZ$ and  $H^1(\eexp(\alpha))\simeq \coker(\exp_{\mc M})\simeq \ol{ \mc M^{\gr}}\coloneqq \mc M^{\gr}/ \mc O_X^\times$.
\end{notation}

\begin{rem}\label{rem: exp computes first truncation}
 As we will see below, the complex $\eexp(\alpha)$ is an explicit model for $R^{\le 1}\pi_*\ul{\mbb Z}$, where $\pi\colon X^{\log} \ra X$ is the natural projection from the Kato-Nakayama space (\Cref{lem:exp is the first truncation}).
\end{rem}	

To compute $\RG^{\log}_{\mr{Betti}}(X,\bZ)$ one can proceed in two steps: namely, first take pushforward $R\pi_* \ul{\mbb Z}$ to $X$ and then compute the derived global sections of the latter. Our goal will be to give a closed formula for the full derived pushforward $R\pi_* \ul{\mbb Z}$ in terms of its first truncation $R^{\le 1}\pi_*\ul{\mbb Z}\simeq \eexp(\alpha)$.

\subsection{Sheaves of binomial rings}\label{ssec:sheaves of binomial algebras} 
\begin{notation}Let $X$ be a topological space and let $\Op(X)$ be the category given by open subsets of $X$. Recall that a \textit{sieve} on $X$ is a collection $\mf U$ of open subsets of $X$, such that if $V\in \mf U$ and $V'\subset V$, then $V'\in \mf U$. If $U\simeq \cup_{V\in \mf U} V$ we will say that $\mf U$ is a covering sieve for $U$. The data $U\mapsto \{\mf U, \text{ $\mf U$ covers $U$}\}$ gives rise to an $(\infty,1)$-site in the sense of \cite[Definition 6.2.2.1]{Lur_HTT}. For an $\infty$-category $\mscr C$ the category of $\mscr C$-valued presheaves is simply defined as the functor category $\PShv(X, \mscr C)\coloneqq \Fun(\Op(X)^\op,\mscr C)$. Assume  $\mscr C$ has all small limits. Then a presheaf $\mc F\in \PShv(X, \mscr C)$ is called a sheaf if for any covering sieve $\mf U$ of $U$ the natural map
	$$
	\mc F (U)\ra \lim_{V\in \mf U} \mc F(V) 
	$$
is an equivalence. We will denote by $\Shv(X, \mscr C)\subset \PShv(X, \mscr C)\coloneqq \Fun(\Op(X)^\op,\mscr C)$ the $\infty$-category of $\mscr C$-valued sheaves on $X$.

By \cite[Remark 1.3.1.6]{Lur_SAG}, if $\mscr C$ is presentable, $\Shv(X, \mscr C)$ can be identified with the Lurie's tensor product $\Shv(X, \Spc)\otimes \mscr C$, and is also presentable. Also, by \cite[Lemma 1.3.4.3]{Lur_SAG}, the inclusion $i_X\colon \Shv(X, \mscr C)\hra \PShv(X, \mscr C)$ has a left adjoint $\#_X\colon \PShv(X, \Spc)\to \Shv(X, \Spc)$, which we will call the sheafification functor. 
\end{notation}
\begin{rem}
	
	Recall that a paracompact topological space $X$ is said to have \textit{covering dimension $\le n$} if the following condition is satisfied: for any open covering
	$\{U_\alpha\}$ of $X$, there exists an open refinement $\{V_\beta\}$ such that each intersection $V_{\beta_0}\cap \ldots \cap V_{\beta_{n+1}}=\emptyset$ provided the $\beta_i$ are pairwise distinct. By \cite[Theorem 7.2.3.6]{Lur_HTT} if $X$ is paracompact and has locally finite covering dimension, then the $\infty$-topos $\Shv_\infty(X)\coloneqq \Shv(X, \Spc)$ has locally finite \textit{homotopy dimension}. This implies some nice properties of $\Shv_\infty(X)$: most importantly Postnikov towers are convergent\footnote{Meaning that any $F\in \Shv_\infty(X)$ the natural map $F\ra \lim_{n}\tau^{\le n}F$ is an equivalence.} \cite[Proposition 7.2.1.10]{Lur_HTT} and, consequently, $\infty$-topos $\Shv_\infty(X)$ is hypercomplete; in other words, $\Shv_\infty(X)\simeq \hShv(X, \Spc)$, where $\hShv$ denotes the category of hypersheaves. Note that this also implies that $\Shv(X, \mscr C)\simeq \hShv(X, \mscr C)$ for any presentable $\infty$-category $\mscr C$. By \cite[Corollary 7.2.1.17]{Lur_HTT} it also follows that a map $p\colon \mc F\ra \mc G$ in $\Shv(X, \mscr C)$ is an equivalence if and only if the map of stalks $p_x\colon \mc F_x \ra \mc G_x$ in $\Shv(\mr{pt}, \mscr C)\simeq \mscr C$ is an equivalence for all $x\in X$. 
\end{rem}

\begin{assumption}\label{ass:finite covering dimension}
	Below we are going to assume that all topological spaces are paracompact locally of finite covering dimension.
	We note that any locally closed subspace of $\mbb R^n$ has a covering dimension $\le n$; in particular any complex analytic space $X$, or Kato-Nakayama space $X^{\log}$ corresponding to some fs log structure $\mc M$ are paracompact of locally finite covering dimension.
\end{assumption}	

\begin{rem}[$t$-structure]
	We will be primarily working with the category $\Shv(X, {\mscr D}(\mbb Z))$ of ${\mscr D}(\mbb Z)$-valued sheaves. By \cite[Proposition 2.1.1.1]{Lur_SAG} it has a natural $t$-structure $(\Shv(X, {\mscr D}(\mbb Z))^{\le 0},\Shv(X, {\mscr D}(\mbb Z))^{\ge 0})$ with the subcategory $\Shv(X, {\mscr D}(\mbb Z))^{\ge 0}$ being spanned by those $\mc F\in \Shv(X, {\mscr D}(\mbb Z))$ such that $\mc F(U)\in {\mscr D}(\mbb Z)^{\ge 0}$ for any $U$. The heart $\Shv(X, {\mscr D}(\mbb Z))^\heartsuit\subset \Shv(X, {\mscr D}(\mbb Z))$ is naturally identified with the abelian category $\Shv(X,\Ab)$ of classical sheaves of abelian groups on $X$. 
	
	We note that the image of $\mc F\in \Shv(X, \Ab)$ under the resulting functor\footnote{Note that the functor $\PShv(X,\Ab)\ra \PShv(X,\mscr D(\mbb Z))$ does not preserve the subcategories of sheaves. The functor in question is obtained as the composition of the natural functor $\Shv(X,\Ab)\ra \PShv(X,\mscr D(\mbb Z))$ and sheafification $\#_X\colon \PShv(X, \mscr D(\bZ))\to \Shv(X, \mscr D(\bZ))$.}
	$$
	\Shv(X, \Ab)\simeq \Shv(X, {\mscr D}(\mbb Z))^\heartsuit\ra \Shv(X, {\mscr D}(\mbb Z))
	$$ 
	is sending an open $U$ to $\RG(U,\mc F)$ rather than $\mc F(U)\in \Ab\subset D(\mbb Z)$.
	
\end{rem}	

\begin{rem}[Comparison with the derived category of sheaves] Identification $\Shv(X,\Ab)\simeq \Shv(X, \mscr {\mscr D}(\mbb Z))^\heartsuit$ extends to a fully faithful functor 
	$$
	\mscr D^+\!(\Shv(X,\Ab)) \ra  \Shv(X, \mscr {\mscr D}(\mbb Z)).
	$$
	If the $\infty$-topos $\Shv(X, \Spc)$ is hypercomplete, the $t$-structure on $\Shv(X, \mscr {\mscr D}(\mbb Z))$ is left-separated, and by \cite[Section 2.1]{Lur_SAG} the latter functor extends to an equivalence
	$$
	\mscr D(\Shv(X,\Ab)) \ra \Shv(X, \mscr {\mscr D}(\mbb Z))
	$$
	with the unbounded derived category $\mscr D(\Shv(X,\Ab))$.
\end{rem}

The monad $\mbf L\Bin$ on ${\mscr D}(\mbb Z)$ naturally extends to a monad on the category of presheaves $\PShv(X, {\mscr D}(\mbb Z))\coloneqq \Fun(\Op(X)^\op,{\mscr D}(\mbb Z))$ via post-composition with $\mbf L\Bin$. We will denote this monad by $\mbf L\Bin_X$. Via Barr-Beck-Lurie theorem one can identify the category of modules over $\mbf L\Bin_X$ in $\PShv(X, {\mscr D}(\mbb Z))$ with the category of presheaves $\PShv(X, \DBinAlg)$ on $X$ with values in $\DBinAlg$.

Let us as before denote by $i_X$ the inclusion $\Shv(X, {\mscr D}(\mbb Z))\hra \PShv(X, {\mscr D}(\mbb Z))$, and by $\#_X$ its left adjoint "sheafification" functor $\PShv(X, {\mscr D}(\mbb Z))\to \Shv(X, {\mscr D}(\mbb Z))$. Note that since $i_X$ is fully faithful, by definition, $\Shv(X, {\mscr D}(\mbb Z))$ is a reflective localization of $\PShv(X, {\mscr D}(\mbb Z))$. We can define an endofunctor of $\Shv(X, {\mscr D}(\mbb Z))$ by the formula:
$$\mbf L\mathcal Bin_X\coloneqq \#_X \circ \mbf L\Bin_X \circ i_X.$$

\begin{prop}[{\cite[Proposition 4.1.9.]{Arpon}}]\label{Arpon th}
	Let $\tau\colon \mscr C\rightleftarrows \mscr C_0:\! i$ be a localization of $\infty$-categories. Let $\mr T$ be a monad on $\mscr C$ such that the unit transformation $\mathrm{id}_{\mscr C}\to i\circ\tau$ induces an equivalence $\tau\circ \mr T\sa{\sim}\tau\circ \mr T\circ i \circ \tau$. Then there is an induced monad structure on the composite $\mr T_0\coloneqq \tau\circ \mr T \circ i$ and an induced localization
	$$\tau\colon \mr{L}\mscr{M}\mathrm{od}_{\mr T}(\mscr C)\rightleftarrows \mr{L}\mscr{M}\mathrm{od}_{\mr T_0}(\mscr C_0) \ \!\!:\! i$$
	where the embedding $i$ identifies $\mr{L}\mscr{M}\mathrm{od}_{\mr T_0}(\mscr C_0)$ with the fiber product $\mr{L}\mscr{M}\mathrm{od}_{\mr T}(\mscr C)\times_{\mscr C}\mscr C_0$.
\end{prop}

\begin{cor}\label{corr of Arpon th}
	The endofunctor $\mbf L\mathcal B\mr{in}_X\coloneqq \#_X \circ \mbf L\Bin_X \circ i_X$ has a natural structure of a monad. 
\end{cor}
\begin{proof} We apply \Cref{Arpon th}.
 Using adjunction $\#_X\dashv i_X$, we may consider $\Shv(X, {\mscr D}(\mbb Z))$ as the reflective localization of $\PShv(X, {\mscr D}(\mbb Z))$. Let $\mc F\in  \PShv(X, {\mscr D}(\mbb Z))$ be a presheaf. Unit of the adjunction $\#_X\dashv i_X$ induces the map 
$$(\mbf L\Bin_X(\mc F))^\#\sa{\phi} (\mbf L\Bin_X(\mc F^\#))^\#,$$
and we only need to check that $\phi$ is an equivalence. By our assumptions on $X$, it is enough to check that it induces an equivalence on all stalks $\phi_x$ at all points $x\in X$.  Note that the sheafification does not change the stalks, so we can as well replace $\phi_x$ with the map
	$$
	\mbf L\Bin_X(\eta)_x\colon \mbf L\Bin_X(\mc F)_x\ra \mbf L\Bin_X(\mc F^\#)_x,
	$$
where $\eta\colon F\to F^\#$ is the unit of $\#_X\dashv i_X$ adjunction. But, since the stalk is given by a filtered colimit, and $\mbf L\Bin$ commutes with filtered colimits, this map can be identified with the identity map  $\mbf L\Bin(\mc F_x)\ra \mbf L\Bin((\mc F^\#)_x)\simeq \mbf L\Bin(\mc F_x)$.
\end{proof}
\begin{notation}[Derived binomial rings over $X$]
We will denote by $$\DBinAlg(X)\coloneqq \mr{L}\mscr{M}\mathrm{od}_{\mbf L\mathcal Bin_X}(\Shv(X, {\mscr D}(\mbb Z)))$$ the $\infty$-category of modules over $\mbf L\mathcal Bin_X$ in $\Shv(X, {\mscr D}(\mbb Z))$. Slightly abusing notation, for $\mc F\in \Shv(X, {\mscr D}(\mbb Z))$ we will now denote by $\mbf L\mathcal Bin_X(\mc F)\in \DBinAlg(X)$ the free derived binomial ring on $\mc F$, as well as its underlying object in $\Shv(X, {\mscr D}(\mbb Z))$. 
\end{notation}

\begin{ex}
	Let $X\coloneqq \mr{pt}$ be a point. In this case $\Shv(\mr{pt}, {\mscr D}(\mbb Z))\simeq \PShv(\mr{pt}, {\mscr D}(\mbb Z))\simeq {\mscr D}(\mbb Z)$ and $\mbf L\mc Bin_{\mr{pt}}\simeq \mbf L\Bin_{\mr{pt}}\simeq \mbf L\Bin$; we get an equivalence of $\infty$-categories
	$$
	\DBinAlg(\mr{pt})\simeq \DBinAlg.
	$$
\end{ex}	
\begin{rem}\label{presentability of DBinAlg}
	Since $\mbf L\Bin$ commutes with sifted colimits, and colimits in $\PShv(X, {\mscr D}(\mbb Z))$ are computed pointwise, we see that $\mbf L\Bin_X$ also commutes with sifted colimits. We claim that $\mbf L\mathcal B\mr{in}_X$ commutes with sifted colimits as well. Indeed, given a sifted diagram $I$ and a functor $I\ra \Shv(X, {\mscr D}(\mbb Z))$, $i\mapsto \mc F_i$
	it suffices to show that the canonical map
	$$\colim_{i\in I}\mbf L\mathcal B\mr{in}_X(\mathcal F_i)\ra \mbf L\mathcal B\mr{in}_X(\colim_{i\in I} \mathcal F_i)$$
	is a quasi-isomorphism. Again, this can be checked on the level of stalks. Here we have
	$$\mbf L\mathcal B\mr{in}_X(-)_x\iso \mbf L\Bin_X(-)_x\iso \mbf L\Bin((-)_x)$$
	so we can rewrite the map above as
	$$\colim_{i\in I} \mbf L\Bin((\mathcal F_i)_x)\ra \mbf L\Bin(\colim_{i\in I}((\mathcal F_i)_x)).$$
	and use that $\mbf L\Bin$ commutes with sifted colimits.
	As a  corollary, by \cite[Proposition 4.1.10]{Arpon} we get that the categories $\LMod_{\mbf L\Bin_X}(\PShv(X,{\mscr D}(\mbb Z))$ and $\DBinAlg(X)$ are presentable. In particular, both categories admit all small limits and colimits. 
	
	By \cite[Corollary 4.2.3.5]{Lur_HA} it follows that the forgetful functor $$U_X\colon \DBinAlg(X) \ra \Shv(X,{\mscr D}(\mbb Z))$$
	commutes with sifted colimits. 
\end{rem}

\begin{rem}\label{R5.5}
	By
	\Cref{Arpon th} and the natural equivalence $$\mr{L}\mscr{M}\mathrm{od}_{\mbf L\Bin_X}(\PShv(X,{\mscr D}(\mbb Z)))\simeq \PShv(X,\DBinAlg)$$ we get another natural equivalence
	\begin{equation}\label{eq:}
	\DBinAlg(X)\iso \PShv(X,\DBinAlg)\times_{\PShv(X,{\mscr D}(\mbb Z))}\Shv(X,{\mscr D}(\mbb Z)).
	\end{equation}
	In other words, the category $\DBinAlg(X)$ is given by the full subcategory in $ \PShv(X,\DBinAlg)$ of presheaves of derived binomial rings on $X$ whose underlying complex is a sheaf. Since the forgetful functor $\PShv(X,\DBinAlg)\ra \PShv(X,{\mscr D}(\mbb Z))$ is conservative and commutes with limits, we get an identification of$	\DBinAlg(X)$ with the category of $\DBinAlg$-valued sheaves on $X$:
	$$
	\DBinAlg(X)\simeq \Shv(X,\DBinAlg).
	$$
	We will denote by 
	$$\#_X^\Bin \colon  \PShv(X,\DBinAlg)\rightleftarrows \DBinAlg(X)\ \!\!:\! i_X^\Bin$$ the associated sheafification functor and its right adjoint. We might occasionally omit the superscript "$\Bin$" to simplify notation.
	
	Let us also note that via the identification in \Cref{Arpon th} as the fiber product, identifying two projections, we get a natural equivalence
	$$
	U_X^{\mr{psh}}\circ i_X^\Bin \simeq i_X\circ U_X,
	$$
	where we denoted by $U_X^{\mr{psh}}\colon \PShv(X, \DBinAlg) \ra \PShv(X, {\mscr D}(\mbb Z))$ the post-composition with the forgetful functor $U\colon \DBinAlg\ra {\mscr D}(\mbb Z)$.
\end{rem}

\begin{lem}\label{lem:U and sheafification}
	Let $X$ be as in \Cref{ass:finite covering dimension}. There is a natural equivalence of functors
	$$
	\#_X\circ U_X^{\mr{psh}} \xra{\sim} U_X\circ \#_X^\Bin
	$$
	from $\PShv(X,\DBinAlg)$ to $\Shv(X,{\mscr D}(\mbb Z))$
\end{lem}
\begin{proof}
From $U_X^{\mr{psh}}\circ i_X^\Bin \simeq i_X\circ U_X$ and $\#_X\dashv i_X$ adjunction, one gets a natural transformation $\#_X\circ U_X^{\mr{psh}}\circ i_X^\Bin \ra U_X$. Pre-composing it with $\#_X^\Bin$, and using $\#_X^\Bin\dashv i_X^\Bin$ adjunction we get transformations 
$$
\#_X\circ U_X^{\mr{psh}} \ra \#_X\circ U_X^{\mr{psh}}\circ i_X^\Bin \circ \#_X^\Bin \ra U_X\circ \#_X^\Bin,
$$
giving the required map as the composition. Moreover, none of the two functors change the stalks (or rather their underlying objects in ${\mscr D}(\mbb Z)$), and the induced map on stalks is the identity map, so this transformation is an equivalence.
\end{proof}

The next step is to define direct and inverse image functors for $\DBinAlg(-)$. We will do so by sheafifying the corresponding functors on the level of presheaves.

\begin{rem}[Pull-back and pushforward for presheaves]\label{rem:pull/push_for_preasheaves}
		Recall that for any cocomplete category $\mscr C$, the pull-back $f^{\dagger}_{\mscr C}\colon \PShv(X,\mscr C) \ra \PShv(Y,\mscr C)$ on the categories of $\mscr C$-valued presheaves is defined as the left adjoint to the functor $Rf_{\dagger,\mscr C} \colon \PShv(Y,\mscr C) \ra \PShv(X,\mscr C)$ obtained by pre-composition with $f^{-1,\op}\colon \Op(X)^\op \ra \Op(Y)^\op$. Any functor $F\colon \mscr C\ra \mscr D$ induces functors $$F_X\colon  \PShv(X,\mscr C) \ra  \PShv(X,\mscr D), \qquad F_Y\colon  \PShv(Y,\mscr C) \ra  \PShv(Y,\mscr D)$$ via post-composition with $F$, and we have a natural equivalence $Rf_{\dagger,\mscr D}\circ F_Y\simeq F_X\circ Rf_{\dagger,\mscr C}$ of functors $\PShv(Y,\mscr C)\ra \PShv(X,\mscr D)$. If $F$ has a left adjoint $G$ we also get a natural equivalence 
	$$
	G_Y\circ f^\dagger_{\mscr D} \simeq f^\dagger_{\mscr C} \circ G_X.
	$$
	
\end{rem}	

Note that since $f^{-1}\colon \Op(X)\ra \Op(Y)$ sends open covers to open covers, $Rf_{\dagger,\Bin} \colon \PShv(Y,\DBinAlg) \ra \PShv(X,\DBinAlg)$ sends sheaves to sheaves.
\begin{construction}[{$Rf_{*,\Bin}$ and $f^*_\Bin$}]
	Let $f\colon Y\ra X$ be a continuous map of topological spaces satisfying \Cref{ass:finite covering dimension}. We define the \textit{(derived) pushforward functor} 
	$$
	Rf_{*,\Bin}\colon \DBinAlg(Y)\ra \DBinAlg(X)
	$$
	as the restriction of $Rf_{\dagger,\Bin}$ to $\DBinAlg(Y)\subset \PShv(Y,\DBinAlg)$. Note that $Rf_{\dagger,\Bin}$ commutes with all small limits, and since limits of sheaves can be computed on the level of presheaves, so does $Rf_{*,\Bin}$. We define the \textit{pull-back functor} $$f_{\Bin}^*\colon \DBinAlg(X)\ra \DBinAlg(Y)$$ as the left adjoint to $Rf_{*,\Bin}$, that exists by Lurie's adjoint functor theorem. Being a left adjoint, $f_{\Bin}^*$ commutes with all small colimits.
\end{construction}	

\begin{rem}\label{rem:sheafy pull-back is the sheafification of presheaf pull-back}
	The chain of natural equivalences
	$$
	\Map(f^*_\Bin\mc A,\mc B)\simeq \Map(\mc A,Rf_{*,\Bin}\mc B)\simeq \Map(i_X^\Bin\mc A,i_X^\Bin (Rf_{*,\Bin}\mc B))\simeq 
	$$
	$$
	\simeq \Map(i_X^\Bin\mc A,Rf_{\dagger,\Bin}(i_Y^\Bin \mc B))= \Map(f^\dagger_\Bin(i_X^\Bin\mc A), i_Y^\Bin \mc B)\simeq \Map((\#^\Bin_Y\circ f^\dagger_\Bin \circ i_X^\Bin)(\mc A),\mc B)
	$$
	shows that there is a natural equivalence of functors 
	$$
	f^*_\Bin\simeq \#^\Bin_Y\circ f^\dagger_\Bin \circ i_X^\Bin.
	$$
	In other words, $f^*_\Bin$ is obtained from the presheaf pull-back $f^\dagger_\Bin$ by sheafification.
\end{rem}

Let us now show some basic properties of how the functors $Rf_{*,\Bin}$ and $f^*_\Bin$ interact with $\mbf L\mc Bin_{-}$ and $U_{-}$.

\begin{prop}\label{prop:properties of pullback and pushforward}	Let $f\colon Y\ra X$ be a continuous map of topological spaces satisfying \Cref{ass:finite covering dimension}. 
	There are the following natural equivalences:
	$$f^*_{\Bin}\circ \mbf L\mathcal Bin_X\iso \mbf L \mathcal Bin_Y \circ f^*;\,\,\,\, U_X\circ Rf_{*,\Bin}\iso Rf_*\circ U_Y;\,\,\,\, f^*\circ U_X\iso U_Y\circ f_{\Bin}^*.$$ 
\end{prop} 
\begin{proof} By the discussion in \Cref{rem:pull/push_for_preasheaves} applied to $\mscr C\coloneqq \DBinAlg$, $\mscr D\coloneqq {\mscr D}(\mbb Z)$, $F\coloneqq U:\DBinAlg \ra {\mscr D}(\mbb Z)$ we have a natural equivalence of functors
	\begin{equation}\label{eq:pushforward+forgetful on the level of presheaves}
	U_X^{\mr{psh}} \circ Rf_{\dagger,\Bin}\simeq Rf_{\dagger}\circ U_Y^{\mr{psh}}.
	\end{equation}
We also have equivalences
$$
	\#_X\circ U_X^{\mr{psh}} \circ Rf_{\dagger,\Bin}\circ i_Y^\Bin\simeq U_X\circ 	\#_X^\Bin  \circ Rf_{\dagger,\Bin}\circ i_Y^\Bin \simeq U_Y  \circ Rf_{*,\Bin}
$$	
	by \Cref{lem:U and sheafification} and the fact that $Rf_{\dagger,\Bin}$ commutes with $i_X$. On the other hand 
	$$
	\#_X\circ Rf_{\dagger,{\mscr D}(\mbb Z)}\circ U_Y^{\mr{psh}}\circ  i_Y^\Bin \simeq \#_X \circ Rf_{\dagger,{\mscr D}(\mbb Z)} \circ  i_Y\circ U_Y \simeq Rf_* \circ U_Y,
	$$
	since $U_Y^{\mr{psh}}\circ  i_Y^\Bin \simeq i_Y\circ U_Y$. This way we get the second equivalence, by pre-composing \Cref{eq:pushforward+forgetful on the level of presheaves} with $i_Y^\Bin$ and post-composing with $\#_Y$. The first equivalence then is obtained by taking left adjoints.

	 Let us now construct the third equivalence.
	 	Using $(\mbf L\mathcal B\mr{in}_{Y})\dashv (U_{Y})$ adjunction, the category of transformations $ f^*\circ U_X\Rightarrow U_Y\circ f_{\Bin}^*$ is equivalent to the category of transformations $\mbf L\mathcal B\mr{in}_{Y}\circ f^*\circ U_X \Rightarrow f_{\Bin}^*$. Via the natural equivalence $\mbf L\mathcal B\mr{in}_{Y}\circ f^* \simeq f^*_{\Bin}\circ \mbf L\mathcal B\mr{in}_X$ and counit $\mbf L\mathcal B\mr{in}_X\circ U_X\ra \id$ of the adjunction we get a transformation  $f^*\circ U_X\ra U_Y\circ f_{\Bin}^*$. Since the target category is stable, it is enough to check that it is an equivalence when evaluated on any object (of the source category). For $\mc A\in\DBinAlg(X)$, $(f^*\circ U_X)(\mc A)\simeq (\#_Y\circ f^\dagger \circ U_X^{\mr{psf}}\circ i_X^\Bin)(\mc A)$, thus $(f^*\circ U_X)(\mc A)$ is the sheafification of the presheaf
	 	$$
	 	W\mapsto \colim_{f(W)\subset V} U(\mc A(V)).
	 	$$ 
	Similarly, using \Cref{lem:U and sheafification} one sees that $(U_Y\circ f_{\Bin}^*)(\mc A)$ is obtained as the sheafification of the functor 
	$$
	W\mapsto U\circ \colim_{f(W)\subset V} \mc A(V),
	$$
	and that the transformation above is induced by the natural maps 
	$$
	\colim_{f(W)\subset V} U(\mc A(V)) \ra U\circ \colim_{f(W)\subset V} \mc A(V),
	$$
	for $W\in \Op(Y)$ (here the colimit on the left is computed in ${\mscr D}(\mbb Z)$, while the colimit on the right is in $\DBinAlg$).
	Since $U$ commutes with filtered colimits these maps are equivalences, and thus $f^*\circ U_X\simeq U_Y\circ f_{\Bin}^*$.

\end{proof}

\begin{rem} \Cref{prop:properties of pullback and pushforward} shows in particular that functors $f^*_\Bin$ and $Rf_{*,\Bin}$ are compatible with functors $f^*$ and $Rf_*$ for ${\mscr D}(\mbb Z)$-valued sheaves, in the sense that their underlying objects in $\Shv(-,{\mscr D}(\mbb Z))$ are computed by the latter.
\end{rem}	

\begin{ex}(Constant derived binomial rings)\label{ex:constant binomial algebra}
	Let $X$ be a  topological space satisfying \Cref{ass:finite covering dimension} and let $p_X\colon X\ra \mr{pt}$ be the unique map. Given a derived binomial ring $A\in \DBinAlg\simeq \DBinAlg(\mr{pt})$ we define the \textit{constant binomial ring} $\ul A\in \DBinAlg(X)$ as the pull-back $\ul A\coloneqq p_{X,\Bin}^*A$. By \Cref{rem:sheafy pull-back is the sheafification of presheaf pull-back}, $\ul A$ can be also seen as sheafification of the constant presheaf $A\in \PShv(X,\DBinAlg)$, $U\mapsto A$.
	
	Recall that $\bZ\in \DBinAlg$ is the initial object; this way we see that $\ul \bZ\in \DBinAlg(X)$ (being the sheafification of the initial presheaf $\bZ\in \PShv(X,\DBinAlg)$) is the initial object of $\DBinAlg(X)$. By \Cref{prop:properties of pullback and pushforward} the underlying ${\mscr D}(\mbb Z)$-valued sheaf of $\ul \bZ$ is the constant sheaf.
\end{ex}	

If $X$ is also locally weakly contractible the pushforward of $\ul \bZ\in \DBinAlg(X)$ under $p_X$ can be described in terms of \Cref{ex:cohomology of topological spaces as a derived binomial algebra}. 
\begin{lem}\label{lem:sing vs Cech}
	Assume that $X$ satisfies (\ref{ass:finite covering dimension})  and is locally weakly contractible. Then there is a natural equivalence of derived binomial rings
	$$
	C^*_\sing(X,\mbb Z) \simeq Rp_{X*,\Bin}\ul\bZ \in \DBinAlg.
	$$
\end{lem}
\begin{proof}
We follow the proof in \cite[Theorem 12]{Clausen} replacing ${\mscr D}(\mbb Z)$ with $\DBinAlg$. We have a map of presheaves of binomial rings $\mbb Z \ra C^*_\sing(-,\mbb Z)$; since $X$ is weakly contractible this map is an isomorphism on stalks and thus induces an equivalence after sheafification. By \cite[Theorem 1.3]{Isaksen} the presheaf $C^*_\sing(-,\mbb Z)\colon U\mapsto C^*_\sing(U,\mbb Z)$ is in fact an $(\infty,1)$-sheaf; consequently, we get
$$
Rp_{X*,\Bin}\ul\bZ \simeq \RG(X,\ul\bZ) \simeq C^*_\sing(X,\mbb Z) \in \DBinAlg.
$$
\end{proof}

We end the subsection with the sheafy version of \Cref{lem:forgetful functor commutes with everything}. Note that the forgetful functor $(-)^\circ\colon\DBinAlg\ra \AAlg_{\mbb E_\infty}\!(\mbb Z)$ (see \Cref{not:forgetful functors from DAlg}) commutes with limits and induces a functor 
$$(-)^\circ\colon \DBinAlg(X)\ra \Shv(X,\AAlg_{\mbb E_\infty}\!(\mbb Z))$$ between the corresponding categories of sheaves. 

\begin{lem}\label{lem:forgetful sheaves to Einfty algebras commutes with colimits}
	The forgetful functor $(-)^\circ\colon \DBinAlg(X)\ra \Shv(X,\AAlg_{\mbb E_\infty}\!(\mbb Z))$ commutes will all small colimits.
\end{lem}	
\begin{proof}
	Indeed, the induced functor $(-)^\circ\colon \PShv(X,\DBinAlg) \ra \PShv(X,\AAlg_{\mbb E_\infty}\!(\mbb Z))$ commutes will all small colimits, since the colimits in the category of presheaves are computed pointwise. The statement for sheaves then follows since the sheafification functors also commutes with colimits.
\end{proof}	

\begin{notation}
We will denote the pushout of $\mc A$ and $\mc C$ along $\mc B$ in the category $\DBinAlg(X)$ by $\mc A\otimes_{\mc B}\mc C$.
\end{notation}

\begin{rem}\label{rem:all forgetful functors commute with pull/push}
	Similarly to the argument in \Cref{prop:properties of pullback and pushforward} one shows that the forgetful functor $$(-)^\circ\colon \DBinAlg(X)\ra \Shv(X,\AAlg_{\mbb E_\infty}\!(\mbb Z))$$ (as well as the similar forgetful functor to $\Shv(X,\DAlg(\mbb Z))$)  commutes with push-forwards and pull-backs.
\end{rem}	

\subsection{Functor $\mbf L \mc Bin_X^{\mr{coaug}}$}\label{ssec:LBin'}

In this section we define and study the functor $\mbf L \mc Bin^{\mr{coaug}}_X$ that takes as an input not just a sheaf $\mc F\in \Shv(X,{\mscr D}(\mbb Z))$, but rather a pointed (=coaugmented) sheaf  $\ul \bZ\ra \mc F$.

Let $X$ be a topological space satisfying \Cref{ass:finite covering dimension}. Let us denote by $\Shv(X,{\mscr D}(\mbb Z))_{\setminus \ul \bZ}$ the under category with respect to the constant sheaf $\underline \bZ$. Further, we will call an object of this category \textit{a pointed sheaf}. By \Cref{ex:constant binomial algebra} the constant sheaf $\ul \bZ$ has a structure of derived binomial ring that makes it an initial object of $\DBinAlg(X)$; thus there is a natural equivalence 
$$\DBinAlg(X)\simeq \DBinAlg(X)_{\setminus \mbb Z}.$$  We can define the corresponding forgetful functor $$U_{X,\setminus \ul \bZ}\colon \DBinAlg(X)\simeq \DBinAlg(X)_{\setminus \ul \bZ} \ra \Shv(X,{\mscr D}(\mbb Z))_{\setminus \ul \bZ}.$$

\begin{construction}[Functor $\mbf L \mc Bin_X^{\mr{coaug}}$]
 Note that the homomorphism of derived binomial rings
\begin{center}
	$\Bin(\bZ)\simeq \Map_{\mr{poly}}(\mbb Z,\mbb Z)\sa{\mathrm{ev_1}}\bZ;\,\,\,\, f \mapsto f(1)$
\end{center}
by sheafification gives rise to a map
$$\ev_1\colon \mbf L\mathcal Bin_X(\underline \bZ)\ra\underline \bZ$$
in $\DBinAlg(X)$.
Let us define a functor 
$$\mbf L\mc Bin_X^{\mr{coaug}}\colon \Shv(X,{\mscr D}(\mbb Z))_{\setminus \ul \bZ}\ra \DBinAlg(X)$$ 
which maps a pointed sheaf $s\colon \ul \bZ\ra \mc F$ to the pushout 
$$
\mbf L \mc Bin_X^{\mr{coaug}}(s) \coloneqq\mbf L\mathcal Bin_X(\mc F)\otimes_{\mbf L\mathcal Bin_X(\ul \bZ)}\ul \bZ
$$
of the diagram
\begin{equation}\label{Bin* pushout}
	\xymatrix{\mbf L\mathcal Bin_X(\mc F)\\
	\mbf L	\mathcal Bin_X(\ul \bZ)\ar[u]^{\mathcal Bin_X(s)}\ar[r]^(.61){\ev_1} & \ul \bZ
	}
\end{equation} 
in $\DBinAlg(X)$.

\end{construction}

For $\mc A\in \DBinAlg(X)$ let $\mbb I_{\mc A}\colon \ul \bZ \ra \mc A$ denote the unique up to homotopy map (see \Cref{ex:constant binomial algebra}).

\begin{thm}\label{free}
	Functor $\mbf L\mc Bin_X^{\mr{coaug}}$ is left-adjoint to the forgetful functor
	$$U_{X, \setminus \ul \bZ}\colon \DBinAlg(X)\ra \Shv(X,{\mscr D}(\mbb Z))_{\setminus \ul \bZ}.$$ 
\end{thm}
\begin{proof}
Note that for $(s_1\colon \bZ \ra \mc F_1), (s_2\colon \bZ \ra \mc F_2)\in \Shv(X,{\mscr D}(\mbb Z))_{\setminus \ul \bZ}$ the space of maps $\Map(s_1,s_2)$ is canonically identified with the homotopy fiber product
$$
\Map(s_1,s_2)\simeq \Map(\mc F_1,\mc F_2)\times_{\Map(\ul\bZ,\mc F_2)}\{s_2\},
$$
where the map of spaces $\Map(\mc F_1,\mc F_2)\ra \Map(\ul\bZ,\mc F_2)$ is induced by pre-composition with $s_1$. On the other hand, for $\mc A\in \DBinAlg(X)$ and $(s\colon \bZ \ra \mc F)\in \Shv(X,{\mscr D}(\mbb Z))_{\setminus \ul \bZ}$, we have a canonical equivalence
$$
\Map(\mbf L\mc Bin_X^{\mr{coaug}}(s), \mc A)\simeq \Map(\mbf L\mc Bin_X(\mc F),\mc A)\times_{\Map(\mbf L\mc Bin_X(\ul \bZ),\mc A)} \Map(\ul \bZ,\mc A)
$$
with the maps in the fiber product given by pre-composition with $\mbf L\mc Bin_X(s)$ and $\ev_1$. Note that $\Map(\ul \bZ,\mc A)\simeq \{*\}$ by \Cref{ex:constant binomial algebra}, and that the map of spaces $\Map(\mbf L\mc Bin_X(\mc F),\mc A)  \ra \Map(\mbf L\mc Bin_X(\ul \bZ),\mc A)$ can be identified with the map
$$
\Map(\mc F,U_X(\mc A))\ra \Map(\ul \bZ, U_X(\mc A))
$$
via the $\mbf L\mathcal Bin_X\dashv U_X$ adjunction. Moreover, the map $\mr{pt}\simeq \Map(\ul \bZ,\mc A)\ra \Map(\ul \bZ, U_X(\mc A))$ is given exactly by the canonical map $\mbb I_A\colon \ul \bZ \ra U_X(\mc A)$. This way we get a functorial in $s$ and $\mc A$ equivalence
$$
\Map(\mbf L\mc Bin_X^{\mr{coaug}}(s), \mc A)\simeq \Map(F,U_X(\mc A))\times_{\Map(\ul \bZ, U_X(\mc A))}\{\mbb I_A\}\simeq \Map(s, U_{X,\setminus \ul\bZ}(\mc A)).
$$
\end{proof}

\begin{rem} \label{Commutation with inverse images}
	Let $X,Y$ be topological spaces satisfying \Cref{ass:finite covering dimension} and let $f\colon Y\to X$ be a continuous map.  Then $\mbf L\mc Bin^{\mr{coaug}}_{-}$ commutes with the inverse image functors $f^*$: namely, there is a natural equivalence of functors
	$$
	f^*_{\Bin}\mbf L\mc Bin^{\mr{coaug}}_{X}(-)\simeq \mbf L\mc Bin^{\mr{coaug}}_{Y}(f^*(-)).
	$$ 
	Indeed, by \Cref{prop:properties of pullback and pushforward} $f^*_\Bin$ and $f^*$ intertwine the diagram \Cref{Bin* pushout} for $X$ and $Y$, and also commute with colimits. 
\end{rem}

\begin{ex}\label{ex:Bin^* in split case}
	Let $(s\colon \ul{\bZ} \ra \mc F)\in \Shv(X,{\mscr D}(\mbb Z))_{\setminus \ul \bZ}$ be a pointed sheaf on $X$ and suppose we have fixed a splitting $t\colon \mc F\ra \ul{\bZ}$. Then $t$ induces an equivalence 
	$$
	\mbf L\mc Bin_X^{\mr{coaug}}(s)\simeq \mbf L \mc Bin_X(\ol{\mc F}),
	$$
	where $\ol{\mc F}\coloneqq \cofib(s)$. Indeed, $t$ induces an equivalence $\mc F\simeq \ul \bZ \oplus \ol{\mc F}$, with $s$ identified with the embedding of the first summand. Since $\mbf L \mc Bin_X$ commutes with colimits, we get an equivalence $\mbf L \mc Bin_X( \mc F)\simeq \mbf L \mc Bin_X( \ol{\mc F}) \otimes \mbf L \mc Bin_X(\ul \bZ)$ and then
	$$
	\mbf L\mc Bin_X^{\mr{coaug}}(s)\simeq \mbf L \mc Bin_X( \ol{\mc F}) \otimes \mbf L \mc Bin_X(\ul \bZ)\otimes_{ \mbf L \mc Bin_X(\ul \bZ)}\ul \bZ\simeq \mbf L \mc Bin_X( \ol{\mc F}).
	$$
\end{ex}

\subsection{Cohomology of the Kato-Nakayama space}\label{ssec:cohomology of KN-space}

We return to the setup of \Cref{ssect:log_Betti_coh}. Namely, let $(X, \mathcal M\sa{\alpha}\cO_X)$ be an fs log complex analytic space and $\pi\colon X^{\log}\ra X$ be the associated Kato-Nakayama space. We also consider a complex $\mathpzc{exp(\alpha)}\in \Shv(X,{\mscr D}(\mbb Z))$ defined as the fiber of the map
$$\cO_{X}\xra{\exp_{\mc M}}\mathcal M^{\gr}.$$

Now, we would like to relate $\eexp(\alpha)$ and $R\pi_*\ul \bZ$. Let us  denote by $C(\bR)\colon U \mapsto \Map_{\mr{cont}}(U,\mbb R)$ and $C(S^1)\colon U\mapsto \Map_{\mr{cont}}(U,S^1)$ the sheaves of $\bR$-valued and $S^1$-valued continuous functions correspondingly. Note that the universal cover $\bR\ra S^1$ (given by mapping $x\in \bR$ to $e^{2\pi i x}\in S^1\subset \bC$) induces a short exact sequence
\begin{equation}\label{eq:ses for Z}
	0\ra \ul \bZ\ra C(\bR)\sa{\exp}C(S^1)\ra 0
\end{equation}	
of sheaves of abelian groups on $X^{\log}$.

In the case of the Kato-Nakayama space $X^{\log}$ we can further apply $R\pi_*$ to \Cref{eq:ses for Z}, this way getting a fiber sequence 
$$
R\pi_* \ul\bZ \ra R\pi_*C(\bR)\xra{\exp} R\pi_*C(S^1)
$$
in $\Shv(X,{\mscr D}(\mbb Z))$. One then also has a natural map of complexes
\begin{equation}\label{from exp to exp}
	\xymatrix{\cO_{X}\ar[r]^{\exp_{\mc M}}\ar@{-->}[d] & \mathcal{M}^{\gr}\ar@{-->}[d]\\
		\pi_{*}C(\bR)\ar[r]^{\exp} & \pi_{*}C(S^{1})
	}
\end{equation}
 defined by the formulas 
\begin{center}
	$f\mapsto 2\pi i \cdot \mathrm{Re}(\pi^*(f)));\,\,\,\,\,\,\, \mathrm m\mapsto \mathrm{arg(m)}$\footnote{Here $\Re$ the real part of a complex-valued function; also see \Cref{constr:Kato-Nakayama space} for definition of $\arg(m)\in \Map_{\mr{cont}}(X^{\log},S^1)$.},
\end{center}
where $f\in \cO_{X}(U)$, $\mathrm m\in \mathcal M^{\gr}(U)$ and $U\subset X$. On the other hand, we have the natural transformation $\pi_*\to R\pi_*$. So we also get a commutative diagram
\begin{equation}
	\xymatrix{
		\mc O_X \ar[r]^{{\exp_{\mc M}}}\ar[d] &\mc M^{\gr}\ar[d] \\
		R\pi_{*}C(\bR)\ar[r]^{\exp} & R\pi_{*}C(S^{1})
	}	
\end{equation} 
which induces a map
\begin{equation}\label{eq:map_from_exp_ to_the_first_truncation}
c\colon \eexp(\alpha)\simeq \fib(\mc O_X \xra{\exp_{\mc M}} \mc M^{\gr}) \tto R\pi_*\ul \bZ.
\end{equation}
\begin{rem}\label{rem:exp factors through first truncation}
Note that since $\eexp(\alpha)$ only has non-trivial 0-th and 1-st cohomology sheaves, the map $\eexp(\alpha)\ra R\pi_*\ul\bZ$ factors through $R^{\le 1}\pi_*\ul\bZ \coloneqq \tau^{\le 1}(R\pi_*\ul \bZ)$. Also, the sheaf $C(\mbb R)$ is soft and so $R^{>0}\pi_* C(\mbb R)\simeq 0$, and this way $R^{\le 1}\pi_*\ul\bZ$ can be represented by the 2-term complex
$$
\pi_*C(\mbb R)\ra \pi_*C(S^1)
$$
placed in cohomological degrees 0 and 1.
\end{rem}
The following lemma has been proven among other things in \cite{shuklin} (see also \cite[proof of Theorem 4.2.2]{Achinger_Ogus}, or \cite[Lemma 2.7]{Steenbrink} in the case of normal crossings divisor). We include the proof for reader's convenience.
\begin{lem}[{\cite[Lemma 5.5]{shuklin}}]\label{lem:exp is the first truncation}
	Map \Cref{eq:map_from_exp_ to_the_first_truncation} induces a quasi-isomorphism 
	$$\eexp(\alpha)\iso R^{\le 1}\pi_*\ul \bZ.$$
\end{lem}
\begin{proof} By \Cref{rem:map pi} the map $\pi$ is proper with connected fibers, which allows us to identify $\pi_*\ul{\bZ}$ with $\ul\bZ$, and it is easy to see that the map \Cref{eq:map_from_exp_ to_the_first_truncation} induces an equivalence on $H^0$. Note also that $H^1(\eexp(\alpha))=\coker(\cO^\times_{X}\to \mathcal M^{\gr})=\overline{\mathcal M^{\gr}}$; it thus suffices to check that the induced morphism $\overline{\mathcal M^{\gr}}\to R^1\pi_*\bZ$ is an isomorphism.
	
	The latter is enough to do on stalks at all points $x\in X$. By proper base change
	$$
	(R^{\le 1}\pi_*\ul \bZ)_x\simeq \tau^{\le 1}\RG(\pi^{-1}(x),\ul \bZ),
	$$ 
	where $\pi^{-1}(x)\simeq \Hom(\ol{\mathcal M^{\gr}},S^1)$. Note that similarly to the \Cref{rem:exp factors through first truncation} we can identify $H^1(\pi^{-1}(x),\ul \bZ)$ with the cokernel of the map 
	$$
	\Map_{\mr{cont}}(\pi^{-1}(x), \mbb R) \ra \Map_{\mr{cont}}(\pi^{-1}(x), S^1).
	$$Tracing through the definition of the map $m\mapsto \arg(m)$ one sees that the corresponding map $\overline{\mathcal M^{\gr}_x} \ra H^1(\pi^{-1}(x),\ul \bZ)$ can be represented by the diagram
	$$
	\xymatrix{& \ol{\mathcal M^{\gr}_{x}} \ar[d]\\
\Map_{\mr{cont}}(\pi^{-1}(x), \mbb R)\ar[r]&	\Map_{\mr{cont}}(\pi^{-1}(x), S^1)}
	$$
	where (via the identification $\pi^{-1}(x)\simeq \Hom(\ol{\mathcal M^{\gr}_{x}},S^1)$) the map 
	$$
	 \ol{\mathcal M^{\gr}_{x}} \ra \Map_{\mr{cont}}(\pi^{-1}(x), S^1)
	$$
	is given by sending $m$ to the "evaluation at $m$ map" $\ev_m\colon \Hom(\ol{\mathcal M^{\gr}_{x}},S^1) \ra S^1$, $\phi\mapsto \phi(m)$. By \Cref{rem:pairing for torus} below this map induces an isomorphism 
	$$
	\ol{\mathcal M^{\gr}_{x}} \xra{\sim} \coker(\Map_{\mr{cont}}(\pi^{-1}(x), \mbb R) \ra \Map_{\mr{cont}}(\pi^{-1}(x), S^1))\simeq H^1(\pi^{-1}(x),\ul{\mbb Z}).
	$$

\end{proof}
\begin{rem}\label{rem:pairing for torus}
		Let $L\in \Lat$ be a lattice, and let $\mbb T_{L^\vee}\coloneqq \Hom(L,S^1)$ be the compact real torus given by $S^1$-valued characters of $L$.  Then the composite identification 
		$$
		\coker(\Map_{\mr{cont}}(\mbb T_{L^\vee},\mbb R)\ra \Map_{\mr{cont}}(\mbb T_{L^\vee},S^1))\xra{\sim}H^1(\mbb T_{L^\vee},\ul\bZ) \simeq H^1_\sing(\mbb T_{L^\vee},\mbb Z)\simeq L
		$$
	(with the first isomorphism coming from the short exact sequence \Cref{eq:ses for Z}) is in fact given by the map 
	$$
	g\colon \Map_{\mr{cont}}(\mbb T_{L^\vee},S^1)\ra H^1_\sing(\mbb T_{L^\vee},\mbb Z)
	$$ that sends $\phi\in \Map_{\mr{cont}}(\mbb T_{L^\vee},S^1)$ to the pull-back $\phi^*[S^1]\in H^1_\sing(\mbb T_{L^\vee},\mbb Z)$ of the fundamental class $[S^1]\in H^1_\sing(S^1,\bZ)$. 
	
	Note that we have a map $L\ra \Map_{\mr{cont}}(\mbb T_{L^\vee},S^1)$ sending $\ell \in L$ to the evaluation map 
	$$
	\ev_{\ell}\colon \Hom(L,S^1) \ra S^1, \quad f\mapsto f(\ell).
	$$
	Then its composition with $g$ is the identity: indeed, we have a natural identification $$H_1^\sing(\mbb T_{L^\vee},\mbb Z)\simeq L^\vee$$ with $\lambda\in L^\vee$ producing a map $f_\lambda\colon S^1\ra \mbb T_{L^\vee}$, and the class $\ev_{\ell}^*[S^1]$ pairs with the 1-cycle given by $f_\lambda(S^1)$ as $\lambda(\ell)\in \mbb Z$. 
\end{rem}	

Now it remains to show that in the above situation $R\pi_*\ul \bZ$ can be reconstructed from $R^{\le 1}\pi_*\ul \bZ$ (coaugmented via the natural equivalence $\pi_*\ul\bZ\simeq \ul \bZ$) by applying $\mbf L\mc Bin_X^{\mr{coaug}}$. The following can be considered as a (partial) generalization of Proposition \ref{prop: description of free coconnective binomial algebras} in the relative setting.
\begin{thm}\label{Bin* main th}
	Let $\pi\colon Y\ra X$ be a map of Hausdorff topological spaces satisfying \Cref{ass:finite covering dimension}. Suppose that 
	\begin{itemize}
		\item[1)] $\pi$ is proper;
		\item[2)] for any $x\in X$ the fiber $\pi^{-1}(x)\hookrightarrow Y$ is a compact finite dimensional real torus\footnote{Here the dimension of torus can vary with the point.}.
	\end{itemize} 
Consider $R^{\le 1}\pi_*\ul \bZ\in \Shv(X,{\mscr D}(\mbb Z))_{\setminus \ul \bZ}$ with the pointing $\ul \bZ\ra R^{\le 1}\pi_*\ul \bZ$ given by $\ul \bZ \simeq \pi_*\ul \bZ \ra R^{\le 1}\pi_*\ul \bZ$.
	Then there is a natural quasi-isomorphism of sheaves of derived binomial rings 
	$$\mbf L\mc Bin_X^{\mr{coaug}}(R^{\le 1}\pi_*\ul \bZ)\xra{\sim} R\pi_{*,\Bin}\ul \bZ $$
	\begin{proof}
		We have a map 
		$$
		(\pi_*\ul \bZ \ra R^{\le 1}\pi_*\ul \bZ) \tto (\pi_*\ul \bZ \ra R\pi_*\ul \bZ) \in \Shv(X,{\mscr D}(\mbb Z))_{\setminus \ul \bZ}.
		$$ 
		By \Cref{free} this induces a map 
		$$
		\beta\colon \mbf L\mc Bin_X^{\mr{coaug}}(R^{\le 1}\pi_*\ul \bZ) \ra R\pi_{*,\Bin}\ul \bZ
		$$
		in $\DBinAlg(X)$, and it is enough to check that it is an equivalence on fibers at all points of the underlying complexes. 
		
		Let $x\in X$ be a point, and let $i_x\colon \mr{pt} \ra X $ be the corresponding embedding. Also let $\pi_x\colon \pi^{-1}(x)\ra x$ be the corresponding projection. Note that $i_x^*$ is $t$-exact. Applying proper base change and \Cref{Commutation with inverse images} we see that the restriction of $\beta$ to $x$ can be identified with the analogous map
		\begin{equation}\label{eq:map from Bin^*}
		\mbf L\mc Bin_X^{\mr{coaug}}(R^{\le 1}\pi_{x,*}\ul \bZ) \ra R\pi_{x*,\Bin}\ul \bZ
		\end{equation}
		for $\pi_x\colon \pi^{-1}(x)\ra \mr{pt}$. Thus we can reduce to the case of the projection $p_{\mbb T}\colon \mbb T\ra \mr{pt}$ of a compact real torus to the point. Picking a point $t\in \mbb T$ gives a splitting $\RG^{\le 1}(\mbb T,\ul \bZ)\simeq C^{\le 1}_{\sing}(\mbb T,\bZ)\simeq \bZ \oplus H^1_\sing (\mbb T,\mbb Z)[-1]$, and also identifies $\mbb T$ with $K(H_1^\sing (\mbb T,\mbb Z)[1])$. Then via \Cref{ex:Bin^* in split case}, the map \Cref{eq:map from Bin^*} is identified with the equivalence
		$$
		\mbf L \Bin(M^\vee) \xra{\sim} C^*_\sing(K(M),\mbb Z)
		$$
		from \Cref{prop: description of free coconnective binomial algebras} in the case $M=H_1^\sing (\mbb T,\mbb Z)[1]$.
	\end{proof}	
\end{thm}

\begin{ex}\label{ex:when first truncations splits}
		Assume that the first truncation $R^{\le 1}\pi_* \ul\bZ$ splits as $\ul\bZ\oplus (R^1\pi_*\ul \bZ)[-1]$ and let $\mc L_{\pi}\coloneqq R^1\pi_*\ul \bZ$. Then in this case by \Cref{ex:Bin^* in split case} we also get a (functorial in the splitting) equivalence
			$$
	R\pi_{*,\Bin}\ul\bZ \xra{\sim} \mbf L\mc Bin_X (\mc L_{\pi}[-1]).
	$$
	Moreover, one sees that the natural filtration by  $\mbf L\mc Bin^{\le *}_X(\mc L_{\pi}[-1])$  (obtained by sheafification) coincides with the Postnikov filtration: indeed, the $n$-th associated graded piece is given by (the defined analogously) $\mbf L\Gamma_X^n(\mc L_{\pi}[-1])\simeq (\wedge^n\mc L_{\pi})[-n]$, which is concentrated in a single cohomological degree $n$. One can see from this that the way different $R^i\pi_*\ul{\bZ}$'s are glued together inside $R\pi_*\ul\bZ$ is controlled by how functors $\Gamma^i$ are glued together inside $\Bin$. We plan to study this relation in more detail in the sequel.
\end{ex}

Finally, we get the  formula \Cref{intro_eq:pushforward as Bin} for $\pi\colon X^{\log}\ra X$ as a consequence of \Cref{lem:exp is the first truncation} and \Cref{Bin* main th}:
\begin{thm}\label{thm:cohomology of Kato-Nakayama space} Let $(X,\mathcal M)$ be an fs log analytic space over $\bC$. Let $\pi\colon X^{\log}\ra X$ be the natural map from the Kato-Nakayama space back to $X$.  Let $\eexp(\alpha)=\fib(\exp_{\mc M}\colon \mc O_X \ra \mc M^{\gr})$, as in \Cref{not:exp(a)}. Then there is a natural equivalence of sheaves of derived binomial rings 
	$$R\pi_{*,\Bin}\ul\bZ \xra{\sim} \mbf L\mc Bin_X^{\mr{coaug}}(\eexp(\alpha)).$$
	Consequently, we also get a closed formula for the log Betti cohomology of $X$:
	$$\RG^{\log}_{\mr{Betti}}(X,\bZ)\xra{\sim} \RG(X,\mbf L\mc Bin_X^{\mr{coaug}}(\eexp(\alpha))).$$
\end{thm}

\begin{rem}
	Via \Cref{rem:all forgetful functors commute with pull/push} we also get formulas for the pushforward $R\pi_{*}\ul\bZ$ in the categories of sheaves of $\mbb E_\infty$-algebras or derived commutative algebras as the corresponding underlying object of $\mbf L\mc Bin^{\mr{coaug}}_X(\eexp(\alpha))$.
\end{rem}	

\subsection{Cosheaf of spaces $[X^{\log},\pi]$ and its properties}
\label{ssec:cosheaf shit}

The goal of this section is to try to reconstruct the Kato-Nakayama space $\pi\colon X^{\log}\ra X$ from just knowing the pushforward $R\pi_{*}\ul\bZ$ viewed as a sheaf of binomial algebras. While the data of an honest topological space is too much to handle by a homotopy-algebraic object like $R\pi_{*}\ul\bZ$, we will show that one can at least reconstruct a rather fine homotopy-theoretic approximation: namely, the compatible data of spaces $\pi^{-1}(U)$ for open subsets $U\subset X$ considered up to homotopy; this data is naturally packaged into a \textit{cosheaf of spaces} on $X$ (see \Cref{constr:cosheaf represented by a space}).
\begin{notation}\label{not:cosheaves}
	Let $\mscr C$ be an $\infty$-category with small colimits and let $X$ be a topological space. We denote by $\hcoShv(X,\mscr C)\coloneqq (\hShv(X,\mscr C^\op))^\op$ the category of $\mscr C$-valued hyper-cosheaves on $X$. By a similar argument to \cite[Lemma 2.21]{Volpee} one gets that for cocomplete $\mscr C$ the left Kan extension along $\Op(X)\ra \Fun(\Op(X)^\op,\Spc) \xra{} \hShv(X,\Spc)$ identifies  $\hcoShv(X,\mscr C)\subset \Fun(\Op(X)^{\op}, \mscr C^\op)^{\op}\simeq \Fun(\Op(X), \mscr C)$ with
	$$
	\hcoShv(X,\mscr C) \simeq \Fun^{\mr{L}}(\hShv(X,\Spc),\mscr C)
	$$
	(where $ \Fun^{\mr{L}}$ denotes the subcategory of small colimit preserving functors). In particular, by \cite[Proposition 5.5.3.8]{Lur_HTT} if $\mscr C$ is presentable, then so is $\hcoShv(X,\mscr C)$. Via the adjoint functor theorem we get a right adjoint $R\colon \CoPShv(X,\mscr C) \ra \hcoShv(X,\mscr C)$ to the embedding $\hcoShv(X,\mscr C)\subset \CoPShv(X,\mscr C)\coloneqq \Fun(\Op(X), \mscr C)$, which we will call \textit{hypercosheafification}. As usual one has an explicit formula for it: namely
	$$
	R(\mc F)(U)\simeq \lim_{V_\bullet \ra U} |\mc F(V_\bullet)|\in \mscr C
	$$
	where the limit is taken over all hypercovers $V_\bullet \ra U$, and $|\mc F(V_\bullet)|$ denotes geometric realization of the simplicial object $\mc F(V_\bullet)$.
	
	We will be primarily interested in the cases where $\mscr C$ is given by $\Spc$, ${\mscr D}(\mbb Z)^{\le 0}$, $\Stk_{/\mbb Z}$ or $\Ab(\Stk_{/\mbb Z})$. We have natural adjoint functors 
	$$
	C_*^\sing(-,\mbb Z)\colon \hcoShv(X,\Spc) \ra 	\hcoShv(X,{\mscr D}(\mbb Z)^{\le 0}), \qquad  K\colon \hcoShv(X,{\mscr D}(\mbb Z)^{\le 0}) \ra 	\hcoShv(X,\Spc)
	$$
	and also the functor
		$$
	  K\colon \hcoShv(X,\Ab(\Stk_{/\mbb Z})) \ra 	\hcoShv(X,\Stk_{/\mbb Z})
	$$
	induced by the functors in Constructions \ref{constr:generalized EM-space} and \ref{constr:abelian group objects} correspondingly. Note that $C_*^\sing(-,\mbb Z)$ can be computed on the level of copresheaves since $C_*^\sing(-,\mbb Z)\colon \Spc \ra {\mscr D}(\mbb Z)$ commutes with colimits, but $K$ needs to be cosheafified. We also have functors
	$$
	\iota\colon \hcoShv(X,\Spc) \ra \hcoShv(X,\Stk_{/\mbb Z}), \qquad \iota\colon \hcoShv(X,{\mscr D}(\mbb Z)^{\le 0}) \ra \hcoShv(X,\Ab(\Stk_{/\mbb Z}))
	$$
	which are induced by the colimit preserving functor $\Spc \ra \Stk_{/\mbb Z}$ given by constant sheaf, and the identification ${\mscr D}(\mbb Z)^{\le 0}\simeq \Ab(\Spc)$. 
\end{notation}	

\begin{construction}\label{constr:cosheaf represented by a space}
	Let $\pi\colon Y\ra X$ be a map of topological spaces. Then to $\pi$ we can associate a cosheaf $[Y,\pi]\in \hcoShv(X,\Spc)$ by putting $[Y,\pi](U)\in \Spc$ for $U\in \Op(X)$ to be\footnote{More precisely, this means that we take the geometric realization $|\mr{Sing}(\pi^{-1}(U))_{\bullet}|$ of the simplicial set of continuous simplices of $\pi^{-1}(U)$.} $\pi^{-1}(U)\in \Spc$; by \cite{Isaksen} this indeed is a hyper-cosheaf. We will occasionally denote the corresponding object of $\hcoShv(X,\Spc)$ simply by $Y$.
	
	For any $n\ge 0$ let us denote by $\tau_{/X,\le n}([Y,\pi])$ (or simply $\tau_{/X,\le n} Y$) the relative $n$-truncation of $Y$: namely, the hypercosheafification of the copresheaf $U\mapsto \tau_{\le n}(\pi^{-1}(U))$ where $\tau_{\le n}\colon \Spc\ra \Spc$ denotes the $n$-truncation functor (see e.g. \cite[Proposition 5.5.6.18]{Lur_HTT}). For any $n\ge 0$ we have a natural map 
	$$
	[Y,\pi]\ra \tau_{/X,\le n}([Y,\pi]).
	$$
\end{construction}	

\begin{ex}\label{ex:final object in hypercosheaves}
	Assume $X$ is locally weakly contractible and consider the identity map $\id\colon X\ra X$. Then we claim that the associated object $[X,\id]\in \hcoShv(X,\Spc)$ is equivalent to the cosheafification of the constant copresheaf given by the point. Indeed, one has\footnote{Since $\{*\}$ is a terminal object of $\Spc$.} a map of copresheaves $[X,\id]\ra \{*\}$  which is locally an equivalence since $X$ is locally weakly contractible; thus it induces an equivalence after cosheafifying the target. In particular, in this case $[X,\id]$ is the final object of $\hcoShv(X,\Spc)$.
	
	More generally, one can see that the cosheafification of a constant copresheaf given by (the underlying homotopy type) of some topological space $S$ is explicitly given by $[X\times S,p_X]$ for the projection $p_X\colon X\times S\ra X$.
\end{ex}	

\begin{construction}[Relative truncated singular homology] Let $\pi\colon Y\ra X$ be a map of topological spaces and consider the cosheaf 
	$$C_*^\sing(\pi,\mbb Z)\coloneqq C_*^\sing([Y,\pi],\mbb Z) \in \hcoShv(X,{\mscr D}(\mbb Z)^{\le 0})$$
	sending $U\mapsto C_*^\sing(\pi^{-1}(U), \mbb Z)$.
	For each $n\ge 0$ let us denote by $C_{\le n}^\sing(\pi,\mbb Z)\in \hcoShv(X,{\mscr D}(\mbb Z)^{\le 0})$ the corresponding relative $n$-truncation: namely, the cosheafification of the copresheaf $U\mapsto C^\sing_{\le n}(\pi^{-1}(U),\mbb Z)$. 
	
	For any $\infty\ge m\ge n$ there exists a natural map $C_{\le m}^\sing(\pi,\mbb Z)\ra C_{\le n}^\sing(\pi,\mbb Z)$ and by adjunction we also have a natural map 
	$$
	[Y,\pi]\ra K(C_*^\sing(\pi,\mbb Z)),
	$$
	which we can also compose with the above one to get compatible maps 
	$$
	[Y,\pi] \ra K(C_{\le n}^\sing(\pi,\mbb Z))
	$$
	for any $n\ge 0$.
\end{construction}

\begin{rem}
	Let $Y\in \Spc$ be a space. The natural map of spaces $Y\ra K(C_*^\sing(Y,\mbb Z))$ induces a map $\tau_{\le 0}Y\simeq \pi_0(Y) \ra K(C_{\le 0}^\sing(Y,\mbb Z))$ by sending a (path-)connected component $Y_s\in \pi_0(Y)$ to the class $[y_s]$ of any point $y_s\in Y_s$. Moreover, the induced map $\mbb Z\langle\pi_0(Y)\rangle \ra C_{\le 0}^\sing(Y,\mbb Z)$ is an isomorphism, and so $K(C_{\le 0}^\sing(Y,\mbb Z))\simeq \pi_0(Y)\times \mbb Z$, where the map $\tau_{\le 0}Y\ra \pi_0(Y)\times \mbb Z$ is identified with the embedding $\pi_0(Y)\times \{1\}\hookrightarrow \pi_0(Y)\times \mbb Z$. This way, for any space $Y$ we get a commutative square
	\begin{equation}\label{eq:commutative square for truncations of spaces}
	\xymatrix{Y\ar[r]\ar[d]& K(C_{\le 1}^\sing(Y,\mbb Z)) \ar[d]\\
	\tau_{\le 0}Y\ar[r] & K(C_{\le 0}^\sing(Y,\mbb Z)).}
	\end{equation}
	
	If now $\pi \colon Y\ra X$ is a continuous map, by cosheafifying the above square for $f^{-1}(U)$ we get a commutative square 
	\begin{equation}\label{eq:commutative square for any cosheaf of spaces}
	\xymatrix{[Y,\pi]\ar[r]\ar[d]& K(C_{\le 1}^\sing(\pi,\mbb Z)) \ar[d]\\
		\tau_{/ X,\le 0}([Y,\pi])\ar[r] & K(C_{\le 0}^\sing(\pi,\mbb Z)),}
	\end{equation}
in $\hcoShv(X,\Spc)$, where we can also identify the bottom horizontal arrow with the embedding $$\tau_{/ X,\le 0}([Y,\pi])\times\{1\}\hookrightarrow \tau_{/ X,\le 0}([Y,\pi])\times \mbb Z.$$
\end{rem}	

\begin{lem}\label{lem:union of tori is a fiber product} Let $Y\in \Spc$ be weakly homotopy equivalent to a finite dimensional compact real torus. Then the square \Cref{eq:commutative square for truncations of spaces} is a fiber square.
\end{lem}	
\begin{proof} Let $Y$ be homotopy equivalent equivalent to a torus $\mbb T$.
	Picking a point $x\in \mbb T$ we can split $C^\sing_{\le 1}(\mbb T, \mbb Z)\simeq \mbb Z\oplus H_1^\sing(\mbb T,\mbb Z)[1]$ and identify $\mbb T\simeq K(L,1)$ where $L=H_1^\sing(\mbb T,\mbb Z)$. Then \Cref{eq:commutative square for truncations of spaces} is identified with  
	$$
	\xymatrix{\{1\}\times K(L,1)\ar[rr]^{y\mapsto ([x],y)}\ar[d]&& \mbb Z \times K(L,1) \ar[d]\\
		\{1\} \ar[rr]^{1\mapsto [x]} && \mbb Z,}
	$$
	which is a fiber square.
\end{proof}	

\begin{cor}\label{cor:toric map}
	Let $\pi\colon Y\ra X$ be a map of topological spaces, and assume that there is a base $\mc B$  of weakly contractible open subsets $\{U_i\}_{i\in I}$ in $X$ such that $\pi^{-1}(U_i)$ is homotopy equivalent to a compact real torus. Then \Cref{eq:commutative square for any cosheaf of spaces} is a fiber square and takes the following form
	$$
	\xymatrix{[Y,\pi]\ar[r]\ar[d]& K(C_{\le 1}^\sing(\pi,\mbb Z)) \ar[d]\\
		[X\times \{1\},\id]\ar[r] & [X\times \mbb Z,p_X].}
	$$
\end{cor}	

\begin{proof}
	Denote by $Z$ the fiber product of terms in \Cref{eq:commutative square for any cosheaf of spaces}. By the explicit description of hypercosheafification for any $U\in \Op(X)$ we get an equivalence
	$$
	Z(U)\xra{\sim} \lim_{V_\bullet \ra U} |Z'(V_\bullet)|
	$$
	where the limit runs over all hypercovers $V_\bullet \ra U$ and for every term $V_i\coloneqq \sqcup_s \{U_s\subset U\}$ the space $Z'(V_i)$ is defined as the union over $s$ of (homotopy) fiber products
	$$
	\xymatrix{Z'(U_s) \ar[r]\ar[d]&K(C^\sing_{\le 1}(\pi^{-1}(U_s),\mbb Z))\ar[d]\\
	\pi_0(\pi^{-1}(U_s))\ar[r]& K(C^\sing_{\le 0}(\pi^{-1}(U_s),\mbb Z)).
}
	$$
	By our assumptions, there is a cofinal family of hypercoverings $V_\bullet \ra U$ where each connected component $U_s\subset U$ of each term $V_i$ appearing belongs to $\mc B$. The map $Y(U) \ra Z(U)$ then can be represented by 
	$$
	\lim_{\stackrel{V_\bullet \ra U}{V_i = \sqcup_{j} U_s, U_s\in \mc B}} |Y(V_\bullet)| \ra \lim_{\stackrel{V_\bullet \ra U}{V_n = \sqcup U_s, U_s\in \mc B}} |Z'(V_\bullet)|.
	$$ 
	This map is an equivalence since for each $i$ the map 
	$$
	Y(V_i) \simeq \sqcup_s Y(U_s) \ra \sqcup_s Z'(U_s) \simeq Z'(V_i) 
	$$
	is an equivalence by our assumptions on $\mc B$ and \Cref{lem:union of tori is a fiber product}. Since $\pi^{-1}(U_i)$ is connected for any $U_i\in \mc B$, by a similar argument we also see that $\tau_{/X,\le 0}Y$ is identified with the cosheafification of the constant copresheaf given by $\{*\}$, and so is equivalent to $\id\colon X\ra X$ by \Cref{ex:final object in hypercosheaves}. Meanwhile $K(C_{\le 0}^\sing(Y,\mbb Z))\simeq \mbb Z\times \tau_{/X,\le 0}Y \simeq \mbb Z\times X$ and the map $\tau_{/X,\le 0}Y \ra K(C_{\le 0}^\sing(Y,\mbb Z))$ is identified with embedding $\{1 \}\times X \hookrightarrow \mbb Z\times X$. 
\end{proof}

\begin{rem}[Gerbe structure on $Y$]\label{rem:gerbe structure on Y} Let $\pi\colon Y\ra X$ be as in \Cref{cor:toric map} and let $\mc L_\pi^\vee\simeq H_1(\pi,\mbb Z)\in \hcoShv(X,{\mscr D}(\mbb Z)^{\le 0})$ denote cosheafification of the copresheaf given by $U\mapsto H_1^\sing(\pi^{-1}(U),\mbb Z)$. One can think of $\mc L_\pi^\vee$ as a dual version of $R^1\pi_*\ul\bZ$. We have a natural map $\mc L_\pi^\vee[1] \ra C^\sing_{\le 1}(\pi,\mbb Z)$, inducing the map $B\mc L_\pi^\vee \ra K(C^\sing_{\le 1}(\pi,\mbb Z))$ (here we denoted by $B\mc L_\pi^\vee$ the cosheaf of spaces given by $K(\mc L_\pi^\vee[1])$).
	
	The abelian group object structure  on $C^\sing_{\le 1}(\pi,\mbb Z)$ provides in particular the addition map $$K(C^\sing_{\le 1}(\pi,\mbb Z))\times K(C^\sing_{\le 1}(\pi,\mbb Z)) \ra K(C^\sing_{\le 1}(\pi,\mbb Z)).$$
	The induced map 
	$$
	B\mc L_\pi^\vee\times K(C^\sing_{\le 1}(\pi,\mbb Z)) \ra K(C^\sing_{\le 1}(\pi,\mbb Z))
	$$ 
	then preserves the fibers of the projection $K(C^\sing_{\le 1}(\pi,\mbb Z))\ra K(C^\sing_{\le 0}(\pi,\mbb Z))\simeq [X\times \mbb Z,p_X]$. In particular, it induces an action 
	$$
	B\mc L_\pi^\vee \times [Y,\pi] \ra [Y,\pi]
	$$
via \Cref{cor:toric map}. Moreover, it endows $[Y,\pi]$ with the structure of a torsor over $B\mc L_\pi^\vee $, meaning that on any $U\in \mc B$ it can be identified with the addition map $B\mc L_\pi^\vee(U) \times B\mc L_\pi^\vee(U) \ra B\mc L_\pi^\vee(U)$. This way we can view the cosheaf of spaces $[Y,\pi]\in \hcoShv(X,\Spc)$ represented by $\pi\colon Y\ra X$ as a sort of gerbe over the cosheaf $\mc L_\pi^\vee\in \hcoShv(X,{\mscr D}(\mbb Z)^{\le 0})$. 

\end{rem}	

\begin{rem}
	Alternatively, one can also identify $\mc L_\pi^\vee[1]\in \hcoShv(X,{\mscr D}(\mbb Z)^{\le 0})$ with the tensor product $\mc L_\pi^\vee \otimes S^1$ where $S^1$ is the cosheafification of the constant presheaf given by $S^1\in \Ab(\Spc)\simeq {\mscr D}(\mbb Z)^{\le 0}$. This way one can view $Y$ as a torsor over $\mc L_\pi^\vee \otimes S^1$ (in the sense of \Cref{rem:gerbe structure on Y}).
\end{rem}	

Let us now supply some examples for \Cref{cor:toric map}. One of the suitable but also rather general setups is given by \textit{real analytic maps} between \textit{real semi-analytic spaces}. More precisely, consider the following definition:
\begin{defn}\label{def:locally semi-analytic}
	We will call a map $\pi\colon Y\ra X$ of topological spaces \textit{locally semi-analytic} if locally on $X$ it is homeomorphic to the restriction of a real semi-analytic map $f\colon \mbb R^m\ra \mbb R^n$ to some relatively compact semi-analytic subsets $V\subset \mbb R^m$ and $U\subset \bR^n$ (in particular, we want $f(V)\subset U$).
\end{defn}	 The key result in this context is the following stronger version of base change theorem, for which we were not able to find a reference, so include a proof.

\begin{prop}\label{prop:local homotopy triviality of subanalytic maps}
Let $\pi\colon Y\ra X$ be a proper map between topological spaces that is locally semi-analytic in the sense of \Cref{def:locally semi-analytic}. Then each point $x\in X$ has a basis of neighborhoods $\{U_i\}_{i\in I}$, $x\in U_i$ such that the embedding 
$$
\pi^{-1}(x) \ra \pi^{-1}(U_i)
$$
is a homotopy equivalence.	
\end{prop}

\begin{proof}First of all, question is local, so we can assume that $\pi$ is the restriction of some analytic map from $\pi\colon \mbb R^m\ra \mbb R^n$ to some relatively compact (and automatically locally closed) semi-analytic subsets. We then can assume that both $X$ and $Y$ are Hausdorff, paracompact and locally weakly contractible, since bounded closed subanalytic subsets of $\mbb R^n$ admit a triangulation (\cite{Hardt_triangulations}).
	 Moreover, the square of the distance function $d(x,-)^2\colon \mbb R^n \ra \mbb R_{\ge 0}\subset \bR$ defines an analytic map $\delta\colon X\ra  \mbb R_{\ge 0}$  with the property that the preimage $\delta^{-1}(0)$ is given exactly by the point $\{x\}\subset X$, and subsets $U_\epsilon\coloneqq \delta^{-1}([0,\epsilon))\subset X$ form a basis of neighborhoods of $x$ in $X$. By replacing $X$ with a suitable compact neighborhood\footnote{For example, $\delta^{-1}([0,t])$ for sufficiently small $t >0$.} of $x$, we can assume that the composition of $\pi'\coloneqq \delta\circ \pi$ is proper, and it would be  enough to show that the map 
	$$
	j\colon \pi'^{-1}(0) \ra \pi'^{-1}([0,\epsilon))
	$$
	is a homotopy equivalence for small enough $\epsilon >0$. For this, by Whitehead's theorem it is enough to show that $j^*$ induces an isomorphism on {\v C}ech cohomology $H^1(-,\ul G)$ for any (not necessarily commutative) group $G$ functorially in $G$ (this implies that $j$ induces on isomorphism on $\pi_1(-)$) and, then, that $j^*$ is an isomorphism on $H^i(-,\mc L)$ for any local system $\mc L$. Here we implicitly used that $\pi'^{-1}([0,\epsilon))$ is locally weakly contractible to relate {\v C}ech to singular cohomology.
	
	The main tool that we will use is Hardt's theorem (see \cite[Th\'eor\`eme de Hardt, p.171]{Teissier}, which we apply to the graph of $\pi'$), which (under our assumptions on $\pi$) guarantees that for small enough $\epsilon$ one can identify the map $\pi'$ over the open interval $(0,\epsilon)$ with the projection $(0,\epsilon)\times S\ra (0,\epsilon)$ for some subanalytic set $S$.  Now, for the first step, consider sheaves $R^0\pi'_*\ul G$ and $R^1\pi'_*\ul G$ (the latter is the sheaf of pointed sets on $[0,\epsilon)$). Restrictions of both sheaves to $(0,\epsilon)$ are constant, and so their global sections on $[0,\epsilon)$ are given by the fibers at $\{0\}\inj [0,\epsilon)$. By proper base change the latter can be identified with $H^0( \pi'^{-1}(0), \ul G)$ and $H^1(\pi'^{-1}(0), \ul G)$ correspondingly. One also sees that $H^1(R^0\pi'_*\ul G)$ is a single point, and so global sections $\Gamma([0,\epsilon), R^1\pi'_*\ul G)$ compute $H^1(\pi'^{-1}([0,\epsilon), \ul G)$. We get that 
	$$j^*\colon H^1(\pi'^{-1}([0,\epsilon)), \ul G)\ra H^1(\pi'^{-1}(0), \ul G)$$
	is an isomorphism for any $G$.
	
Then, having a local system $\mc L$ on $ \pi'^{-1}([0,\epsilon))$ it remains to show that the restriction map 
$$
j^*\colon H^i(\pi'^{-1}([0,\epsilon)), \mc L) \ra  H^i(\pi'^{-1}(0), j^*\mc L)
$$
is an isomorphism.
By induction on $i$ and similar argument as above one can see that sheaves $R^i\pi'_*\mc L$ do not have higher cohomology (since their global sections are given by the restriction to $\{0\}$), and that $\Gamma([0,\epsilon), R^i\pi'_* \mc L)$ is naturally equivalent to $H^i(\pi'^{-1}([0,\epsilon)), \mc L)$; consequently the pull-back $j^*$ above is an isomorphism for any $i$.
\end{proof}		

\begin{rem}
Let us note that \Cref{prop:local homotopy triviality of subanalytic maps} in fact holds in a bigger generality of \textit{locally definable maps}: namely $\pi$ should satisfy the property in \Cref{def:locally semi-analytic} with \textit{semi-analytic} replaced by \textit{definable} for any given $o$-minimal model structure. Indeed, the only facts that we used are the existence of triangulation and Hardt's theorem, that also have analogues in this generality (see e.g. \cite[Theorems 4.4 and 5.22]{Coste_o_minimal}).
\end{rem}	
Let us also record the following straightforward observation. 
\begin{lem}\label{lem: Xlog is real analytic}
	Let $(X,\mc M)$ be an fs log complex analytic space. Then the map $\pi\colon X^{\log} \ra X$ is locally semi-analytic in the sense of \Cref{def:locally semi-analytic}. Consequently, \Cref{cor:toric map} applies to $\pi$.
\end{lem}

\begin{proof}
For a fine and saturated monoid $P$ let $\mbb A(P)$ denote the analytification of the affine toric variety $\Spec \mbb C[P]$; it has a natural log-structure given by $P\cdot \mc O^\times \subset \mc O$. Since the log structure on $X$ is fs, by \cite[Corollary II.2.3.6]{Ogus_Book} for any point $x\in X$ there exists an open neighborhood $x\in U\subset X$ and an fs monoid $P$ with the complex analytic map $f\colon U\ra \mbb A(P)$ such that $\mc M$ is obtained as a pull-back of the log-structure on $\mbb A(P)$. Moreover, one also has a homeomorphism
$$
U^{\log}\simeq \mbb A(P)^{\log}\times_{\mbb A(P)} U,
$$
and this way it will be enough to endow $\mbb A(P)^{\log}\ra {\mbb A(P)}$ with the real analytic structure. By \cite[Proposition V.1.2.7]{Ogus_Book} this map is identified with $\Map_{\mr{Mon}}(P,\mbb R_{\ge 0}\times S^1) \ra \Map_{\mr{Mon}}(P,\mbb C)$, where $\mbb C$ is considered as a monoid by multiplication, the sets of maps are endowed with the topology induced by the targets, and the map itself is induced by the homomorphism $\mbb R_{\ge 0}\times S^1\ra \mbb C$ sending $(r,z)$ to $r\cdot z$.  Let $p_1,\ldots,p_n$ be some generators of $P$; then we have a commutative diagram 
$$
\xymatrix{\Map_{\mr{Mon}}(P,\mbb R_{\ge 0}\times S^1)\ar[r]\ar[d]& \Map_{\mr{Mon}}(P,\mbb C)\ar[d] \\
	(\mbb R_{\ge 0}\times S^1)^n \ar[r]& \mbb C^n
}
$$
with vertical arrows induced by evaluations at $p_i$'s. Note that each map $\mbb R_{\ge 0}\times S^1 \ra \mbb C$ naturally embeds into $\mbb R \times \mbb C \ra \mbb C$, $(t, z)\mapsto t\cdot z$, where the map is clearly real-analytic. It is also easy to see that the preimage of any relatively compact $U\subset \mbb C$ in $\mbb R_{\ge 0}\times S^1$ is also relatively compact in $\mbb R \times \mbb C$. Moreover the spaces of maps above define semi-analytic subspaces of $(\mbb R \times \mbb C)^n$ and $\mbb C^n$ correspondingly (they are cut out by finitely many equations induced by relations between $p_i$'s). Thus the restriction of the map $\Map_{\mr{Mon}}(P,\mbb R_{\ge 0}\times S^1)\ra \Map_{\mr{Mon}}(P,\mbb C)$ to the intersection with any relatively compact $U\subset \mbb C^n$ is obtained by restriction to some relatively compact semi-analytic subspaces in $(\mbb R \times \mbb C)^n \ra \mbb C^n$, which is exactly what we wanted to show.
\end{proof}		

\begin{rem}\label{rem:reconstructing KN from exp}
	Let $X$ satisfy \Cref{ass:finite covering dimension}. The functor 
	$$
	\Map_{\DBinAlg}(-,\mbb Z)\colon \DBinAlg^\op \ra \Spc
	$$ 
	extends to a functor
	$$[-]_X\coloneqq \Map_{\DBinAlg}(-,\mbb Z)\colon \DBinAlg(X)^\op \ra \hcoShv(X,\Spc).$$
	There is also a natural functor 
	$$
	C^*_\sing(-/X, \mbb Z)\colon \hcoShv(X,\Spc)^\op \ra  \DBinAlg(X)
	$$
	with $C^*_\sing(\mc F/X, \mbb Z)$ sending $U\in \Op(X)$ to $C^*_\sing(\mc F(U),\mbb Z)$. From \Cref{cor:cohomology defines fully faithful embedding} and descent one can see that $C^*_\sing(-/X, \mbb Z)$ induces a fully faithful embedding when restricted to the subcategory of $\hcoShv(X,\Spc)^\op$ spanned by cosheaves whose values locally are connected nilpotent spaces of finite type. In particular, for such cosheaves $\mc F$ the natural map 
	$$
	\mc F\ra [C^*_\sing(\mc F/X, \mbb Z)]_X
	$$
	is an equivalence. Applying this to $\mc F=[X^{\log},\pi]$ represented by $\pi\colon X^{\log}\ra X$ we see from \Cref{thm:cohomology of Kato-Nakayama space} that one can reconstruct $[X^{\log},\pi]$ from $\eexp(\alpha)$ as
	$$
	[X^{\log},\pi] \xra{\sim} [\mbf L\mc Bin_X^{\mr{coaug}}(\eexp(\alpha))]_X.
	$$
\end{rem}	
\subsection{Algebraic approximation $X^{\log}_{\mbb H}$ to $[X^{\log},\pi]$}\label{ssec:geometrization}

Let us finish with our construction of an "algebraic approximation" $X^{\log}_{\mbb H}$ of $X^{\log}$. The construction in fact applies to any $\pi\colon Y\ra X$ as in \Cref{cor:toric map}, but for brevity let us just restrict to the case of Kato-Nakayama space.

\begin{construction}[A cosheaf of stacks $X^{\log}_{\mbb H}$] Let $\mc F\in \hcoShv(X,{\mscr D}(\mbb Z)^{\le 0})$ be a ${\mscr D}(\mbb Z)^{\le 0}$-valued cosheaf. Recall the colimit preserving functor $-\otimes_{\mbb Z}\mbb H \colon {\mscr D}(\mbb Z)^{\le 0} \ra \Ab(\Stk_{/\mbb Z})$ (\Cref{rem:tensoring with a group scheme}). It induces a functor 
	$$
	-\otimes_{\mbb Z}\mbb H \colon \hcoShv(X,{\mscr D}(\mbb Z)^{\le 0}) \ra \hcoShv(X,\Ab(\Stk_{/\mbb Z})).
	$$
	Recall also the map $i\colon \ul{\mbb Z}\ra \mbb H$ in $\Ab(\Stk_{/\mbb Z})$ (\Cref{constr:map from Z to H}). For any $\pi\colon Y\ra X$ it induces a commutative diagram
	\begin{equation}\label{eq:base change to H}
	\xymatrix{C_{\le 1}^\sing(\pi,\mbb Z) \ar[r]^(.43){\sim}\ar[d]&C_{\le 1}^\sing(\pi,\mbb Z)\otimes_{\mbb Z} \ul\bZ \ar[r]^{\id\otimes i}\ar[d]& C_{\le 1}^\sing(\pi,\mbb Z)\otimes_{\mbb Z} \mbb H\ar[d]\\
	C_{\le 0}^\sing(\pi,\mbb Z) \ar[r]^(.43){\sim}& C_{\le 0}^\sing(\pi,\mbb Z)\otimes_{\mbb Z}\ul \bZ \ar[r]^{\id\otimes i}& C_{\le 0}^\sing(\pi,\mbb Z)\otimes_{\mbb Z} \mbb H}
	\end{equation}
	(here we view $C_{\le i}^\sing(\pi,\mbb Z)$ as a cosheaf of algebraic stacks via the functor $\iota$ from \Cref{not:cosheaves}).
	Recall that for $\pi\colon X^{\log} \ra X$ by \Cref{cor:toric map} the cosheaf $[X^{\log},\pi]$ is equivalent to the fiber product 
	$$
	\xymatrix{[X^{\log},\pi] \ar[d]\ar[r]& K(C_{\le 1}^\sing(\pi,\mbb Z))\ar[d] \\
	[X,\id]\times \{1\}\ar[r]& [X,\id]\times \bZ}	
	$$
	as in \Cref{eq:commutative square for any cosheaf of spaces}, with right vertical arrow identified with $K$ applied to $C_{\le 1}^\sing(\pi,\mbb Z)\ra C_{\le 0}^\sing(\pi,\mbb Z)$. By composing with the outer square of \Cref{eq:base change to H} we get a commutative diagram 
		\begin{equation}\label{eq:diagram of cosheaves}
	\xymatrix{[X^{\log},\pi] \ar[d]\ar[r]& K(C_{\le 1}^\sing(\pi,\mbb Z))\ar[d] \ar[r]& K(C_{\le 1}^\sing(\pi,\mbb Z)\otimes_{\mbb Z} \mbb H)\ar[d] \\
		[X,\id]\times \{1\}\ar[r]& [X,\id]\times \bZ \ar[r]& [X,\id]\times \mbb H; }
	\end{equation}
	here we identified $K(C_{\le 0}^\sing(\pi,\mbb Z)\otimes_{\mbb Z} \mbb H)$ with $[X,\id]\times \mbb H$ (and we did not distinguish between $\mbb H$ and $K(\mbb H)$).
We then finally define the cosheaf of stacks $X^{\log}_{\mbb H}$ as the fiber product
$$
\xymatrix{X^{\log}_{\mbb H} \ar[r]\ar[d]& K(C_{\le 1}^\sing(\pi,\mbb Z)\otimes_{\mbb Z} \mbb H)\ar[d]\\
	[X,\id]\times \{1\} \ar[r]^{\id\times i}& [X,\id]\times \mbb H.}
$$
  Outer commutative square in \Cref{eq:diagram of cosheaves} induces a natural map 
	$$
	i_{X}\colon [X^{\log},\pi] \ra X_{\mbb H}^{\log}
	$$
	of cosheaves of algebraic stacks on $X$.
\end{construction}	

\begin{rem}
	Similarly to \Cref{rem:gerbe structure on Y} one can view $X_{\mbb H}^{\log}$ as a gerbe, but over the cosheaf of group schemes given by $\mc L_\pi^\vee\otimes_{{\mbb Z}} \mbb H$; more precisely, by construction, it is induced from the $\mc L_\pi^\vee$-gerbe given by $X^{\log}$ via the map $\mc L_\pi^\vee \ra \mc L_\pi^\vee\otimes_{{\mbb Z}} \mbb H$ induced by $\ul{\mbb Z}\ra \mbb H$. Note also that $\mc L_\pi$ in this case is identified with $\ol{\mc M^{\gr}}$.
\end{rem}	

Finally, we would like to reformulate \Cref{thm:cohomology of Kato-Nakayama space} in terms of the map $i_X$. 

\begin{construction}[$\mc O$-cohomology of a cosheaf of stacks] Recall the functor
	$$
	\mscr O_{\DAlg}\colon \Stk^\op_{/\mbb Z} \ra \DAlg(\mbb Z) 
	$$
	from \Cref{constr:O-cohomology}. It is limit commuting, and thus induces a functor 
	$$
	\mscr O_{\DAlg}\colon \hcoShv(X,  \Stk_{/\mbb Z})^\op \ra \hShv(X,{\mscr D}(\mbb Z)).
	$$
	
\end{construction}	

\begin{ex}\label{ex:cohomology of structure sheaf for the case of space}
	Let $\pi\colon Y\ra X$ be a continuous map and assume that $Y$ is locally weakly contractible. Let $\iota([Y,\pi])\in \hcoShv(X,  \Stk_{/\mbb Z})$ be the cosheaf of stacks represented by $Y$. Then 
	$$
	\mscr O_{\DAlg}(\iota([Y,\pi]))\simeq R\pi_{\DAlg,*} \ul{\mbb Z},
	$$
	where pushforward is taken in the category of $\DAlg(\mbb Z)$-valued hypersheaves. Indeed, the sections on a given open $\mscr O_{\DAlg}(\iota([Y,\pi]))(U)\in \DAlg(\mbb Z)$ are naturally identified with $C^*_\sing(\pi^{-1}(U),\mbb Z)$, while sections of $R\pi_{\DAlg,*} \ul{\mbb Z}$ are naturally identified with the {\v C}ech cohomology $\RG(\pi^{-1}(U),\ul{\mbb Z})$. One has a functorial in $U$ map $\RG(\pi^{-1}(U),\ul{\mbb Z})\ra C^*_\sing(\pi^{-1}(U),\mbb Z)$ which is an equivalence by \Cref{lem:sing vs Cech}.
\end{ex}

\begin{ex} Let us consider the cosheafification of the constant copresheaf on $X$ given by $K(\mbb H)$: it is naturally equivalent to $[X,\id]\times \mbb H$. Assume $X$ is locally weakly contractible. Then $\mscr O_{\DAlg}(X\times \mbb H)$ is described as the sheafification of the constant presheaf given by $\Bin(\mbb Z)$, or, in other words, the sheaf of derived commutative algebras that underlies $\mbf L\mc Bin_X(\ul \bZ)$.
\end{ex}		

\begin{lem}\label{lem:main iso through cosheaves} Let $(X,\mc M)$ be an fs log complex analytic space. Then the map 
	$$
	i_X\colon [X^{\log},\pi] \ra X^{\log}_{\mbb H}
	$$
	of cosheaves of stacks on $X$ induces an equivalence
	$$
	i_X^*\colon \mscr O_{\DAlg}(X^{\log}_{\mbb H}) \ra \mscr O_{\DAlg}([X^{\log},\pi]),
	$$
	of $\DAlg(\mbb Z)$-valued sheaves on $X$, which is explicitly given by
	$$
	\mbf L\mc Bin_X^{\mr{coaug}}(R^{\le 1}\pi_*\ul \bZ)\xra{\sim} R\pi_{\DAlg,*} \ul{\mbb Z}.
	$$
\end{lem}
\begin{proof}
	First of all, by \Cref{ex:cohomology of structure sheaf for the case of space}  we can identify $ \mscr O_{\DAlg}(X^{\log})$ and  $R\pi_{\DAlg,*} \ul{\mbb Z}$. Moreover, via \Cref{rem:LBin as cohomology of a stack} $\mscr O_{\DAlg}(K(C_{\le 1}^\sing(\pi,\mbb Z)\otimes_{\mbb Z} \mbb H))$ can be naturally identified with (the underlying sheaf of derived commutative algebras of) $\mbf L\mc Bin_X(R^{\le 1}\pi_*\ul \bZ)$, while $\mscr O_{\DAlg}(X\times \mbb H)\simeq \mbf L\mc Bin_X(\ul \bZ)$. Moreover, $\mscr O_{\DAlg}(X\times \{1\})\simeq \ul \bZ$ and the map $\mscr O_{\DAlg}(X\times \mbb H) \ra \mscr O_{\DAlg}(X\times \{1\})$  is given exactly by the "evaluation at 1" $\ev_1\colon  \mbf L\mc Bin_X(\ul \bZ) \ra \ul \bZ$. This way we obtain the map
	$$
	\mbf L\mc Bin_X^{\mr{coaug}}(R^{\le 1}\pi_*\ul \bZ)\coloneqq \mbf L\mc Bin_X(R^{\le 1}\pi_*\ul \bZ) \otimes_{\mbf L\mc Bin_X(\ul \bZ)} \ul \bZ \tto R\pi_{\DAlg,*} \ul{\mbb Z}
	$$
	via the map of cosheaves
	$$
	[X^{\log},\pi] \ra X^{\log}_{\mbb H} \coloneqq K(C_{\le 1}^\sing(\pi,\mbb Z)\otimes_{\mbb Z} \mbb H)\times_{(X\times \mbb H)} (X\times\{1\}).
	$$
	We know that this map is an equivalence by \Cref{thm:cohomology of Kato-Nakayama space}. It thus remains to show that the natural map $\mbf L\mc Bin_X^{\mr{coaug}}(R^{\le 1}\pi_*\ul \bZ) \ra \mscr O_{\DAlg}(X^{\log}_{\mbb H})$ is also an equivalence. By \Cref{lem: Xlog is real analytic} there is a base of opens $\mc B$, where for any $U\in\mc B$ the preimage $\pi^{-1}(U)$ is homotopy equivalent to a torus. On such $U$ the map $\mbf L\mc Bin_X^{\mr{coaug}}(R^{\le 1}\pi_*\ul \bZ(U)) \ra \mscr O_{\DAlg}(X^{\log}_{\mbb H}(U))$ is an equivalence by \Cref{prop: description of free coconnective binomial algebras}; it then follows that the map is an equivalence in general by passing to hypersheafification. 
\end{proof}	

\begin{rem}
	In fact, via a version\footnote{To our knowledge this is a part of work in progress by Mathew-Mondal; they show that the functor $\DAlg(R)^\op\ra \Stk_{\!/\!R}$ given by sending $A\mapsto \Map_{\DAlg(R)}(-,A)$ gives a fully faithful embedding of a certain natural subcategory of commutative derived rings.} of the theory of affine stacks of To\"en \cite{Toen_affine_stacks} one can view $X^{\log}_{\mbb H}$ as the "affinization" of $X^{\log}$.
\end{rem}

\addcontentsline{toc}{section}{References}
\printbibliography

\bigskip

\noindent Dmitry~Kubrak, {\sc CNRS, Laboratoire de Mathématiques d'Orsay, Université Paris-Saclay, Bâtiment 307, F-91405 Orsay, France},
\href{mailto:dmkubrak@gmail.com}{dmkubrak@gmail.com}

\smallskip

\noindent 
Georgii~Shuklin, {\sc School of Mathematics and Statistics, University of Sheffield, Hicks Building, Hounsfield Road, Sheffield S3 7RH, United Kingdom }, \href{mailto:shuklin.ium@yandex.ru}{shuklin.ium@yandex.ru}

\smallskip

\noindent 
Alexander~Zakharov, {\sc Faculty of Mathematics, Higher School of Economics, Usacheva ulitsa 6, Moscow 119048, Russia; Chennai Mathematical Institute, H1, SIPCOT IT Park, Siruseri Kelambakkam 603103, India}, \href{mailto:xaxa3217@gmail.com}{xaxa3217@gmail.com}

\end{document}